\documentclass[12pt]{amsart}
\usepackage{mathrsfs}
\usepackage{}

\usepackage{amsmath}
\usepackage{amssymb}
\usepackage{amsfonts}
\usepackage{amsthm}
\usepackage{enumerate}
\usepackage{hyperref}
\usepackage{color}
\usepackage{psfrag}
\usepackage[all]{xy}
\usepackage{tikz}
\usepackage{epstopdf}
\usepackage{graphicx}
\usepackage{subcaption}
\usepackage{skak}

\theoremstyle{plain}
\newtheorem{thm}{Theorem}[section]

\newtheorem{prop}[thm]{Proposition}
\newtheorem{lem}[thm]{Lemma}
\newtheorem{remark}[thm]{Remark}
\theoremstyle{definition}
\newtheorem{dfn}[thm]{Definition}
\newtheorem{ob}[thm]{Observation}
\newtheorem{ex}[thm]{Example}
\newtheorem{conj}[thm]{Conjecture}
\newtheorem{cor}[thm]{Corollary}
\newtheorem{setup}[thm]{Setup}

\newtheorem*{main}{Main Theorem}

\newcommand{\al}{\alpha}
\newcommand{\seq}{\subseteq}

\begin{document}


\title[A class of polyocollection ideals with quadratic Gr\"{o}bner bases]{A class of polyocollection ideals with quadratic Gr\"{o}bner bases}


\author[Wang]{H. Wang}

\address{H. Wang, School of mathematics and statistics,  Suzhou university of technology, Changshu 215500, China}

\email{wanghong@szut.edu.cn}

\author[Guo]{J. Guo$^{*}$}

\address{J. Guo, School of Mathematics and Statistics, Hainan University, Haikou 570228, China}

\email{guojinecho@163.com}

\let\thefootnote\relax\footnotetext{*Corresponding Author: J. Guo -- {\tt guojinecho@163.com}}


\begin{abstract}
In 2024, Cisto et al. introduced polyocollections and polyocollection ideals, which generalize the notions of collections of cells and inner $2$-minors ideals, respectively. In this paper, we naturally extend the notions and results about zig-zag walk, zero-sum condition, rook number, and switching polynomial from collections of cells to polyocollections, and we define a class of polyocollections, called natural polyocollections, whose structure is similar to that of collections of cells.
We focus on natural polyocollections whose polyocollection ideals admit quadratic Gr\"obner bases with respect to lexicographic order induced by some specific orders on vertices. Using the zero-sum condition, we prove the primality of this class of polyocollections, which shows that the zig-zag conjecture proposed by Mascia et al. holds true for collections of cells satisfying the aforementioned properties. Moreover, we show that $h$-polynomials and regularities of the coordinate rings of this class of polyocollections are equal to their switching rook polynomials and rook numbers, respectively. These results give positive answers to the switching rook polynomial conjecture proposed by Jahangir and Navarra, and the rook number open problem proposed by Rinaldo and Romeo for collections of cells satisfying the aforementioned properties.
\end{abstract}


\let\thefootnote\relax





\thanks{This research was supported by the National Natural Science Foundation of China (Grant No. 12361003), the Natural Science Research of the Jiangsu Higher Education Institutions of China(Grant No. 25KJB110015), and the Hainan Provincial Natural Science Foundation of China (Grant No. 126MS0005).}


\maketitle


\section{Introduction} \label{sec1}

Polyominoes are two-dimensional objects obtained by joining unit squares edge-to-edge, which are introduced by Solomon Golomb in 1953. Originating in recreational mathematics and tiling problems. As investigations deepened, mathematicians needed to handle more general structures. This leads to the notion of a collection of cells, which is a finite set of unit squares in $\mathbb{Z}^2$. In 2012, Qureshi, in \cite{2012q}, introduced the  ideal $I_{\mathcal{P}}$ of inner $2$-minors  of a collection $\mathcal{P}$ of cells as the ideal in $S_{\mathcal{P}} = \mathbb{K}[x_{\mathbf{v}} \mid \mathbf{v} \in V(\mathcal{P})]$ generated by the inner $2$-minors corresponding to its inner intervals, where $ V(\mathcal{P})$ denotes the vertex set of  $\mathcal{P}$. The quotient ring $\mathbb{K}[{\mathcal{P}}]:=S_{\mathcal{P}}/I_{\mathcal{P}}$ is called the {\it coordinate ring} of  $\mathcal{P}$. This ideal is sometimes referred to as the inner $2$-minors ideal. If $\mathcal{P}$ is a polyomino, $I_{\mathcal{P}}$ is called a polyomino ideal. Since then, algebraic and homological properties of inner $2$-minors ideals have attracted the interest of many researchers. Primality, Gr\"{o}bner bases, Cohen--Macaulayness, Hilbert series, regularity of these ideals have been widely studied in \cite{2025cjn, 2023cn, 2022cnu, 2023cnv, 2023dn, 2021ehq, 2015hq, 2014hs, 2015hq1, 2024jn, 2020mrr, 2022mrr, 2025nrr, 2012q, 2022qrr, 2017q, 2021rr,2018s, ZGW, ZGW2}.

Among these properties,
primality is one of the most fundamental and well-studied. However, a complete characterization about this property remains a challenging problem. Usually, we say that a collection of cells is prime if its inner 2-minors ideal is prime. In \cite{2012q}, it is proved that the collections of cells whose connected components are row or column convex are prime. In \cite{2015hq}, a special class of prime polyominoes,  called balanced polyominoes, is introduced. Moreover, it is proved in \cite{2014hs} that a polyomino is balanced if and only if it is simple, thereby establishing that simple polyominoes are prime. Later, in \cite{2017q}, the primality of simple polyominoes is also verified by showing that their coordinate rings are isomorphic to the edge rings of weakly chordal bipartite graphs. In \cite{2022cnu}, the method of \cite{2017q} is extended to  more general case of simple collections of cells, proving that such collections are also prime. The focus has now turned to nonsimple collections of cells, which are much more complicated. It is proved in \cite{2015hq1} that the polyomino obtained by removing a convex polyomino from a rectangular polyomino is prime, which demonstrates that certain nonsimple polyominoes can still retain primality. A deeper investigation  of primality in nonsimple polyominoes leads to the development of the zig-zag walk, which is introduced by Mascia et al.. It is known that a prime polyomino contains no zig-zag walk. The following converse is conjectured.\\[-0.4cm]
\begin{conj}{ (\cite[Conjecture 4.6]{2020mrr})}\label{conj2}
Let $\mathcal{P}$ be a polyomino. Then $\mathcal{P}$ is prime if and only if  $\mathcal{P}$ contains no zig-zag walk.
\end{conj}

To date, all known results concerning primality on all kinds of polyominoes are consistent with this conjecture, such as  grid polyominoes (\cite{2020mrr}), closed path polyominoes (\cite{2023cn}), weakly closed path polyominoes (\cite{2022cnu}), certain polyominoes whose polyomino ideals have quadratic Gr\"obner bases (\cite{2022mrr}),  SR-polyominoes (\cite{ZGW}) and extended bipartite collections of cells (\cite{ZGW2}).

A binomial ideal is prime if and only if it is a toric ideal. Hence, a standard way to verify the primality of a collection of cells is to prove that its inner $2$-minors ideal coincides with a toric ideal. For example,  primality is established in \cite{2023cn,2022cnu,2020mrr} by showing that the ideals under discussion equal the toric ideals corresponding to some toric rings. In \cite{2015hq}, a lattice $\Lambda_{\mathcal{P}}$ associated with a polyomino $\mathcal{P}$ is defined as the linear space over $\mathbb{Z}$ spanned by the vectors corresponding to the cells of $\mathcal{P}$. It is proved that $\Lambda_{\mathcal{P}}$ is saturated, and consequently, the lattice ideal $I_{\Lambda_{\mathcal{P}}}$ is a toric ideal. These definitions and conclusions extend naturally to collections of cells. In \cite{2022mrr,ZGW,ZGW2}, the primality of the ideal $I_{\mathcal{P}}$ is obtained by proving the equality with the corresponding lattice ideal $I_{\Lambda_{\mathcal{P}}}$. It is worth noting that, in \cite{ZGW}, the zero-sum condition is introduced to characterize the elements in $I_{\Lambda_{\mathcal{P}}}$, and consequently it plays an important role in verifying the equality between $I_{\mathcal{P}}$ and $I_{\Lambda_{\mathcal{P}}}$. For instance, this condition is used in \cite{ZGW2}  to establish  primality for SR-polyominoes and extended bipartite collections of cells.

Besides the primality, another hot topic concerns characterizing the Hilbert series and regularity (see \cite{2025nq}) of the coordinate ring of a collection of cells in terms of its rook polynomial or switching rook polynomial. For a collection $\mathcal{P}$ of cells, the {\it rook polynomial} of $\mathcal{P}$ is \\[-0.2cm]
$$r_{\mathcal {P}}(t)=\sum_{j=0}^{r(\mathcal {P})}r_j(\mathcal {P})t^j,$$
where  $r(\mathcal {P})$ is the rook number (the maximum number of  non-attacking rooks that can be placed on $\mathcal {P}$), $r_0(\mathcal {P})=1$, and $r_j(\mathcal {P})$ denotes the number of $j$-rook configurations (placements of $j$ non-attacking rooks) for all $1\leq j\leq r(\mathcal {P})$. See \cite{2021rr} for more details. In \cite{2021ehq}, it is shown that the coordinate ring of an $L$-convex polyomino $\mathcal{P}$ is isomorphic to the edge ring of a Ferrers graph, and consequently its regularity equals $r(\mathcal {P})$. This result sparks further interest in rook polynomials.  In \cite{2021rr}, it is proved that the $h$-polynomial and  regularity of coordinate ring for a simple thin polyomino $\mathcal{P}$ are the rook polynomial $r_{\mathcal {P}}(t)$ and rook number $r(\mathcal {P})$, respectively. Moreover, it is conjectured that for a polyomino $\mathcal{P}$, the $h$-polynomial of $\mathbb{K}[{\mathcal{P}}]$ equals $r_{\mathcal {P}}(t)$ if and only if $\mathcal{P}$ is thin, while an open question is proposed: whether the regularity of  $\mathbb{K}[{\mathcal{P}}]$ always equals $r(\mathcal {P})$ for any polyomino. For some other classes of thin polyominoes, such as closed path polyominoes and grid polyominoes, the conjecture holds  and the question has a positive answer, refer to \cite{2025cjn} and \cite{2023dn}. As many authors investigate the Hilbert series and regularity of coordinate rings for non-thin collections, the switching rook polynomial has emerged. In  \cite{2022qrr}, an equivalence relation $\overset{\scriptscriptstyle\mathcal{P}}{\thicksim}$ is introduced on  the set of rook configurations. The {\it  switching rook polynomial} is \\[-0.2cm]
$$\widetilde{r}_{\mathcal {P}}(t):=\sum_{j=0}^{r(\mathcal {P})}\widetilde{{r}}_j(\mathcal {P})t^j,$$
where $\widetilde{{r}}_0(\mathcal {P})=1$ and $\widetilde{{r}}_j(\mathcal {P})$ is the number of equivalence classes of $j$-rook configurations of $\mathcal {P}$ under the equivalence relation $\overset{\scriptscriptstyle\mathcal{P}}{\thicksim}$ for all $1\leq j\leq r(\mathcal {P})$.
More details can be  found in \cite{2022qrr}. Observe that  $r_{\mathcal {P}}(t)=\widetilde{r}_{\mathcal {P}}(t)$ exactly when $\mathcal {P}$ is thin. Based on  \cite[Conjecture 4.9]{2024jn} and \cite[Question 4.6]{2021rr}, the following  conjecture can be  formulated.

\begin{conj}\label{conj1}
Let $\mathcal{P}$ be a collection of cells. Then  $h$-polynomial  of  $\mathbb{K}[{\mathcal{P}}]$ is equal to $\widetilde{r}_{\mathcal {P}}(t)$ and its regularity is equal to $r(\mathcal {P})$.
\end{conj}

This conjecture has been verified for parallelogram polyominoes (\cite{2022qrr}), frame polyominoes (\cite{2024jn}) and  collections of cells whose inner $2$-minors ideals have quadratic Gr\"obner bases, and each of whose weakly connected components is convex (\cite{2025nrr}).

In 2024, Cisto et al., in \cite{2024cnv}, generalized the concept of collections of cells and introduced polyocollections, whose structures are similar to that of collections of cells.
They also defined polyocollection ideals by analogy with the inner $2$-minors ideals. See section \ref{sec2} for more details. Let $\mathcal{C}$ be a polyocollection and $I_{\mathcal{C}}$ its polyocollection ideal, and let $V(\mathcal {C})$ be its vertex set. For any $\mathbf{a},\mathbf{b}\in V(\mathcal {C})$ with $\mathbf{a}=(i,j)$ and $\mathbf{b}=(k,\ell)$, define $x_{\mathbf{a}}<x_{\mathbf{b}}$ if and only if  $i<k$, or $i=k$ and $j<\ell$. Let $<_{\text{lex}}$ be the lexicographical order induced by this order of the variables. In \cite[Theorem 4.1]{2012q},  for the case that $\mathcal{C}$ is a collection of cells, Qureshi gave an equivalent combinatorial characterization for $I_{\mathcal{C}}$ having a reduced quadratic Gr\"obner basis with respect to the monomial order $<_{\text{lex}}$.
Note that polyocollections have similar structures to collections of cells, especially inner intervals on polyocollections share similar intrinsic meaning with inner intervals on collections of cells (see section \ref{sec2} for details). The following theorem can be derived naturally.

\begin{thm} \label{condition}
Let $\mathcal{C}$ be a polyocollection. The set of inner $2$-minors of $\mathcal {C}$ forms a reduced quadratic Gr\"obner basis with respect to $<_{\text{\em lex}}$  if and only if $\mathcal {C}$ satisfies the following condition$:$ \\[-0.4cm]
\begin{itemize}
\item[$(\dag)$] For any two inner intervals $[\mathbf{a}, \mathbf{b}]$ and $[\mathbf{b}, \mathbf{c}]$ of $\mathcal {C}$, at least one of $[\mathbf{e},\mathbf{c}]$ or $[\mathbf{d},\mathbf{c}]$ is an inner interval of $\mathcal {C}$, where $\mathbf{d}$ and $\mathbf{e}$ are the anti-diagonal corners of $[\mathbf{a}, \mathbf{b}]$, {\em see Figure \ref{fig1} (a)}.  \\[-0.4cm]
\end{itemize}
\end{thm}

Similarly, for the order defined by  $x_{\mathbf{a}}>x_{\mathbf{b}}$ if and only if  $i<k$, or $i=k$ and $j>\ell$, where $\mathbf{a}=(i,j)$ and $\mathbf{b}=(k,\ell)$, let $<'_{\text{lex}}$ be the lexicographical order induced by this order of the variables. Then one can easily verify that the set of inner $2$-minors of $\mathcal {C}$ forms a reduced quadratic Gr\"obner basis with respect to $<'_{\text{lex}}$  if and only if $\mathcal {C}$ satisfies the following condition$:$
\begin{itemize}
{\em \item[$(\dag\dag)$] For any two inner intervals $[\mathbf{a}, \mathbf{b}]$ and $[\mathbf{c},\mathbf{d}]$ of $\mathcal {C}$ whose anti-diagonals are $\mathbf{e}$, $\mathbf{f}$ and $\mathbf{e},\mathbf{g}$, respectively, at least one of $[\mathbf{h},\mathbf{e}]$ or $[\mathbf{e},\mathbf{i}]$ is an inner interval of $\mathcal {C}$}, see Figure \ref{fig1} (b).\\[-0.4cm]
\end{itemize}

\begin{figure}[h]
\centering
    \begin{tikzpicture}
    \tikzstyle{every node}=[font=\small,scale=0.9]
    \coordinate[](a1)at(0,0){};\coordinate[](a2)at(2,0){};\coordinate[](a3)at(4,0){};
    \coordinate[](b1)at(0,1){};\coordinate[](b2)at(2,1){};\coordinate[](b3)at(4,1){};
    \coordinate[](c1)at(0,2){};\coordinate[](c2)at(2,2){}; \coordinate[](c3)at(4,2){};
    \draw[fill=black!16,line width=0.8pt]
    (a1)--(a2)--(b2)--(b1)--(a1)--cycle;
    \draw[fill=black!16,line width=0.8pt]
    (b2)--(b3)--(c3)--(c2)--(b2)--cycle;
    \draw[densely dashed, line width=0.8pt] (b1)--(c1)--(c2);
    \draw[densely dashed, line width=0.8pt] (a2)--(a3)--(b3);
 \coordinate[label=225:$\mathbf{a}$]()at(a1){};
 \coordinate[label=225:$\mathbf{b}$]()at(b2){};
 \coordinate[label=225:$\mathbf{c}$]()at(c3){};
 \coordinate[label=225:$\mathbf{d}$]()at(a2){};
 \coordinate[label=225:$\mathbf{e}$]()at(b1){};

   \node[below,scale=1] at (2,-0.5) {(a)\hspace{0.15cm} condition $(\dag)$};

   \tikzstyle{every node}=[font=\small,scale=0.9]
    \coordinate[](d1)at(6,0){};\coordinate[](d2)at(8,0){};\coordinate[](d3)at(10,0){};
    \coordinate[](e1)at(6,1){};\coordinate[](e2)at(8,1){};\coordinate[](e3)at(10,1){};
    \coordinate[](f1)at(6,2){};\coordinate[](f2)at(8,2){}; \coordinate[](f3)at(10,2){};
    \draw[fill=black!16,line width=0.8pt]
    (e1)--(e2)--(f2)--(f1)--(e1)--cycle;
    \draw[fill=black!16,line width=0.8pt]
    (e2)--(e3)--(d3)--(d2)--(e2)--cycle;
    \draw[densely dashed, line width=0.8pt] (f2)--(f3)--(e3);
    \draw[densely dashed, line width=0.8pt] (e1)--(d1)--(d2);
 \coordinate[label=225:$\mathbf{h}$]()at(d1){};
 \coordinate[label=225:$\mathbf{a}$]()at(d2){};
 \coordinate[label=225:$\mathbf{f}$]()at(d3){};
 \coordinate[label=225:$\mathbf{c}$]()at(e1){};
 \coordinate[label=225:$\mathbf{e}$]()at(e2){};
 \coordinate[label=225:$\mathbf{b}$]()at(e3){};
 \coordinate[label=225:$\mathbf{g}$]()at(f1){};
 \coordinate[label=225:$\mathbf{d}$]()at(f2){};
 \coordinate[label=225:$\mathbf{i}$]()at(f3){};

    \node[below,  scale=1] at (8,-0.5) {(b)\hspace{0.15cm} condition $(\dag\dag)$};
    \end{tikzpicture}
\vspace{0.1cm}\caption{An illustration for condition $(\dag)$ or $(\dag\dag)$}\label{fig1}
\end{figure}
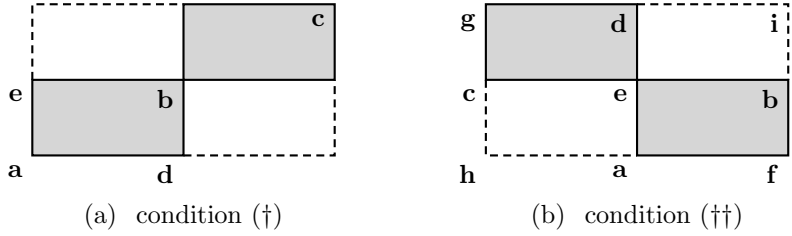

For convenience, we say $\mathcal {C}$ is of type $\mathcal {Q}_1$ (resp. $\mathcal {Q}_2$) if it satisfies the condition $(\dag)$ (resp. $(\dag\dag)$). In \cite{2025nrr}, for a convex collection $\mathcal{P}$ of cells of type $\mathcal {Q}_2$, the key for computing the Hilbert series of $\mathbb{K}[\mathcal{P}]$ is to choose a suitable vertex $\mathbf{v}$ such that \\[-0.3cm]
\[
\text{in}_{<'_{\text{lex}}}(I_{\mathcal{P}}):x_{\mathbf{v}} = \text{in}_{<'_{\text{lex}}}(I_{\mathcal{P}'})+J,\qquad
\big(\text{in}_{<'_{\text{lex}}}(I_{\mathcal{P}}),x_{\mathbf{v}}\big) = \text{in}_{<'_{\text{lex}}}(I_{\mathcal{P}''})+(x_{\mathbf{v}}),
\]\\[-0.3cm]
where $\mathcal{P}'$ and $\mathcal{P}''$ are also convex collections of cells of type $\mathcal {Q}_2$, and $J$ is an ideal generated by certain variables. In general, for an arbitrary collection of cells of type $\mathcal {Q}_2$, such a vertex $\mathbf{v}$ may not exist. However, we find that when $\mathcal{P}$ is a polyocollection of type $\mathcal {Q}_2$, one can still choose an appropriate vertex  $\mathbf{v}$ so that the above relations hold. In this case, $\mathcal{P}'$ and $\mathcal{P}''$ may not be collections of cells but polyocollections that also belong to type $\mathcal {Q}_2$. Moreover, the same situation applies to the condition $(\dag)$.  This motivates us to focus on polyocollections and polyocollection ideals in this paper. In Section \ref{sec2}, we recall some relevant basic notations and results. In Section \ref{sec3}, a class of polyocollections called natural polyocollections is defined. Furthermore, some concepts are extended naturally from collections of cells to polyocollections, and some results that hold for collections of cells are listed for polyocollections without proofs because polyocollections retain some structural features of collections of cells. Sections \ref{sec4} and \ref{sec5} focus on natural polyocollections  of type $\mathcal {Q}_1$ or of type $\mathcal {Q}_2$. In Section \ref{sec4}, it is proved that these classes of natural polyocollections are prime. Consequently, their Cohen--Macaulayness are established. In Section \ref{sec5}, the Krull dimension, Hilbert series  and regularity of these classes of natural polyocollections are computed. The main results are presented in  Theorem \ref{the primality of coc type O}, Corollary \ref{cor1}, Proposition \ref{dim-depth} and Theorem \ref{hs-r}. We conclude this section by stating the following Main Theorem. These results yield that Conjectures \ref{conj2} and \ref{conj1} are true for collection of cells of type  $\mathcal {Q}_1$  and of type  $\mathcal {Q}_2$.

\begin{main}
Let $\mathcal{C}$ be a natural polyocollection of type $\mathcal {Q}_1$ or of type $\mathcal {Q}_2$. Then the following statements hold$:$
\begin{itemize}
\item[(1)] $\mathcal{C}$ is prime. \\[-0.3cm]
\item[(2)] $\mathbb{K}[\mathcal {C}]$ is a Cohen-Macaulay domain with Krull dimension $|V(\mathcal {C})|-|\mathcal {C}|$.  \\[-0.3cm]
\item[(3)] The $h$-polynomial  of  $\mathbb{K}[{\mathcal{C}}]$ is $\widetilde{r}_{\mathcal {C}}(t)$ and its regularity is equal to $r(\mathcal {C})$, where $\widetilde{r}_{\mathcal {C}}(t)$  and $r(\mathcal {C})$ are the switching rook polynomial and rook number of $\mathcal {C}$ respectively.
\end{itemize}	
\end{main}


\section{Preliminaries} \label{sec2}

In this section,  we recall some  basic definitions, notations and results. For undefined terms and concepts, the reader is referred to \cite{2024cnv,hho,2012q}.

\subsection{Some definitions and results about algebraic invariants}

Let $R$ be a ring. The {\it Krull dimension} of $R$, denoted by  $\text{dim} (R)$, is the supermum of the length of all strictly decreasing chains of prime ideals of $R$. Let $M$ be a $R$-module.
The {\it Krull dimension} of $M$, denoted by $\text{dim} (M)$,  is defined as the Krull dimension of the quotient ring $R/\text{Ann}_R(M)$. Here $\text{Ann}_R(M)=\big\{x\in R \mid ax=0 \text{ for all $a\in M$ }\big\}$.

\medskip
Let $\mathbb{K}$ be a field and $R$ a standard graded $\mathbb{K}$-algebra with graded maximal ideal $\frak{m}$, and let $M$ be a finitely generated graded $R$-module. The {\it depth} of $M$, denoted by $\text{depth} (M)$, is the common length of maximal $M$-sequences contained in $\frak{m}$.
The  {\it Castelnuovo-Mumford regularity} (or for short  {\it regularity}) of $M$, denoted by $\text{reg}\,(M)$, is defined as
$$\text{reg}\,(M):=\text{max}\,\big\{j-i\ |\ \beta_{i,j}(M)\neq 0\big\},$$
where the number $\beta_{i,j}(M)$, called $(i,j)$-th graded Betti number of $M$, is an invariant of $M$ that equals the number of minimal generators of degree $j$ in the $i$-th syzygy module of $M$. Denote by $\text{dim}_{\mathbb{K}}(M_i)$ the dimension of the vector space of $i$-th graded component $M_i$ over  $\mathbb{K}$. The numerical function
$$H(M,-):\mathbb{Z}\longrightarrow \mathbb{Z}, \text{   } i\mapsto H(M,i):=\text{dim}_{\mathbb{K}}(M_i)$$
is called the {\it Hilbert function} of $M$. The  {\it Hilbert series} of $M$ is defined as
$$H_M(t)=\sum_{i\in \mathbb{Z}}H(M,i)t^{i}.$$
See \cite[Theorem 6.1.3]{2011hh}, there exists a unique polynomial $h_M(t)$ such that
$$H_M(t)=\frac{h_M(t)}{(1-t)^d},$$
where $d=\text{dim}(M)$. Such polynomial $h_M(t)$ is called the {\it $h$-polynomial} of $M$.

\medskip
We now present some results concerning these algebraic invariants, which will be useful later.

\begin{lem}{\em (\cite[Theorem 4.4.3]{bh})}\label{h-r} Let $S=\mathbb{K}[x_1,x_2,\ldots, x_n]$ be a polynomial ring over a field $\mathbb{K}$ and $I\subseteq S$ a homogenous ideals such that $S/I$ is Cohen-Macaulay. Then \\[-0.4cm]
$$\text{\em reg}(S/I)=\text{\em degree}(h_{S/I}(t)).$$
\end{lem}

\begin{lem}{\em (\cite[Theorem 3.3.4 and Corollary 6.1.5]{2011hh}, \cite[Corollary 2.7]{2020c} )} \label{intial}
Let $S=\mathbb{K}[x_1,x_2,\ldots, x_n]$ be a polynomial ring over a field $\mathbb{K}$ and $I$ a homogenous ideals of $S$. Then the following statements hold for any monomial order $<$ on $S$$:$\\[-0.3cm]
\begin{itemize}
\item[(1)] $\text{\em dim}(S/I)=\text{\em dim}(S/\text{\em in}_<(I))$$;$ \\[-0.3cm]
\item[(2)] The Hilbert functions of $S/I$ and $S/\text{\em in}_<(I)$ coincide. Consequently, $H_{S/I}(t)=H_{S/\text{\em in}_<(I)}(t)$. \\[-0.3cm]
\end{itemize}
In addition, if $\text{\em in}_<(I)$ is square-free, then $\text{\em depth}(S/I)=\text{\em depth}(S/\text{\em in}_<(I))$.
\end{lem}

From \cite[Exercise 2.1.14, Propostion 2.2.20]{2015V} it follows the following lemma.
\begin{lem}\label{sum}
Let $S_{1}=k[x_{1},\dots,x_{m}]$, $S_{2}=k[x_{m+1},\dots,x_{n}]$ be two polynomial rings and $S=S_1\otimes_k S_2$, let $I\subset S_{1}$, $J\subset S_{2}$ be two nonzero homogeneous ideals. Then\\[-0.3cm]
\begin{itemize}
\item[(1)] $\text{\em dim}(S/(I+J))=\text{\em dim}(S_1/I)+\text{\em dim}(S_2/J)$$;$\\[-0.3cm]
\item[(2)] $H_{S/(I+J)}(t)=H_{S_1/I}(t)\cdot H_{S_2/J}(t)$.\\[-0.3cm]
\end{itemize}
\end{lem}

\begin{lem}\label{exact} Let $S=\mathbb{K}[x_1,x_2,\ldots, x_n]$ be a polynomial ring over a field $\mathbb{K}$ and $I\subseteq S$ a homogenous ideals of $S$. Then for any homogenous  polynomial $f$,\\[-0.3cm]
$$H_{S/I}(t)=t^{d}\cdot H_{S/(I:f)}(t)+H_{S/(I,f)}(t),$$
where $d=\text{\em degree}(f)$.
\end{lem}

\begin{proof}
This lemma follows from the short exact sequence:\\[-0.3cm]
$$0\longrightarrow \frac{S}{(I:f)}(-d) \stackrel{ \cdot f} \longrightarrow \frac{S}{I}\longrightarrow  \frac{S}{(I,f)}\longrightarrow  0.$$
\end{proof}

\subsection{Some definitions concerning polyocollections}

Equip  $\mathbb{R}^2$ with the natural partial order $\leq$, defined as follows: for any $(i,j),(k,\ell)\in \mathbb{R}^2$,   $(i,j)\leq (k,\ell)$ if and only if $i\leq k$ and $j\leq \ell$.

Let $\mathbf{a}=(i,j), \mathbf{b}=(k,\ell)\in \mathbb{Z}^2$ with $\mathbf{a}\leq \mathbf{b}$.
If $i<k$ and $j<\ell$, the interval $[\mathbf{a},\mathbf{b}]$ is called {\it proper}. In such a case, the elements $\mathbf{a}, \mathbf{b}$  are called the {\it diagonal corners} of  $[\mathbf{a},\mathbf{b}]$, while $\mathbf{c}=(i,\ell)$ and $\mathbf{d}=(k,j)$ are {\it anti-diagonal corners} of  $[\mathbf{a},\mathbf{b}]$. Moreover, the elements $\mathbf{a},\mathbf{b},\mathbf{c},\mathbf{d}$ are called the lower left, upper right, upper left and lower right corners of $[\mathbf{a}, \mathbf{b}]$ respectively.  The four corners are also called the {\it vertices} of the interval. The vertex set of $[\mathbf{a},\mathbf{b}]$ is denoted by $V([\mathbf{a},\mathbf{b}])$. The edge set of $[\mathbf{a},\mathbf{b}]$ is defined by $E([\mathbf{a},\mathbf{b}])=\{[\mathbf{a}, \mathbf{c}], [\mathbf{a}, \mathbf{d}], [\mathbf{c}, \mathbf{b}], [\mathbf{d}, \mathbf{b}]\}$. In particular, an interval of the form $C=[\mathbf{a}, \mathbf{a}+(1,1)]$ is called a cell. For the interval $[\mathbf{a},\mathbf{b}]$ with $\mathbf{a}=(i,j)\in \mathbb{R}^2$ and $\mathbf{b}=(k,\ell)\in \mathbb{R}^2$, we denote $\text{c}([\mathbf{a},\mathbf{b}])=\big\{(r,s)\in \mathbb{R}^2\mid i\leq r\leq k, j\leq s\leq \ell\big\}$
and $\text{int}([\mathbf{a},\mathbf{b}])=\big\{(r,s)\in \mathbb{R}^2\mid i< r< k, j< s< \ell\big\}$. These two set are called the {\it closure} and {\it interior} of $[\mathbf{a},\mathbf{b}]$ respectively.

Let  $\mathcal {C}$ be a finite collection of proper intervals in  $\mathbb{Z}^2$. Its vertex set and edge set is defined by $V(\mathcal {C})=\bigcup_{I\in \mathcal {C}}V(I)$ and $E(\mathcal {C})=\bigcup_{I\in \mathcal {C}}E(I)$, respectively. The rank of $\mathcal {C}$, denoted by $|\mathcal {C}|$, is the number of proper intervals in  $\mathcal {C}$. In \cite{2024cnv},
$\mathcal {C}$ is called a {\it polyocollection} if for any two distinct intervals $I,J\in \mathcal {C}$, we have $I\nsubseteq J$ and one of the followings holds:
\begin{itemize}
\item[(1)] $I\cap J$  is a common edge of $I$ and $J$;\\[-0.4cm]
\item[(2)] For each $F\in E(I)$ and each $G\in E(J)$, $|F\cap G|\leq 1$.\\[-0.4cm]
\end{itemize}
In this paper, we call the proper intervals in $\mathcal{C}$ {\it basic intervals} of $\mathcal{C}$ to distinguish them from other kinds of intervals in $\mathcal{C}$.


Let $\mathcal {C}$ be a polyocollection, and let $I=[\mathbf{a}, \mathbf{b}]$ be an interval with $\mathbf{a}=(i,j)$ and $\mathbf{b}=(k,\ell)$. If $j=\ell$ and there exist $I_1, I_2,\ldots, I_n\in E(\mathcal {C})$ such that $I$ consists of $I_1, I_2,\ldots, I_n$,
 then $[\mathbf{a}, \mathbf{b}]$ is called a {\it horizontal edge interval} of $\mathcal {C}$. Similarly, if $i=k$ and there exist $I_1, I_2,\ldots, I_m\in E(\mathcal {C})$ such that $I$ consists of $I_1, I_2,\ldots, I_m$,
  then $[\mathbf{a}, \mathbf{b}]$  is called a {\it vertical  edge interval} of $\mathcal {C}$.   A horizontal (vertical) edge interval is called {\it maximal} if it is maximal among all horizontal (vertical) edge intervals with respect to the inclusion of sets. If $i<k$, $j<\ell$ and there exist basic intervals $I_1, I_2,\ldots, I_n\in \mathcal {C}$  such that $\text{c}(I)=\bigcup_{t=1}^n\text{c}(I_t)$ and $\text{int}(I_p)\cap \text{int}(I_q)=\emptyset$ for any $p,q=1,\ldots, n$ with $p\neq q$, then $[\mathbf{a}, \mathbf{b}]$ is called an {\it inner interval} of $\mathcal {C}$.
Denote by $\mathcal {I}(\mathcal {C})$ the collection of inner intervals of $\mathcal {C}$. For any $I\in \mathcal {I}(\mathcal {C})$, denote by
$llc(I)$, $urc(I)$, $ulc(I)$ and $lrc(I)$ its the lower left, upper right, upper left and lower right corners  respectively. Moreover, denote by $I_{\mathbf{v}}(\mathcal {C})$ (resp. $\overline{I_{\mathbf{v}}}(\mathcal {C})$) the basic interval of $\mathcal {C}$ with $\mathbf{v}$ being its lower left (resp. upper right) corner.

Let $\mathbb{K}$ be a field and $S_{\mathcal {C}}=\mathbb{K}[x_{\mathbf{v}}\mid \mathbf{v}\in V(\mathcal {C})]$. For any inner interval $I=[\mathbf{a}, \mathbf{b}]$ of $\mathcal {C}$ with anti-diagonal corners $\mathbf{c}, \mathbf{d}$, the binomial $f_{I}=x_{\mathbf{a}}x_{\mathbf{b}}-x_{\mathbf{c}}x_{\mathbf{d}}$ is called an {\it inner 2-minor} of  $\mathcal {C}$. The ideal $I_{\mathcal {C}}\subseteq S_{\mathcal {C}}$ generated by all the inner $2$-minors of $\mathcal {C}$ is called the {\it polyocollection ideal} of  $\mathcal {C}$. The quotient ring $\mathbb{K}[{\mathcal{C}}]:=S_{\mathcal{C}}/I_{\mathcal{C}}$ is called the {\it coordinate ring} of  $\mathcal{C}$.

 A finite collection $\mathcal {C}$ of proper intervals in $\mathbb{R}^2$ can be naturally regarded as a collection of proper intervals in $\mathbb{Z}^2$. Especially, if the proper intervals in $\mathcal {C}$ satisfy the conditions in the definition of a polyocollection, we regard $\mathcal {C}$ as a polyocollection. Moreover, two distinct polyocollections $\mathcal {C}_1$ and $\mathcal {C}_2$ are called {\it algebraically isomorphic} if $\mathbb{K}[\mathcal {C}_{1}]$ and $\mathbb{K}[\mathcal {C}_{2}]$ are isomorphic as $\mathbb{K}$-algebras. For example, for a polyocollection $\mathcal {C}$ and a $k\in \mathbb{R}$, one can define $k\mathcal {C}$ as follows:
$$k\mathcal {C} = \{[k\mathbf{a}, k\mathbf{b}] \,|\, [\mathbf{a}, \mathbf{b}] \in \mathcal {C}\}.$$
By the above discussion, $k\mathcal {C}$ could be regarded as a polyocollection. It is direct to check that $\mathcal {C}$ and $k\mathcal {C}$ are algebraically isomorphic. Some polyocollections are algebraically isomorphic to some collections of cells, whereas others  cannot, as shown in  \cite[Section 3]{2024cnv}.

Let $\mathcal {C}$ be a polyocollection, and let $I_1=[\mathbf{u}_1, \mathbf{v}_1]$ and $I_2=[\mathbf{u}_2, \mathbf{v}_2]$ be two distinct basic intervals in $\mathcal {C}$ with $\mathbf{u}_1\leq \mathbf{u}_2$. If $[\mathbf{u}_1,\mathbf{v}_2]$ is an inner interval of $\mathcal {C}$ and  $[\mathbf{u}_1, \mathbf{u}_2]$ is a horizontal edge interval (resp. vertical edge interval) of $\mathcal {C}$, then we say that $I_1$ and $I_2$ are in the same row (resp. the same column) of $\mathcal {C}$.  If $[\mathbf{u}_1, \mathbf{v}_2]$ is an inner interval of $\mathcal {C}$ and $I_1, I_2$ are not in a same row or a same column, then  we say that $I_1$ and $I_2$ are in diagonal position of $[\mathbf{u}_1, \mathbf{v}_2]$, and we say that $I'_1$ and $I'_2$ are in anti-diagonal position of $[\mathbf{u}_1, \mathbf{v}_2]$, where $I'_1$ is the basic interval of $\mathcal {C}$ whose upper left corner coincides with the upper left corner of $[\mathbf{u}_1, \mathbf{v}_2]$, and  $I'_2$ is the basic interval of $\mathcal {C}$ whose lower right corner coincides with lower right corner of $[\mathbf{u}_1, \mathbf{v}_2]$.

\begin{ex}
Let $\mathcal {C}$ be a polyocollection as shown in Figure \ref{fig2}. Then $[\mathbf{u}_1, \mathbf{u}_2]$ (resp. $[\mathbf{u}_1, \mathbf{u}_{13}]$) is a maximal horizontal (resp. vertical) edge interval of $\mathcal {C}$ and $[\mathbf{u}_3,\mathbf{u}_{15}]$ is an inner interval of $\mathcal {C}$. Moreover, $[\mathbf{u}_3,\mathbf{u}_9]$ and $[\mathbf{u}_8,\mathbf{u}_{14}]$  (resp. $[\mathbf{u}_3,\mathbf{u}_9]$ and $[\mathbf{u}_4,\mathbf{u}_{10}]$) are in the same column (resp. the same row) of $\mathcal {C}$, and $[\mathbf{u}_3,\mathbf{u}_9]$ and $[\mathbf{u}_9,\mathbf{u}_{15}]$ (resp. $[\mathbf{u}_8,\mathbf{u}_{14}]$ and $[\mathbf{u}_4,\mathbf{u}_{10}]$) are in diagonal position (resp.  anti-diagonal position) of $[\mathbf{u}_3,\mathbf{u}_{15}]$.
\end{ex}

\begin{figure}[h]
\centering
    \begin{tikzpicture}[scale=1.2, transform shape]
    \draw[fill=black!16,line width=0.8pt]
    (0,0)--(4.5,0)--(4.5,0.5)--(0,0.5)--(0,0)--cycle
    (0,0.5)--(0.5,0.5)--(0.5,1)--(1.5,1)--(1.5,1.5)--(3.5,1.5)-- (3.5,0.5)--(4.5,0.5)--(4.5,2.5)--(0,2.5)--(0,0.5)--cycle;
 \draw[fill=blue!16,line width=0.8pt](2,2)--(2,1)--(3,1)--(3,2)--(2,2)--cycle;
  \draw[line width=0.8pt] (0,1)--(0.5,1)  (0,1.5)--(4.5,1.5) (0.5,0)--(0.5,2.5)
  (1.5,1)--(1.5, 2.5) (3.5,0)--(3.5,2.5);
  \draw[fill=black] (0,2.5) circle (1pt);
  \draw[fill=black] (0.5,2.5) circle (1pt);
  \draw[fill=black] (1.5,2.5) circle (1pt);
  \draw[fill=black] (3.5,2.5) circle (1pt);
  \draw[fill=black] (4.5,2.5) circle (1pt);

  \draw[fill=black] (0,1.5) circle (1pt);
  \draw[fill=black] (0.5,1.5) circle (1pt);
  \draw[fill=black] (1.5,1.5) circle (1pt);
  \draw[fill=black] (3.5,1.5) circle (1pt);
  \draw[fill=black] (4.5,1.5) circle (1pt);

  \draw[fill=black] (0,1) circle (1pt);
  \draw[fill=black] (0.5,1) circle (1pt);
  \draw[fill=black] (1.5,1) circle (1pt);
  \draw[fill=black] (2,1) circle (1pt);
  \draw[fill=black] (3,1) circle (1pt);

  \draw[fill=black] (2,2) circle (1pt);
  \draw[fill=black] (3,2) circle (1pt);

  \draw[fill=black] (0,0.5) circle (1pt);
  \draw[fill=black] (0.5,0.5) circle (1pt);
  \draw[fill=black] (3.5,0.5) circle (1pt);
  \draw[fill=black] (4.5,0.5) circle (1pt);

  \draw[fill=black] (0,0) circle (1pt);
  \draw[fill=black] (0.5,0) circle (1pt);
  \draw[fill=black] (3.5,0) circle (1pt);
  \draw[fill=black] (4.5,0) circle (1pt);

   \node[below, font=\small,scale=0.6] at (0,0) {$\mathbf{u}_{1}$};
   \node[below, font=\small,scale=0.6] at (4.5,0) {$\mathbf{u}_{2}$};
   \node[below, font=\small,scale=0.6] at (-0.15,1) {$\mathbf{u}_{3}$};
   \node[below, font=\small,scale=0.6] at (0.35,1) {$\mathbf{u}_{4}$};
   \node[below, font=\small,scale=0.6] at (1.45,1) {$\mathbf{u}_{5}$};
   \node[below, font=\small,scale=0.6] at (1.95,1) {$\mathbf{u}_{6}$};
   \node[below, font=\small,scale=0.6] at (2.95,1) {$\mathbf{u}_{7}$};

   \node[below, font=\small,scale=0.6] at (-0.15,1.5) {$\mathbf{u}_{8}$};
   \node[below, font=\small,scale=0.6] at (0.35,1.5) {$\mathbf{u}_{9}$};
   \node[below, font=\small,scale=0.6] at (1.3,1.5) {$\mathbf{u}_{10}$};
   \node[above, font=\small,scale=0.6] at (2,2) {$\mathbf{u}_{11}$};
   \node[above, font=\small,scale=0.6] at (3,2) {$\mathbf{u}_{12}$};

   \node[above, font=\small,scale=0.6] at (0,2.5) {$\mathbf{u}_{13}$};
   \node[above, font=\small,scale=0.6] at (0.5,2.5) {$\mathbf{u}_{14}$};
   \node[above, font=\small,scale=0.6] at (1.5,2.5) {$\mathbf{u}_{15}$};

    \end{tikzpicture}
\vspace{-0.2cm}\caption{An example of a polyocollection}\label{fig2}
\end{figure}
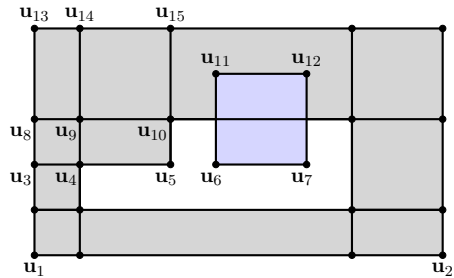


Let  $I,J$  be two distinct basic intervals of $\mathcal {C}$. A {\it path} from $I$ to $J$ in  $\mathcal {C}$ is a sequence $\pi: I=I_0, I_1,\ldots,I_k=J$ of distinct basic intervals  in $\mathcal {C}$ such that $E(I_{i-1})\cap E(I_i)\neq \emptyset$ for each $i=1,\ldots, k$. If such a path exists, we say that intervals $I$ and $J$ are {\it connected}.  The polyocollection $\mathcal {C}$ is {\it connected} if any two basic intervals in $\mathcal {C}$ are connected. Moreover, the polyocollection $\mathcal {C}$ is {\it weakly connected} if for any two basic intervals $I$ and $J$ in $\mathcal {C}$, there exists a sequence $\pi: I=I_0, I_1,\ldots,I_k=J$ of basic intervals in $\mathcal {C}$ such that $V(I_{i-1})\cap V(I_i)\neq \emptyset$  for $i=1,\ldots, k$. A subset $\mathcal {C}'$ of $\mathcal {C}$ is called a {\it connected component} (resp. {\it weakly connected component}) of $\mathcal {C}$ if $\mathcal {C}'$ is connected (resp. weakly connected) and maximal with respect to inclusion.

\section{Generalizing some notations and results on collections of cells to  polyocollections}\label{sec3}

In this section, we generalize certain notions and results on collections of cells to polyocollections.

\subsection{Primality of polyocollection ideals}

Let $\mathcal {C}$ be a polyocollection. A function $\alpha: V(\mathcal {C})\rightarrow \mathbb{Z}$ is called a {\it labeling} of $\mathcal {C}$. It is clear that each labeling could be viewed as a vector on $\mathbb{Z}^{|V(\mathcal {C})|}$. A labeling $\alpha$ is called {\it admissible} if $\alpha([\mathbf{a}, \mathbf{b}]):=\sum_{\mathbf{c}\in[\mathbf{a}, \mathbf{b}]\cap V(\mathcal {C})}\alpha(\mathbf{c})=0$  for each maximal edge interval $[\mathbf{a}, \mathbf{b}]$  of $\mathcal {C}$. Let $I=[\mathbf{a}, \mathbf{b}]$ be an inner interval of $\mathcal {C}$  with anti-diagonal corners $\mathbf{c}$ and $\mathbf{d}$, denote by $\mu_{I}$ the labeling on $\mathcal{C}$ such that $\mu_{I}(\mathbf{a})=\mu_{I}(\mathbf{b})=1$, $\mu_{I}(\mathbf{c})=\mu_{I}(\mathbf{d})=-1$, and  $\mu_{I}(\mathbf{u})=0$ for other vertices $\mathbf{u}$ in $\mathcal{C}$. $\Lambda_{\mathcal{C}}$ is the {\it lattice} spanned by all $\mu_{I}$ for $I\in \mathcal{C}$. A labeling $\alpha$ can be written uniquely as $\alpha =\alpha^+ - \alpha^-$, where $\alpha^+$ denotes the vector obtained from $\alpha$ by replacing all negative components of $\alpha$ by zeroes, and $\alpha^-=\alpha^+ - \alpha$. The ideal $I_{\Lambda_{\mathcal{C}}}$ is generated by all binomials $f_{\alpha}=\mathbf{x}^{\alpha^+}-\mathbf{x}^{\alpha^-}$ for $\alpha \in \Lambda_{\mathcal{C}}$. It is direct to check that $I_\mathcal{C} \seq I_{\Lambda_{\mathcal{C}}}$. From \cite[Theorem 3.8]{2024cnv} it follows that $I_\mathcal{C}$ is prime if and only if $I_\mathcal{C}=I_{\Lambda_{\mathcal{C}}}$. For an inner interval $I=[\mathbf{a}, \mathbf{b}]$ with anti-diagonal corners $\mathbf{c}$ and $\mathbf{d}$, if $\al(\mathbf{a})\al(\mathbf{b})>0$, then denote \\[0.1cm]
\begin{displaymath}
\beta=\left\{
\begin{array}{r@{,\quad}l}
\alpha-\mu_I &\al(\mathbf{a})>0,\\
\alpha+\mu_I &\al(\mathbf{a})<0.
\end{array} \right.
\end{displaymath}

If $\al(\mathbf{c})\al(\mathbf{d})>0$, then denote \\[0.1cm]
\begin{displaymath}
\beta=\left\{
\begin{array}{r@{,\quad}l}
\alpha+\mu_I &\al(\mathbf{c})>0,\\
\alpha-\mu_I &\al(\mathbf{c})<0.
\end{array} \right.
\end{displaymath}

Similar to \cite{2012q}, $\beta$ is said to be {\em obtained} from $\al$ (or $\al$ is {\em reduced} to $\beta$) by a {\em single move}. $\al$ is said to be {\em reduced} to $\beta$ by $k$ moves, if there exists a sequence $\al=\al_0,\al_1,\ldots, \al_k=\beta$ such that $\al_{i+1}$ is obtained from $\al_i$ by a single move for each $i=0,\ldots,k-1$. Especially, we also call that $\al$ can be reduced to $\al$ by $0$ move. Similar to \cite{ZGW}, for an admissible labeling $\al$ of $\mathcal{C}$ and a subset $S \seq V(\mathcal{C})$, denote $\sigma(\alpha|_S)=\sum_{\mathbf{v}\in S} |\alpha(\mathbf{v})|$. $\sigma(\alpha)=\sigma(\alpha|_{V(\mathcal{C})})$ is called the {\em norm} of $\alpha$.  It is clear that if $\al$ can be reduced to $\beta$, then $\sigma(\beta) \leq \sigma(\alpha)$. Furthermore, if $\sigma(\beta) < \sigma(\alpha)$, then $\alpha$ is said to be {\em effectively reduced} to $\beta$. Similar to (\cite[Theorem 4.2]{ZGW}), we can obtain the following proposition.
\begin{prop}\label{reduce}
Let $\mathcal{C}$ be a polyocollection, and let $\al$, $\beta$ be two admissible labelings of $\mathcal{C}$. Assume that $\al$ can be reduced to $\beta$, then $\beta\in \Lambda_{\mathcal {C}} $ if and only if $\al\in \Lambda_{\mathcal {C}}$, and $f_\beta \in I_\mathcal{C}$ implies $f_\alpha \in I_\mathcal{C}$. Assume further $\sigma(\alpha)=\sigma(\beta)$, then $f_\alpha \in I_\mathcal{C}$ if and only if $f_\beta \in I_\mathcal{C}$.
\end{prop}

Theorem 4.6 in \cite{ZGW} presents a method for judging the primality of collections of cells. The following result extends this idea to polycollections.

\begin{prop}\label{equivalent main idea to prove the primality}
Let $\mathcal{C}$ be a polyocollection. If each $\alpha \in \Lambda_{\mathcal{C}}$ with $\sigma(\alpha)>0$ can be effectively reduced, then $I_\mathcal{C} = I_{\Lambda_{\mathcal{C}}}$. As a consequence, $I_\mathcal{C}$ is prime.
\end{prop}
\begin{proof} For each $\alpha \in \Lambda_{\mathcal{C}}$, we show that $f_\alpha\in I_\mathcal{C}$ by induction on $\sigma(\alpha)$. If $\sigma(\alpha)=0$, then $\alpha=\mathbf{0}$ and $f_\alpha=0\in I_\mathcal{C}$. Assume that $f_\alpha\in I_\mathcal{C}$ if $\sigma(\alpha)<k$, where $k$ is a positive integer. We claim that $f_\alpha\in I_\mathcal{C}$ if $\sigma(\alpha)=k$. In fact, note that $\al$ can be effectively reduced to some labeling $\beta$, so it follows from Proposition \ref{reduce} that $\beta\in \Lambda_{\mathcal{C}}$. Since $\sigma(\beta)< \sigma(\alpha)=k$, it follows from induction  hypothesis that $f_\beta \in I_{\mathcal{C}}$. By Proposition \ref{reduce} again, one has $f_\alpha\in I_\mathcal{C}$. Hence $I_\mathcal{C} = I_{\Lambda_{\mathcal{C}}}$. Consequently, $I_\mathcal{C}$ is prime.
\end{proof}

As in \cite[Definition 3.2]{2015hq1},  zig-zag walks in a polyocollection can be defined.

\begin{dfn}
Let $\mathcal {C}$ be a polyocollection.  A {\it zig-zag walk} of $\mathcal {C}$ is a finite sequence of distinct inner intervals $\mathcal {W}: I_1,\ldots, I_k$ with $k\geq 2$ such that, for each $1\leq i\leq k$ the interval $I_i$ has either the diagonal corners $\mathbf{v}_i,\mathbf{z}_i$ and anti-diagonal corners $\mathbf{u}_i,\mathbf{v}_{i+1}$ or the anti-diagonal corners $\mathbf{v}_i,\mathbf{z}_i$ and diagonal corners $\mathbf{u}_i,\mathbf{v}_{i+1}$, where $\mathbf{v}_{k+1}=\mathbf{v}_1$, and the following conditions are satisfied$:$\\[-0.4cm]
\begin{itemize}
\item[(1)] $I_i\cap I_{i+1}=\{\mathbf{v}_{i+1}\}$ for $i \in \{1,\ldots,k\}$;\\[-0.4cm]
\item[(2)] $\mathbf{v}_i$ and $\mathbf{v}_{i+1}$ lie on the same horizontal or vertical edge interval of $\mathcal {C}$ for $i\in \{1,\ldots,k\}$;\\[-0.4cm]
\item[(3)] For any distinct $i,j\in \{1,\ldots,k\}$, there is no inner interval $J$ in $\mathcal {C}$ containing both $\mathbf{z}_i$ and $\mathbf{z}_j$.   \\[-0.4cm]
\end{itemize}
\end{dfn}

For example, the sequence $[\mathbf{u}_{1}, \mathbf{v}_{2}], [\mathbf{v}_{2}, \mathbf{z}_{2}], [\mathbf{v}_{4}, \mathbf{u}_{3}],[\mathbf{z}_{4}, \mathbf{v}_{4}]$ is a zig-zag walk of the polyocollection in Figure \ref{fig3}.

\begin{figure}[h]
\centering
    \begin{tikzpicture}[scale=1.2, transform shape]
    \draw[fill=black!16,line width=0.8pt]
    (0.5,0)--(3.5,0)--(3.5,0.5)--(0.5,0.5)--(0.5,0)--cycle
    (0,0.5)--(0.5,0.5)--(0.5,1)--(1.5,1)--(1.5,1.5)--(0,1.5)--(0,0.5)--cycle
    (0.5,1.5)--(3.5,1.5)--(3.5,2.5)--(0.5,2.5)--(0.5,1.5)--cycle
    (3.5,0.5)--(4.5,0.5)--(4.5,1.5)--(3,1.5)--(3,1)--(3.5,1)--(3.5,0.5)--cycle;
 \draw[fill=blue!16,line width=0.8pt](2,2)--(2,1)--(2.5,1)--(2.5,2)--(2,2)--cycle;
  \draw[line width=0.8pt] (0,1)--(0.5,1) (3.5,1)--(4.5,1) (0,1.5)--(4.5,1.5) (0.5,0)--(0.5,2.5)
  (1.5,1)--(1.5, 2.5) (3.5,0)--(3.5,2.5) (3,1.5)--(3,2.5);

 \textcolor[rgb]{0.00,0.59,0.00}{\draw[line width=0.8pt] (0.5,0)--(3.5,0)--(3.5,0.5)--(0.5,0.5)--(0.5,0)
(0.5,1.5)--(3.5,1.5)--(3.5,2.5)--(0.5,2.5)--(0.5,1.5);}
 \textcolor[rgb]{0.00,0.07,1.00}{\draw[line width=0.8pt]  (0,0.5)--(0.5,0.5)--(0.5,1.5)--(0,1.5)--(0,0.5)
(3.5,0.5)--(4.5,0.5)--(4.5,1.5)--(3.5,1.5)--(3.5,0.5);}

  \draw[fill=black] (0.5,2.5) circle (1pt);
  \draw[fill=black] (1.5,2.5) circle (1pt);
  \draw[fill=black] (3.5,2.5) circle (1pt);
  \draw[fill=black] (3,2.5) circle (1pt);

  \draw[fill=black] (0,1.5) circle (1pt);
  \draw[fill=black] (0.5,1.5) circle (1pt);
  \draw[fill=black] (1.5,1.5) circle (1pt);
  \draw[fill=black] (3.5,1.5) circle (1pt);
  \draw[fill=black] (4.5,1.5) circle (1pt);
  \draw[fill=black] (3,1.5) circle (1pt);

  \draw[fill=black] (0,1) circle (1pt);
  \draw[fill=black] (0.5,1) circle (1pt);
  \draw[fill=black] (1.5,1) circle (1pt);
  \draw[fill=black] (2,1) circle (1pt);
  \draw[fill=black] (3,1) circle (1pt);
  \draw[fill=black] (3.5,1) circle (1pt);
  \draw[fill=black] (4.5,1) circle (1pt);

  \draw[fill=black] (2,2) circle (1pt);
  \draw[fill=black] (2.5,1) circle (1pt);
  \draw[fill=black] (2.5,2) circle (1pt);

  \draw[fill=black] (0,0.5) circle (1pt);
  \draw[fill=black] (0.5,0.5) circle (1pt);
  \draw[fill=black] (3.5,0.5) circle (1pt);

  \draw[fill=black] (0.5,0) circle (1pt);
  \draw[fill=black] (3.5,0) circle (1pt);

  \node[below, font=\small,scale=0.6] at (0.35,0.5) {$\mathbf{v}_{1}$};
  \node[below, font=\small,scale=0.6] at (-0.15,0.5) {$\mathbf{u}_{1}$};
  \node[below, font=\small,scale=0.6] at (-0.15,1.5) {$\mathbf{z}_{1}$};

  \node[below, font=\small,scale=0.6] at (0.35,1.5) {$\mathbf{v}_{2}$};
  \node[above, font=\small,scale=0.6] at (0.5,2.5) {$\mathbf{u}_{2}$};
  \node[above, font=\small,scale=0.6] at (3.5,2.5) {$\mathbf{z}_{2}$};

  \node[below, font=\small,scale=0.6] at (3.35,1.5) {$\mathbf{v}_{3}$};
  \node[right, font=\small,scale=0.6] at (4.5,1.5) {$\mathbf{u}_{3}$};
  \node[right, font=\small,scale=0.6] at (4.5,0.5) {$\mathbf{z}_{3}$};

  \node[below, font=\small,scale=0.6] at (3.35,0.5) {$\mathbf{v}_{4}$};
  \node[below, font=\small,scale=0.6] at (3.5,0) {$\mathbf{u}_{4}$};
  \node[below, font=\small,scale=0.6] at (0.5,0) {$\mathbf{z}_{4}$};

    \end{tikzpicture}
\vspace{-0.2cm}\caption{An example of a zig-zag walk}\label{fig3}
\end{figure}
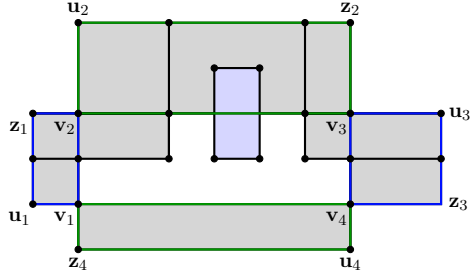

In the following, we introduce a class of polyocollections which shares a similar structure with collections of cells.

\begin{dfn}\label{polyocollection}
A polyocollection $\mathcal {C}$  is called a {\it natural polyocollection} if for any two distinct basic intervals $I,J\in \mathcal {C}$, we have $\text{int}(I)\cap \text{int}(J)=\emptyset$.
\end{dfn}

For example, the polyocollection $\mathcal {C}$ in Figure \ref{fig2} is not natural, whereas the  polyocollection obtained by removing the  basic interval $[\mathbf{u}_{6},\mathbf{u}_{12}]$ is natural.


Let $\mathcal {C}$ be a natural polyocollection, and let $\mathcal{P}_{\mathcal {C}}$ be its associated collection of cells, given by
$$\mathcal{P}_{\mathcal {C}}=\{C \,|\, C \,\, \text{is a cell in} \,\, \mathbb{Z}^2, \,\, \text{and there exists}\,\, I \in \mathcal {C} \,\, \text{such that} \,\, C \seq I \}.$$
We say $\mathcal {C}$ is {\it simple} if $\mathcal {P}_{\mathcal {C}}$ is simple.
Based on the structural similarity between natural polyocollections and collections of cells, together with  \cite[Proposition 1.6 and Theorem 2.1]{2014hs} (or see \cite[Theorem 2.2]{2017q}, or see \cite[Theorem 3.2]{2022cnu}) and \cite[Corollary 3.6]{2020mrr}, the following conclusion is readily obtained.

\begin{thm}
Let $\mathcal {C}$ be a  polyocollection. Then the following statements hold$:$
\begin{itemize}
\item[(1)] If $\mathcal {C}$ is natural and simple, then $\mathcal {C}$ is prime$;$
\item[(2)] If $\mathcal {C}$ is prime, then $\mathcal {C}$ contains no zig-zag walk.
\end{itemize}
\end{thm}

Next, we naturally extend the definition and conclusions about zero-sum condition, established in \cite{ZGW} for collections of cells, to natural polyocollections. This  generalization will be used to investigate the primality of  natural polyocollections in Section \ref{sec4}.


Let $\mathcal {C}$ be a natural polyocollection and $\mathcal {P}_{\mathcal {C}}$ the associated collection of cells. For each $\mathbf{u}\in V(\mathbb{Z}^2)$, $\mathcal{C}$ is divided into four parts by the vertical line and the horizontal line passing through $\mathbf{u}$. The upper right, upper left, lower left, lower right parts
of $\mathbf{u}$ are denoted by  $\mathcal{C}^{(1)}_\mathbf{u}$, $\mathcal{C}^{(2)}_\mathbf{u}$, $\mathcal{C}^{(-1)}_\mathbf{u}$, $\mathcal{C}^{(-2)}_\mathbf{u}$ respectively, and the upper right, upper left, lower left, lower right cells in $\mathbb{Z}^2$ containing $\mathbf{u}$ are denoted by  $C^{(1)}_\mathbf{u}, C^{(2)}_\mathbf{u}, C^{(-1)}_\mathbf{u}, C^{(-2)}_\mathbf{u}$ respectively. For example,
for the natural polyocollection $\mathcal{C}$ and the vertex $\mathbf{u}\in V(\mathcal{C})$ in Figure \ref{fig4}, $\mathcal{C}^{(-2)}_\mathbf{u}$ consists of the basic interval $[\mathbf{b},\mathbf{e}]$, half of the basic interval $[\mathbf{a},\mathbf{b}]$ and the edges $[\mathbf{f},\mathbf{u}]$ and $[\mathbf{u},\mathbf{d}]$. It is clear that $\mathcal{C}^{(-2)}_\mathbf{u}$ is not a polyocollection.
Let $\alpha$ be an admissible labeling of $\mathcal{C}$, and let $\mathcal{W}$ be a subset of $V(\mathcal{C})$, denote $\al(\mathcal{W})=\sum_{\mathbf{v}\in \mathcal{W}} \al(\mathbf{v})$. If $\mathcal{C}'$ is a natural subpolyocollection of $\mathcal{C}$, or $\mathcal{C}'=\mathcal{C}^{(i)}_\mathbf{u}$ for some $i\in\{1, -1, 2, -2\}$, denote by $\al(\mathcal{C}')=\sum_{\mathbf{v}\in V(\mathcal{C}')} \al(\mathbf{v})$.  If for each $i \in \{1, -1, 2, -2\}$,  $C^{(-i)}_\mathbf{u} \not\in \mathcal {P}_{\mathcal {C}}$ implies $\alpha (\mathcal{C}^{(i)}_\mathbf{u})=0$, then $\alpha$ is said to satisfy the {\em zero-sum} condition on the vertex  $\mathbf{u}$. If $\alpha$ satisfies the {\em zero-sum} condition on each  $\mathbf{u}\in V(\mathbb{Z}^2)$, then $\alpha$ is said to satisfy the {\em zero-sum} condition on $\mathcal{C}$. According to the structure of natural  polyocollections, \cite[Theorem 3.12]{ZGW} can be generalized to the case of natural polyocollections.

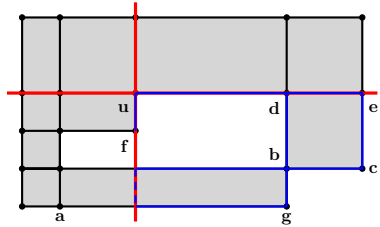
\begin{figure}[h]
\centering
    \begin{tikzpicture}[scale=1, transform shape]
    \draw[fill=black!16,line width=0.8pt]
    (0,0)--(3.5,0)--(3.5,0.5)--(0,0.5)--(0,0)--cycle
    (0,0.5)--(0.5,0.5)--(0.5,1)--(1.5,1)--(1.5,1.5)--(3.5,1.5)-- (3.5,0.5)--(4.5,0.5)--(4.5,2.5)--(0,2.5)--(0,0.5)--cycle;
  \draw[line width=0.8pt] (0,1)--(0.5,1)  (0,1.5)--(4.5,1.5) (0.5,0)--(0.5,2.5)
  (1.5,1)--(1.5, 2.5) (3.5,0)--(3.5,2.5);
  \draw[fill=black] (0,2.5) circle (1pt);
  \draw[fill=black] (0.5,2.5) circle (1pt);
  \draw[fill=black] (1.5,2.5) circle (1pt);
  \draw[fill=black] (3.5,2.5) circle (1pt);
  \draw[fill=black] (4.5,2.5) circle (1pt);

  \draw[fill=black] (0,1.5) circle (1pt);
  \draw[fill=black] (0.5,1.5) circle (1pt);
  \draw[fill=black] (1.5,1.5) circle (1pt);
  \draw[fill=black] (3.5,1.5) circle (1pt);
  \draw[fill=black] (4.5,1.5) circle (1pt);

  \draw[fill=black] (0,1) circle (1pt);
  \draw[fill=black] (0.5,1) circle (1pt);
  \draw[fill=black] (1.5,1) circle (1pt);


  \draw[fill=black] (0,0.5) circle (1pt);
  \draw[fill=black] (0.5,0.5) circle (1pt);
  \draw[fill=black] (3.5,0.5) circle (1pt);
  \draw[fill=black] (4.5,0.5) circle (1pt);

  \draw[fill=black] (0,0) circle (1pt);
  \draw[fill=black] (0.5,0) circle (1pt);
  \draw[fill=black] (3.5,0) circle (1pt);

   \node[left, font=\small,scale=0.6] at (1.5,1.3) {$\mathbf{u}$};
   \node[left, font=\small,scale=0.6] at (1.5,0.8) {$\mathbf{f}$};
   \node[below, font=\small,scale=0.6] at (0.5,0) {$\mathbf{a}$};
   \node[left, font=\small,scale=0.6] at (3.5,0.7) {$\mathbf{b}$};
   \node[right, font=\small,scale=0.6] at (4.5,0.5) {$\mathbf{c}$};
   \node[left, font=\small,scale=0.6] at (3.5,1.3) {$\mathbf{d}$};
   \node[right, font=\small,scale=0.6] at (4.5,1.3) {$\mathbf{e}$};
  \node[below, font=\small,scale=0.6] at (3.5,0) {$\mathbf{g}$};

   {\color{red}  \draw[line width=1.2pt] (1.5,-0.2)--(1.5,2.7) (-0.2,1.5)--(4.7,1.5);}
   {\color{blue}  \draw[line width=0.8pt] (1.5,1.5)--(1.5,1) (3.5,0)--(3.5,1.5)--(4.5,1.5)--(4.5,0.5)--(3.5,0.5) (1.5,0.5)--(3.5,0.5)  (1.5,0)--(3.5,0) (1.5,1.5)--(3.5,1.5);}
{\color{blue} \draw[densely dashed, blue, line width=0.8pt](1.5,0.5)--(1.5,0);}

    \end{tikzpicture}
\vspace{-0.2cm}\caption{An example for $\mathcal{C}^{(-2)}_\mathbf{u}$}\label{fig4}
\end{figure}

\begin{prop}\label{zero sum condition}
Let $\mathcal{C}$ be a natural ployocollection with a labeling $\alpha$. Then $\alpha$ satisfies the zero-sum condition on $\mathcal{C}$ if and only if $\alpha \in  \Lambda_{\mathcal {C}}$.
\end{prop}

\subsection{Switching rook polynomials of polyocollections}

We now introduce the switching rook polynomial of a polyocollection, which is defined similarly to that in \cite[Section 3]{2022qrr}. Let $\mathcal {C}$ be a polyocollection, and let $R_1$ and $R_2$ be two rooks placed on basic intervals $I_1$ and $I_2$ of $\mathcal {C}$, respectively. If $I_1$ and  $I_2$ are in the same row or the same column of $\mathcal {C}$, then  $R_1$ and $R_2$ are said to be in {\it attacking position} (or are called {\it attacking rooks}). Otherwise, $R_1$ and $R_2$ are said to be in {\it non-attacking position} (or are called {\it non-attacking rooks}). The {\it rook number}  $r(\mathcal {C})$ is the maximum number of non-attacking rooks that can be arranged on $\mathcal {C}$. For each $1\leq k\leq r(\mathcal {C})$, a {\it $k$-rook configuration} in $\mathcal {C}$ is a $k$-rook placement (i.e., a set of $k$ rooks) placed on non-attacking positions in  $\mathcal {C}$; let $\mathcal {R}_k(\mathcal {C})$ be the set of all $k$-rook configurations and let $r_k(\mathcal {C}):=|\mathcal {R}_k(\mathcal {C})|$. For $k=0$, set $\mathcal {R}_0(\mathcal {C})=\emptyset$ and $r_0(\mathcal {C})=1$. Moreover, denote $\mathcal {R}(\mathcal {C}):=\bigcup_{k=0}^{r(\mathcal {C})}\mathcal {R}_k(\mathcal {C})$. Two non-attacking rooks $R_1$ and $R_2$ in $\mathcal {C}$ are said to be in {\it switching position} (or are called {\it switching rooks}) if they are arranged on the diagonal or anti-diagonal position of an inner interval of $\mathcal {C}$. Let $F\in \mathcal {R}(\mathcal {C})$ be a $k$-rook configuration containing switching rooks $R_1$ and $R_2$ for some $1\leq k\leq r(\mathcal {C})$, and let $I$ be the inner interval of $\mathcal {C}$ such that $R_1$ and $R_2$ are placed on its diagonal (resp. anti-diagonal) position. Then it is easy to check that
$$(F\setminus\{R_1, R_2\})\cup \{R'_1, R'_2\}$$
is also a $k$-rook configuration, where $R'_1$ and $R'_2$ are rooks placed on the anti-diagonal (resp. diagonal) position of $I$.  This replacements is called a {\it switch} from  $R_1, R_2$ to $R'_1, R'_2$. The switch operation induces an equivalent relation $\overset{\scriptscriptstyle\mathcal{C}}{\thicksim}$ on $\mathcal {R}(\mathcal {C})$: for any $1\leq k \leq r(\mathcal {C})$ and any $F_1, F_2\in \mathcal {R}_k(\mathcal {C})$, $F_1\overset{\scriptscriptstyle\mathcal{C}}{\thicksim} F_2$ is defined to hold if $F_2$ can be obtained from $F_1$ via a sequence of switches: $F_1  \to L_1 \to L_2 \to \cdots \to L_{t-1} \to L_t= F_2$, where $t$ is called the length of the sequence. If $F_1\overset{\scriptscriptstyle\mathcal{C}}{\thicksim} F_2$, then the {\it switching number} from $F_1$ to $F_2$ on $\mathcal {C}$, denoted by $n_{\mathcal {C}}(F_1, F_2)$, is the minimum among all lengths of sequences of switches from $F_1$ to $F_2$.  If $F_1$ and $F_2$ are not equivalent, then denote $F_1\overset{\scriptscriptstyle\mathcal{C}}{\nsim} F_2$. For example, the two $5$-rook configurations in Figure \ref{fig5} are  equivalent under $\overset{\scriptscriptstyle\mathcal{C}}{\thicksim}$ with switching number 1. For each $1\leq k\leq r(\mathcal {C})$, let $\widetilde{\mathcal {R}}_k(\mathcal {C}):=\mathcal {R}_k(\mathcal {C})/\overset{\scriptscriptstyle\mathcal{C}}{\thicksim}$ be the set of equivalence classes of $k$-rook configurations on $\mathcal {C}$,  and let $\widetilde{r}_k(\mathcal {C}):=|\widetilde{\mathcal {R}}_k(\mathcal {C})|$. For $k=0$, set $\widetilde{\mathcal {R}}_0(\mathcal {C})=\emptyset$ and $\widetilde{r}_0(\mathcal {C})=1$. The {\it switching rook polynomial} of $\mathcal {C}$ is defined as
$$\widetilde{r}_{\mathcal {C}}(t):=\sum_{k=0}^{r(\mathcal {C})}\widetilde{{r}}_k(\mathcal {C})t^k.$$

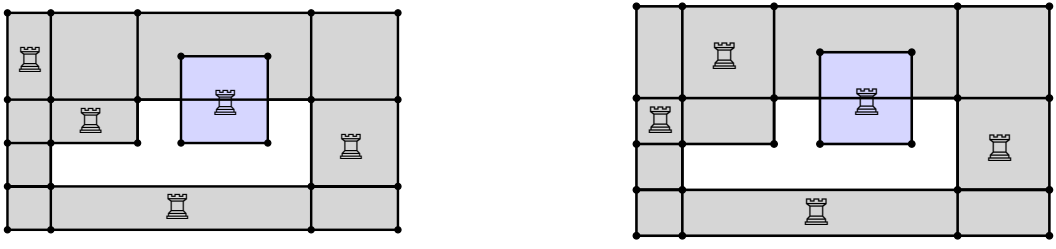
\begin{figure}[htbp]
\centering
\begin{minipage}{0.4\textwidth}
\centering
\resizebox{\linewidth}{!}{%
\begin{tikzpicture}[every node/.style={font=\small}]
      \draw[fill=black!16,line width=0.8pt]
    (0,0)--(4.5,0)--(4.5,0.5)--(0,0.5)--(0,0)--cycle
    (0,0.5)--(0.5,0.5)--(0.5,1)--(1.5,1)--(1.5,1.5)--(3.5,1.5)-- (3.5,0.5)--(4.5,0.5)--(4.5,2.5)--(0,2.5)--(0,0.5)--cycle;
 \draw[fill=blue!16,line width=0.8pt](2,2)--(2,1)--(3,1)--(3,2)--(2,2)--cycle;
  \draw[line width=0.8pt] (0,1)--(0.5,1)  (0,1.5)--(4.5,1.5) (0.5,0)--(0.5,2.5)
  (1.5,1)--(1.5, 2.5) (3.5,0)--(3.5,2.5);
  \draw[fill=black] (0,2.5) circle (1pt);
  \draw[fill=black] (0.5,2.5) circle (1pt);
  \draw[fill=black] (1.5,2.5) circle (1pt);
  \draw[fill=black] (3.5,2.5) circle (1pt);
  \draw[fill=black] (4.5,2.5) circle (1pt);

  \draw[fill=black] (0,1.5) circle (1pt);
  \draw[fill=black] (0.5,1.5) circle (1pt);
  \draw[fill=black] (1.5,1.5) circle (1pt);
  \draw[fill=black] (3.5,1.5) circle (1pt);
  \draw[fill=black] (4.5,1.5) circle (1pt);

  \draw[fill=black] (0,1) circle (1pt);
  \draw[fill=black] (0.5,1) circle (1pt);
  \draw[fill=black] (1.5,1) circle (1pt);
  \draw[fill=black] (2,1) circle (1pt);
  \draw[fill=black] (3,1) circle (1pt);

  \draw[fill=black] (2,2) circle (1pt);
  \draw[fill=black] (3,2) circle (1pt);

  \draw[fill=black] (0,0.5) circle (1pt);
  \draw[fill=black] (0.5,0.5) circle (1pt);
  \draw[fill=black] (3.5,0.5) circle (1pt);
  \draw[fill=black] (4.5,0.5) circle (1pt);

  \draw[fill=black] (0,0) circle (1pt);
  \draw[fill=black] (0.5,0) circle (1pt);
  \draw[fill=black] (3.5,0) circle (1pt);
  \draw[fill=black] (4.5,0) circle (1pt);

\textcolor[rgb]{1.00,1.00,1.00}{   \node[left, font=\small,scale=0.6] at (-0.5,0) {$1$};}

 \coordinate[label=225:$\symrook$]()at(0.55,2.25){};
 \coordinate[label=225:$\symrook$]()at(1.25,1.55){};
  \coordinate[label=225:$\symrook$]()at(2.25,0.55){};
 \coordinate[label=225:$\symrook$]()at(4.25,1.25){};
 \coordinate[label=225:$\symrook$]()at(2.8,1.75){};
\end{tikzpicture}}
\end{minipage}\hfill
\begin{minipage}{0.4\textwidth}
\centering
\resizebox{\linewidth}{!}{%
\begin{tikzpicture}[every node/.style={font=\small}]
      \draw[fill=black!16,line width=0.8pt]
    (0,0)--(4.5,0)--(4.5,0.5)--(0,0.5)--(0,0)--cycle
    (0,0.5)--(0.5,0.5)--(0.5,1)--(1.5,1)--(1.5,1.5)--(3.5,1.5)-- (3.5,0.5)--(4.5,0.5)--(4.5,2.5)--(0,2.5)--(0,0.5)--cycle;
 \draw[fill=blue!16,line width=0.8pt](2,2)--(2,1)--(3,1)--(3,2)--(2,2)--cycle;
  \draw[line width=0.8pt] (0,1)--(0.5,1)  (0,1.5)--(4.5,1.5) (0.5,0)--(0.5,2.5)
  (1.5,1)--(1.5, 2.5) (3.5,0)--(3.5,2.5);
  \draw[fill=black] (0,2.5) circle (1pt);
  \draw[fill=black] (0.5,2.5) circle (1pt);
  \draw[fill=black] (1.5,2.5) circle (1pt);
  \draw[fill=black] (3.5,2.5) circle (1pt);
  \draw[fill=black] (4.5,2.5) circle (1pt);

  \draw[fill=black] (0,1.5) circle (1pt);
  \draw[fill=black] (0.5,1.5) circle (1pt);
  \draw[fill=black] (1.5,1.5) circle (1pt);
  \draw[fill=black] (3.5,1.5) circle (1pt);
  \draw[fill=black] (4.5,1.5) circle (1pt);

  \draw[fill=black] (0,1) circle (1pt);
  \draw[fill=black] (0.5,1) circle (1pt);
  \draw[fill=black] (1.5,1) circle (1pt);
  \draw[fill=black] (2,1) circle (1pt);
  \draw[fill=black] (3,1) circle (1pt);

  \draw[fill=black] (2,2) circle (1pt);
  \draw[fill=black] (3,2) circle (1pt);

  \draw[fill=black] (0,0.5) circle (1pt);
  \draw[fill=black] (0.5,0.5) circle (1pt);
  \draw[fill=black] (3.5,0.5) circle (1pt);
  \draw[fill=black] (4.5,0.5) circle (1pt);

  \draw[fill=black] (0,0) circle (1pt);
  \draw[fill=black] (0.5,0) circle (1pt);
  \draw[fill=black] (3.5,0) circle (1pt);
  \draw[fill=black] (4.5,0) circle (1pt);

 \coordinate[label=225:$\symrook$]()at( 1.25,2.25){};
 \coordinate[label=225:$\symrook$]()at(0.55,1.55){};
  \coordinate[label=225:$\symrook$]()at(2.25,0.55){};
 \coordinate[label=225:$\symrook$]()at(4.25,1.25){};
 \coordinate[label=225:$\symrook$]()at(2.8,1.75){};

\textcolor[rgb]{1.00,1.00,1.00}{\node[left, font=\small,scale=0.6] at (5,0) {$1$};}
\end{tikzpicture}}
\end{minipage}\hfill
\caption{Two equivalent $5$-rook configurations under $\overset{\scriptscriptstyle\mathcal{C}}{\thicksim}$}\label{fig5}
\end{figure}


\section{Primality of natural polyocollections of type  $\mathcal {Q}_1$  and  of type  $\mathcal {Q}_2$ } \label{sec4}


Since every natural polyocollection of type $\mathcal {Q}_2$ has a symmetric counterpart of type $\mathcal {Q}_1$, we only need to examine the latter when investigating their properties.
Let $\mathcal{C}$ be a natural polyocollection,
and let $\al$ be an admissible labeling on $\mathcal{C}$ with $\sigma(\alpha)>0$. The {\it initial horizontal edge interval} of $\al$, denoted by $ini(\al)$, is  the uppermost and leftmost maximal  horizontal edge interval $[\mathbf{a},\mathbf{b}]$ of $\mathcal{C}$ with $\sigma([\mathbf{a},\mathbf{b}]) \neq 0$. In the initial horizontal edge interval of $\al$, there exists some vertex $\mathbf{u}$ with $\al(\mathbf{u}) \neq 0$. Among these vertices with nonzero value, the leftmost one is denoted by $inv(\al)$. Note that $\al$ is an admissible labeling. There exists some vertex $\mathbf{v}$ in $ini(\al)$ such that $\al(\mathbf{v})\al(inv(\al))<0$. Among these $\mathbf{v}$, the leftmost  one is denoted by $inv(\al)^{o}$.

\begin{thm}\label{the primality of coc type O} Let $\mathcal{C}$ be a natural polyocollection of type $\mathcal{Q}_1$ or of type $\mathcal{Q}_2$. Then $\mathcal{C}$ is prime.
\end{thm}

\begin{proof} Let $\mathcal{C}$ be a natural polyocollection of type $\mathcal{Q}_1$. We prove the primality of $\mathcal{C}$ by Proposition \ref{equivalent main idea to prove the primality}. For each $\alpha \in \Lambda_{\mathcal{C}}$ with $\sigma(\alpha)>0$. Assume that  $ini(\al)$ contains $t+1$ vertices $(i_0,i'_0),(i_1,i'_0),\ldots, (i_t,i'_0)$ of $V(\mathcal{C})$, where $t\geq 1$ and $i_0< i_1<\cdots<i_t$. Then we can assume that $inv(\al)=(i_k, i'_0)$ and $inv(\al)^o=(i_j, i'_0)$, where $0 \leq k < j\leq t$. We prove that $\al$ can be effectively reduced by induction on $n=k+j$. Assume without loss of generality that $\al(inv(\al))>0$, it is clear  $\al(inv(\al)^o)<0$. Denote $\mathbf{a}=(i_0, i'_0)$.

If $n=1$, then $inv(\al)=\mathbf{a}=(i_0, i'_0)$ and $inv(\al)^o=(i_1, i'_0)$. Choose a vertex $\mathbf{a}_1$ of $\mathcal{C}$ in the same vertical edge interval with $\mathbf{a}$ such that $\al(\mathbf{a}_1)<0$. We claim that the interval $[\mathbf{a}_1, (i_1, i'_0)]$ is an inner interval. Otherwise, choose the uppermost cell $C_1=[\mathbf{a}_2-(0,1), \mathbf{a}_2+(1,0)]$ in $[\mathbf{a}_1, inv(\al)^o]$ which is not contained in any basic interval of $\mathcal{C}$, see Figure \ref{fig6}. By Proposition \ref{zero sum condition}, one has $\alpha(\mathcal{C}^{(2)}_{\mathbf{a}_2})=0$.   Hence $\mathbf{a}_2$ is not in $ini(\al)$. By the choice of $C_1$,  $[\mathbf{a}_2, (i_1, i'_0)]$ is an inner interval of $\mathcal{C}$. It follows from the definition of natural polyocollection that $\mathbf{a}_2$ is in the vertical edge interval $[\mathbf{a}_1,\mathbf{a}]$. Hence there exists a basic interval $F$ in $\mathcal{C}$ such that $\mathbf{a}_2$ is the upper right corner of $F$. Assume that $F=[\mathbf{a}_2-(p,q), \mathbf{a}_2]$. Since $\mathcal{C}$ is a natural polyocollection of type $\mathcal{Q}_1$, $[\mathbf{a}_2-(p,0), \mathbf{a}]$ is an inner interval of $\mathcal{C}$. Consequently, $[\mathbf{a}-(p,0), (i_1, i'_0)]$ is a horizontal edge interval of $\mathcal{C}$. This contradicts the fact that $\mathbf{a}$ is an end vertex of the maximal horizontal edge interval $ini(\al)$ of $\mathcal{C}$. Hence $[\mathbf{a}_1, (i_1, i'_0)]$ is an inner interval. Denote by $\beta = \al + \mu_{[\mathbf{a}_1, (i_1, i'_0)]}$. It is direct to check that $\al$ can be effectively reduced to $\beta$.

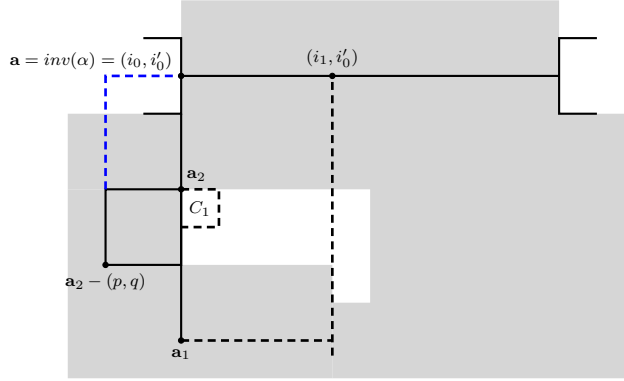
\begin{figure}[h]
\centering
    \begin{tikzpicture}[scale=1, transform shape]
     \fill[black!16]
    (1,6)--(1,4.5)--(-0.5,4.5)--(-0.5,3.5)--(3.5,3.5)--(3.5,2)--(3,2)--(3,1)--(7,1)--(7,4.5)--(6,4.5)--(6,6)--(1,6)--cycle;
     \fill[black!16]
    (-0.5,3.5)--(-0.5,1)--(3,1)--(3,2.5)--(1,2.5)--(1,3.5)--(-0.5,3.5)--cycle;

  \draw[line width=0.8pt]  (0.5,5.5)--(1,5.5)--(1,4.5)--(0.5,4.5)  (1,5)--(6,5)  (6.5,5.5)--(6,5.5)--(6,4.5)--(6.5,4.5);
  \draw[line width=0.8pt]  (1,4.5)--(1,1.5);

   \draw[line width=0.8pt]  (0,2.5)--(0,3.5)--(1,3.5)--(1,2.5)--(0,2.5);

  \draw[densely dashed, black,  line width=1pt]  (1,3.5)--(1.5,3.5)--(1.5,3)--(1,3);
   \draw[densely dashed, blue,  line width=1pt]  (0,3.5)--(0,5)--(1,5);
    \draw[densely dashed, black,  line width=1pt]   (3,1.3)--(3,5) (1,1.5)--(3,1.5) ;

   \draw[fill=black] (1,5) circle (1pt);
   \node[left, font=\tiny,scale=0.8] at (1,5.2) {$\mathbf{a}=inv(\al)=(i_0,i'_0)$};
   \node[above, font=\tiny,scale=0.8] at (3,5) {$(i_1,i'_0)$};
    \draw[fill=black] (3,5) circle (1pt);
       \draw[fill=black] (1,1.5) circle (1pt);
     \node[below, font=\tiny,scale=0.8] at (1,1.5) {$\mathbf{a}_1$};
       \draw[fill=black] (1,3.5) circle (1pt);
     \node[above, font=\tiny,scale=0.8] at (1.2,3.5) {$\mathbf{a}_2$};
      \node[below, font=\tiny,scale=0.8] at (1.25,3.45) {$C_1$};
       \draw[fill=black] (0,2.5) circle (1pt);
     \node[below, font=\tiny,scale=0.8] at (0,2.5) {$\mathbf{a}_2-(p,q)$};
    \end{tikzpicture}
\caption{The illustration for $n=1$}\label{fig6}
\end{figure}

Assume that $\al$ can be effectively reduced when $n<n_0$. Now, we consider the case when $n=n_0$.
We choose a vertex $\mathbf{a}_1$ of $\mathcal{C}$ in the same vertical edge interval with $\mathbf{a}$ such that $\al(\mathbf{a}_1)<0$. If the interval $[\mathbf{a}_1, inv(\al)^o]$ is an inner interval, then one can prove that $\al$ can be effectively reduced by the same way with the above discussion. Otherwise, if $[\mathbf{a}_1, inv(\al)^o]$ is not an inner interval, then choose the uppermost and leftmost cell $[\mathbf{a}_2-(0,1), \mathbf{a}_2+(1,0)]$ in $[\mathbf{a}_1, inv(\al)^o]$ which is not contained in any basic interval of $\mathcal{C}$, see Figure \ref{fig7}. Assume that $\mathbf{a}_2=(s_1, s_2)$. By a similar discussion as above, we have $[\mathbf{a}_2, inv(\al)^o]$ is an inner interval of $\mathcal{C}$.
Choose the smallest $l_1$ such that $[(l_1, l'_1), (s_1, s_2)]$ is an inner interval of $\mathcal{C}$ for some $l'_1$. It is easy to check that $i_0 \leq l_1 \leq i_{k}\leq s_1 < i_{j}$. Since $\mathcal{C}$ is a natural polyocollection of type $\mathcal{Q}_1$, $[(l_1,l'_1), (s_1,i'_0)]$ is an inner interval of $\mathcal{C}$.
By the choice of $l_1$, the cell $[(l_1-1,s_2-1),(l_1,s_2)]$ is not contained in any basic interval of $\mathcal{C}$. Note that $\alpha \in \Lambda_{\mathcal{C}}$, it follows from Proposition \ref{zero sum condition} that $\alpha$ satisfies the zero-sum condition on $\mathcal{C}$. Hence $\al(\mathcal{C}^{(2)}_{(s_1, s_2)})=0$ and $\al(\mathcal{C}^{(2)}_{(l_1-1, s_2)})=0$. It is easy to check that $\al(\mathcal{C}^{(2)}_{(s_1, s_2)} \cap \mathcal{C}^{(1)}_{(l_1, s_2)})=0$. Note that $inv(\al) \in V(\mathcal{C}^{(2)}_{(s_1, s_2)} \cap \mathcal{C}^{(1)}_{(l_1, s_2)})$, so there exists a vertex $\mathbf{u}=(c_1, c_2) \in [(l_1, s_2), (s_1, i'_0)]\cap V(\mathcal{C})$ such that $\al(u)<0$. By the definition of $inv(\al)^o$, $c_2< i'_0$ and $\al((c_1, i'_0))\geq 0$.
Denote by $\beta = \al + \mu_{[\mathbf{u}, inv(\al)^o]}$. If $\al((c_1, i'_0)) > 0$, then it is direct to check that $\al$ can be effectively reduced to $\beta$. If $\al((c_1, i'_0)) = 0$, then $\al$ can be reduced to $\beta$, and $ini(\beta)=ini(\al)$. Assume that $inv(\beta)=(i_{k_0}, i'_0)$ and $inv(\beta)^o=(i_{j_0}, i'_0)$, then $k_0 \leq k$ and $j_0 < j$. It follows that $k_0+j_0< k+j =n=n_0$. By induction hypothesis, $\beta$ can be effectively reduced. Hence $\al$ can be effectively reduced.

It follows from Proposition \ref{equivalent main idea to prove the primality} that $\mathcal{C}$ is prime.

By a similar way, one can prove that a natural polyocollection of type $\mathcal{Q}_2$ is prime.
\end{proof}

\begin{figure}[h]
\centering
    \begin{tikzpicture}[scale=0.7, transform shape]
     \fill[black!16]
    (-5,2.2)--(-5,0)--(5,0)--(5,1)--(6,1)--(6,2.2)--(-5,2.2)--cycle;
  \fill[black!16]
 (-6,-3)--(-6,-1)--(-5,-1)--(-5,-0)--(5,0)--(5,-1)--(6,-1)--(6,-6.5)--(4,-6.5)--(4,-3)--(2,-3)--(2,-5)--(1,-5)--(1,-6.5)--(-6,-6.5)--(-6,-4)--(-4,-4)--(-4,-3)--(-6,-3)--cycle;

  \draw[line width=0.8pt]  (-6,1)--(-5,1)--(-5,-1)--(-6,-1);
  \draw[line width=0.8pt](-5,0)--(5,0);
  \draw[line width=0.8pt](6,1)--(5,1)--(5,-1)--(6,-1);

  \draw[line width=0.8pt]  (-5,-3)--(-4,-3)--(-4,-4)--(-5,-4);
   \draw[line width=0.8pt]  (1,-5.5)--(1,-5)--(2,-5)--(2,-3)--(4,-3);

   \draw[densely dashed, blue, line width=1pt]  (-4,-3)--(-4,2) (2,-3)--(2,2) (-4,-3)--(2,-3);
    \draw[densely dashed, line width=0.8pt] (-4,-4)--(2,-4);
   \draw[densely dashed, red, line width=1pt]  (-3,0)--(-3,-2)--(4,-2)--(4,0);
    \draw[densely dashed, line width=0.8pt] (-2,0)--(-2,-6)--(4,-6)--(4,0);
   \draw[densely dashed, green, line width=1pt](0,0)--(0,-1)--(4,-1)--(4,0);
    \draw[densely dashed, line width=0.8pt] (2.5,-3)--(2.5,-3.5)--(2,-3.5);

     \draw[fill=black] (-5,0) circle (1.2pt);
   \node[left, font=\small,scale=0.8] at (-5,0) {$\mathbf{a}=(i_0,i'_0)$};
     \draw[fill=black] (-2,0) circle (1.2pt);
   \node[above, font=\small,scale=0.8] at (-2,0) {$inv(\al)=(i_k,i'_0)$};
    \draw[fill=black] (4,0) circle (1.2pt);
     \node[above, font=\small,scale=0.8] at (3.5,0) {$inv(\al)^o=(i_j,i'_0)$};
     \draw[fill=black] (5,0) circle (1.2pt);
          \draw[fill=black] (0,-1) circle (1.2pt);
     \node[below, font=\small,scale=0.8] at (0,-1) {$\mathbf{u}$};
      \draw[fill=black] (-3,-2) circle (1.2pt);
     \node[below, font=\small,scale=0.8] at (-3,-2) {$\mathbf{u}$};
        \draw[fill=black] (-4,-3) circle (1.2pt);
     \node[left, font=\small,scale=0.8] at (-4,-3.3) {$(l_1,s_2)$};
   \draw[fill=black] (-4,-4) circle (1.2pt);
     \node[right, font=\small,scale=0.8] at (-4,-4.3) {$(l_1,l'_1)$};
    \draw[fill=black] (2,-3) circle (1.2pt);
     \node[right, font=\small,scale=0.8] at (2,-2.8) {$\mathbf{a}_2=(s_1,s_2)$};
    \draw[fill=black] (-2,-6) circle (1.2pt);
     \node[below, font=\small,scale=0.8] at (-2,-6) {$\mathbf{a}_1$};

      \node[below, font=\tiny,scale=1] at (2.25,-3) {$C_1$};
 \node[below, font=\small,scale=1] at (-1,1.8) {$\mathcal{C}_{(s_1,s_2)}^{(2)}\cap \mathcal{C}_{(l_1,s_2)}^{(1)}$};
    \end{tikzpicture}
\caption{The illustration for $n=n_0$}\label{fig7}
\end{figure}
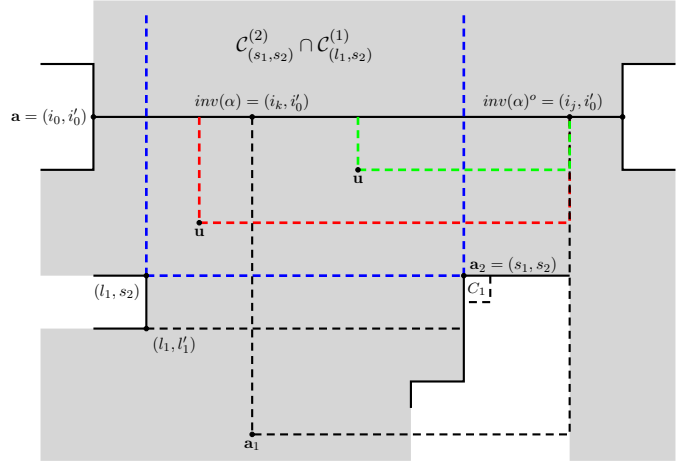


Note that there exists no zig-zag walk in a natural polyocollection of type $\mathcal{Q}_1$ or of type $\mathcal{Q}_2$, so it follows from Theorem \ref{the primality of coc type O} that the zig-zag conjecture holds true for collections of cells of type $\mathcal{Q}_1$ or of type $\mathcal{Q}_2$.
Since some well-studied collections of cells are natural polyocollections of type $\mathcal{Q}_1$ or of type $\mathcal{Q}_2$, Theorem \ref{the primality of coc type O} directly yields the following corollary, which recovers the primality of these collections of cells in a unified way.

\begin{cor}\label{known prime collection of cells}The following collections of cells are prime$:$\\[-0.4cm]
\begin{itemize}
\item[(1)] Polyominoes with the shape ``a rectangle minus a rectangle" (\cite{2015s})$;$\\[-0.4cm]

\item[(2)]Polyominoes obtained by removing a convex polyomino from its ambient rectangle (\cite{2015hq1,2018s})$;$\\[-0.4cm]

\item[(3)] Grid polyominoes (\cite{2020mrr})$;$\\[-0.4cm]

\item[(4)] Subgrid polyominoes (\cite{2022mrr})$;$\\[-0.4cm]

\item[(5)] Thin cycle polyominoes whose maximal inner intervals have length at least 3 (\cite{2022mrr})$;$\\[-0.4cm]

\item[(6)]Bipartite collections of cells (\cite{ZGW2}).\\[-0.4cm]
\end{itemize}
\end{cor}

Let $\mathcal {C}$ be a polyocollection. One can define an associated simple graph, denoted by  $G(\mathcal {C})$, whose vertex set is $V(\mathcal {C})$ and edge set is  $\big\{\{\mathbf{a}, \mathbf{b}\}\mid [\mathbf{a}, \mathbf{b}] \text{ is an inner interval of } \mathcal {C}\big\}$.

\begin{remark}\label{graph}
Let $\mathcal {C}$ be a natural polyocollection of type $\mathcal{Q}_1$, and let $<_{\text{\em lex}}$ be the monomial order on $S_{\mathcal {C}}$ defined in Section \ref{sec1}. From Theorem \ref{condition} it follows that
 $$\text{\em in}_{<_{\text{\em lex}}}(I_{\mathcal {C}})=\big(x_{\mathbf{a}}x_{\mathbf{b}}\mid [\mathbf{a}, \mathbf{b}] \in \mathcal{I}(\mathcal {C})\big).$$
Consequently, $\text{\em in}_{<_{\text{\em lex}}}(I_{\mathcal {C}})$ is exactly the edge ideal of $G(\mathcal {C})$.
 \end{remark}

\begin{cor} \label{cor1}
Let $\mathcal {C}$ be a natural polyocollection of type $\mathcal{Q}_1$ or of type $\mathcal{Q}_2$. Then
$\mathbb{K}[\mathcal {C}]$ is a Cohen-Macaulay domain.
\end{cor}
\begin{proof}
It follows from Remark \ref{graph}, \cite{1995s} and \cite[Theorem 6.3.5]{bh}.
\end{proof}

\section{The algebraic invariants of natural polyocollections of type $\mathcal{Q}_1$ and of type $\mathcal{Q}_2$}\label{sec5}

The following result, together with Corollary \ref{cor1}, implies that (2) of Main Theorem is true.

\begin{prop} \label{dim-depth}
Let $\mathcal {C}$ be a natural polyocollection of type $\mathcal{Q}_1$ or of type $\mathcal{Q}_2$.  Then \\[-0.2cm]
$$\text{\em dim}(\mathbb{K}[\mathcal {C}])=\text{\em depth}(\mathbb{K}[\mathcal {C}])=|V(\mathcal {C})|-|\mathcal {C}|.$$
\end{prop}

\begin{proof}
We  need only consider the case that  $\mathcal {C}$ is a natural polyocollection of type $\mathcal{Q}_1$. Since $\mathbb{K}[\mathcal {C}]$ is Cohen-Macaulay by Corollary \ref{cor1}, then $\text{dim}(\mathbb{K}[\mathcal {C}])=\text{depth}(\mathbb{K}[\mathcal {C}])$. Let  $<_{\text{lex}}$ be the monomial order on $S_{\mathcal {C}}$ defined in Section \ref{sec1}. Since $\text{in}_{<_{\text{lex}}}(I_{\mathcal {C}})$ is square-free by Remark \ref{graph}, we have
\[\text{dim}(S_{\mathcal {C}}/\text{in}_{<_{\text{lex}}}(I_{\mathcal {C}}))= \text{depth}(S_{\mathcal {C}}/\text{in}_{<_{\text{lex}}}(I_{\mathcal {C}})) \tag{$\ast$} \label{formula1}\]
by Lemma \ref{intial} and the equality  $\text{dim}(\mathbb{K}[\mathcal {C}])=\text{depth}(\mathbb{K}[\mathcal {C}])$. Note that $\text{in}_{<_{\text{lex}}}(I_{\mathcal {C}})$ can be viewed as the edge ideal of the graph $G(\mathcal {C})$.
Formula (\ref{formula1}) yields that $G(\mathcal {C})$ is a Cohen-Macaulay graph. Thus all maximal independent sets of $G(\mathcal {C})$ have the same cardinality, which equals $\text{dim}(S_{\mathcal {C}}/\text{in}_{<_{\text{lex}}}(I_{\mathcal {C}}))$. Let $T$ be the set of all lower left corners of basic intervals in $\mathcal {C}$. We claim that $V(\mathcal {C})\setminus T$ is a maximal independent set of $G(\mathcal {C})$. Based on this claim, we obtain that
$$\text{dim}(S_{\mathcal {C}}/\text{in}_{<_{\text{lex}}}(I_{\mathcal {C}}))=|V(\mathcal {C})\setminus T|=|V(\mathcal {C})|-|\mathcal {C}|,$$
as desired.

We now establish the claim. First, we show that $V(\mathcal {C})\setminus T$ is an independent set of $G(\mathcal {C})$. Suppose that there exist two vertices $\mathbf{a},\mathbf{b}\in V(\mathcal {C})\setminus T$ such that $\{\mathbf{a},\mathbf{b}\}$ is an edge of $G(\mathcal {C})$. Without loss of generality, assume that $[\mathbf{a},\mathbf{b}]$ is an inner interval of $\mathcal {C}$, then there exists a basic interval $[\mathbf{a}, \mathbf{b}']\subseteq [\mathbf{a},\mathbf{b}]$. Hence  $\mathbf{a}\in T$, a contradiction.

Next, we prove the maximality of $V(\mathcal {C})\setminus T$. Suppose that there exists some $\mathbf{a}=(i,j)\in T$ such that $\{\mathbf{a},\mathbf{v}\}$ is not an edge of  $G(\mathcal {C})$ for any $\mathbf{v}\in V(\mathcal {C})\setminus T$. Let $\mathbf{b}=(k,j)$,
 where $k\!=\!\max\{ p\mid \text{$[(i,j),(p,j')]\in  \mathcal {I}(\mathcal {C})$ for some $p>i$ and $j'>j$} \}$. Here the  existence of $\mathbf{b}$ comes from the fact that $\mathbf{a}$ is a lower left corner of  some basic interval of $\mathcal {C}$. See Figure \ref{fig8}. Set
$\ell=\max\{ q\mid \text{$[\mathbf{a},(k, q)]\in  \mathcal {I}(\mathcal {C})$}\}$.
By the definition of $G(\mathcal {C})$, $\{\mathbf{a},\mathbf{d}\}$ is an edge of $G(\mathcal {C})$, where $\mathbf{d}=(k,\ell)$. This yields that $\mathbf{d}\in T$. Consequently, there exists a basic interval  $[\mathbf{d},\mathbf{d}+(r,s)]$ in $\mathcal {C}$. Applying Theorem \ref{condition} to the inner intervals  $[\mathbf{a}, \mathbf{d}]$ and $[\mathbf{d},\mathbf{d}+(r,s)]$, it follows from the maximality of $k$ that $[\mathbf{a},\mathbf{d}+(0,s)]$ is an inner interval of $\mathcal {C}$, contradicting the choice of $\ell$.
\end{proof}

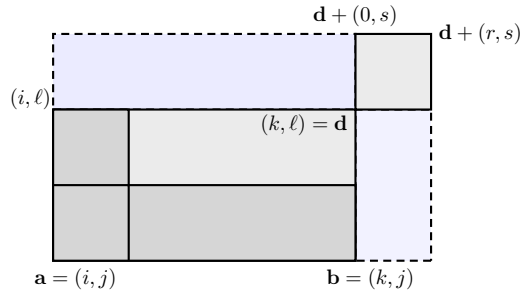
\begin{figure}[h]
\centering
    \begin{tikzpicture}
    \tikzstyle{every node}=[font=\small,scale=0.9]
    \coordinate[](a1)at(0,0){};\coordinate[](a2)at(1,0){};\coordinate[](a3)at(2,0){};
    \coordinate[](a4)at(3,0){};\coordinate[](a5)at(4,0){};\coordinate[](a6)at(5,0){};
    \coordinate[](b1)at(0,1){};\coordinate[](b2)at(1,1){};\coordinate[](b3)at(2,1){};
    \coordinate[](b4)at(3,1){};\coordinate[](b5)at(4,1){};\coordinate[](b6)at(5,1){};
    \coordinate[](c1)at(0,2){};\coordinate[](c2)at(1,2){};\coordinate[](c3)at(2,2){};
    \coordinate[](c4)at(3,2){};\coordinate[](c5)at(4,2){};\coordinate[](c6)at(5,2){};
    \coordinate[](g1)at(0,3){};\coordinate[](g2)at(1,3){};\coordinate[](g3)at(2,3){};
    \coordinate[](g4)at(3,3){};\coordinate[](g5)at(4,3){};\coordinate[](g6)at(5,3){};
    \draw[densely dashed, fill=blue!8,line width=0.8pt]
    (c1)--(c5)--(g5)--(g1)--(c1)--cycle;
    \draw[densely dashed, fill=blue!5,line width=0.8pt]
    (a6)--(a5)--(c5)--(c6)--(a6)--cycle;

    \draw[fill=black!8,line width=0.8pt]
    (a1)--(a5)--(c5)--(c1)--(a1)--cycle;
   \draw[fill=black!8,line width=0.8pt]
    (g6)--(g5)--(c5)--(c6)--(g6)--cycle;

    \draw[fill=black!16,line width=0.8pt]
    (a1)--(a5)--(b5)--(b2)--(c2)--(c1)--(a1)--cycle;
    \draw[line width=0.8pt](b1)--(b2) (b2)--(a2);
    \node[below, font=\small,scale=0.8] at (0.3,0) {$\mathbf{a}=(i,j)$};
   \node[below, font=\small,scale=0.8] at (4.2,0) {$\mathbf{b}=(k,j)$};
    \node[above, font=\small,scale=0.8] at (-0.3,1.9) {$(i,\ell)$};
   \node[left, font=\small,scale=0.8] at (4,1.8) {$(k,\ell)=\mathbf{d}$};
    \node[above, font=\small,scale=0.8] at (4,3) {$\mathbf{d}+(0,s)$};
     \node[right, font=\small,scale=0.8] at (5,3) {$\mathbf{d}+(r,s)$};
    \end{tikzpicture}
\vspace{-0.1cm}\caption{An illustration of the proof of proposition \ref{dim-depth}}\label{fig8}
\end{figure}

Let $\mathcal {C}$ be a natural polyocollection. For $U\subseteq V(C)$, denote by $\mathcal {C}\setminus U$ the collection of proper intervals obtained from $\mathcal {C}$ by removing all basic intervals $I$ with $V(I) \cap U \neq \emptyset$. In particular, if $U=\{\mathbf{v}\}$ for some $\mathbf{v}\in V(\mathcal{C})$, we write $\mathcal {C}\setminus \mathbf{v}$ for $\mathcal {C}\setminus \{\mathbf{v}\}$. One easily observes that $\mathcal {C}\setminus U$ is also a natural polyocollection and $\mathcal {C}\setminus U\subseteq \mathcal {C}$.   A vertex $\mathbf{u}$ is called a {\it single vertex} of $\mathcal {C}$  if there exists exactly one basic interval in  $\mathcal {C}$ containing $\mathbf{u}$.
For each vertex $\mathbf{v} \in V(\mathcal {C})$, one easily sees that $V(\mathcal {C}\setminus \mathbf{v})\subseteq V(\mathcal {C})\setminus \{\mathbf{v}\}$, and the equality holds if and only if there is no single vertex $\mathbf{w}$  such that $\mathbf{v}$ and $\mathbf{w}$ are two vertices of a common basic interval of $\mathcal {C}$. It is easy to observe the following results.

\begin{ob} \label{ob-1}
If $\mathbf{v}\in V(\mathcal {C})$ is an upper right or lower left single vertex of $\mathcal {C}$, then, for any $\mathbf{a},\mathbf{b}\in V(\mathcal {C}\setminus \mathbf{v})$,
$[\mathbf{a},\mathbf{b}]\in \mathcal {I}(\mathcal {C}\setminus \mathbf{v})$ if and only if $[\mathbf{a},\mathbf{b}]\in \mathcal {I}(\mathcal {C})$.
\end{ob}

Next, we will prove part (3) of Main Theorem by using the method of \cite{2025nrr}. Specifically, let $\mathcal {C}$ be a natural polyocollection of type $\mathcal{Q}_1$ and $\mathbf{v}\in V(\mathcal {C})$  the uppermost and rightmost vertex, then by Lemmas \ref{intial} and  \ref{exact},
$$H_{\mathbb{K}[\mathcal {C}]}(t)=H_{S_{\mathcal {C}}/\text{in}_{<_{\text{lex}}}(I_{\mathcal {C}})}(t)=t\cdot H_{S_{\mathcal {C}}/(\text{in}_{<_{\text{lex}}}(I_{\mathcal {C}}):x_{\mathbf{v}})}(t)+H_{S_{\mathcal {C}}/(\text{in}_{<_{\text{lex}}}(I_{\mathcal {C}}),x_{\mathbf{v}})}(t).$$
Hence, in order to compute $H_{\mathbb{K}[\mathcal {C}]}(t)$, we should study some properties of $\text{in}_{<_{\text{lex}}}(I_{\mathcal {C}}):x_{\mathbf{v}}$ and $(\text{in}_{<_{\text{lex}}}(I_{\mathcal {C}}),x_{\mathbf{v}})$. To this end, we make the following setup. The reader may refer to Figure \ref{fig9} for illustration.

\begin{setup}\label{setup}
Let $\mathcal {C}$ be a natural polyocollection, and let $\mathbf{v}$ be the uppermost and rightmost vertex of $\mathcal{C}$. Suppose that $\mathcal{C}$ contains $m-1$ basic intervals in the same row as $\overline{I_{\mathbf{v}}}(\mathcal {C})$ (including $\overline{I_{\mathbf{v}}}(\mathcal {C})$) with upper left corners (from left to right) $\mathbf{u}_{1,n}, \ldots, \mathbf{u}_{m-1,n}$, and suppose that $\mathcal{C}$ contains $n-1$ basic intervals in the same column as $\overline{I_{\mathbf{v}}}(\mathcal {C})$ (including $\overline{I_{\mathbf{v}}}(\mathcal {C})$) with lower right corners (from bottom to top) $\mathbf{u}_{m,1}, \ldots, \mathbf{u}_{m,n-1}$. Let $\mathbf{v} = \mathbf{u}_{m,n}$. For any $1\leq i<m$ and any $1\leq j < n$, let $\mathbf{u}_{i,j}$ be the lower left corner of the proper interval with $\mathbf{u}_{i,n}$ and $\mathbf{u}_{m,j}$ being its anti-diagonal corners. Define $$U_{\mathbf{v}}(\mathcal {C})=\big\{\mathbf{u}\mid [\mathbf{u},\mathbf{v}]\in \mathcal {I}(\mathcal {C})\big\}.$$
Assume that $\mathbf{u}_{k_1, \ell_1},\ldots, \mathbf{u}_{k_s, \ell_s}$ are all vertices of $\mathcal {C}$ such that each $[\mathbf{u}_{k_i, \ell_i}, \mathbf{v}]$ is a maximal inner interval of  $\mathcal {C}$.  Without  loss of generality, we also assume that $1=k_1<\cdots<k_s<m$. Then, by the  maximality of each $[\mathbf{u}_{k_i, \ell_i}, \mathbf{v}]$, we obtain that  $n>\ell_1>\cdots>\ell_s=1$. Set $\ell_0=n$ and $k_{s+1}=m$. Under such assumptions, we have
$$U_{\mathbf{v}}(\mathcal {C})=\{\mathbf{u}_{p, q}\mid k_i \leq p<m \text{ and } \ell_i \leq q<n \text{ for some } i=1,2,\ldots,s\}.$$
We define
\[\mathcal {C}_{\mathbf{v}}=\big(\mathcal {C}\setminus U_{\mathbf{v}}(\mathcal {C})\big)\cup \mathcal {T}_1\cup \mathcal {T}_2, \tag{$\ast\ast$} \label{formula2}\]
where
$\mathcal {T}_1 = \{ [llc(\overline{I_{\mathbf{u}_{k_i,q}}}(\mathcal {C})),\mathbf{u}_{m,q}]\mid$  $\mathbf{u}_{k_i,q}$ is the upper right corner of a basic interval of $\mathcal {C}$ for some $\ell_i<q<\ell_{i-1}$ and $1\leq i\leq s\}$, and
$\mathcal {T}_2 = \{[llc(\overline{I_{\mathbf{u}_{p,\ell_j}}}(\mathcal {C})),\mathbf{u}_{p,n}] \mid$ $\mathbf{u}_{p,\ell_j}$ is the upper right corner of a basic interval of $\mathcal {C}$ for some $k_j<p<k_{j+1}$ and $1\leq j\leq s\}$.
Moreover, we define
\[\mathcal {C}'_{\mathbf{v}}=\big(\mathcal {C}\setminus U_{\mathbf{v}}(\mathcal {C})\big)\cup \mathcal {T}_1\cup \mathcal {T}'_2, \tag{$\ddagger$} \label{formula3}\]
where $\mathcal {T}'_2=\big\{[llc(\overline{I_{\mathbf{u}_{p,\ell_j}}}(\mathcal {C})),\mathbf{u}_{p,\ell_j}-(0,1/2)] \mid$ $\mathbf{u}_{p,\ell_j}$ is the upper right corner of a basic interval of $\mathcal {C}$ for some $k_j<p<k_{j+1}$ and $1\leq j\leq s \big\}$.
In other words, $\mathcal {C}'_{\mathbf{v}}$ can be obtained from $\mathcal {C}_{\mathbf{v}}$ by replacing  each basic interval in $\mathcal {T}_2$ with a proper  interval in $\mathbb{R}^2$ that shares the same lower left corner and lower right corner but smaller ``height".
\end{setup}

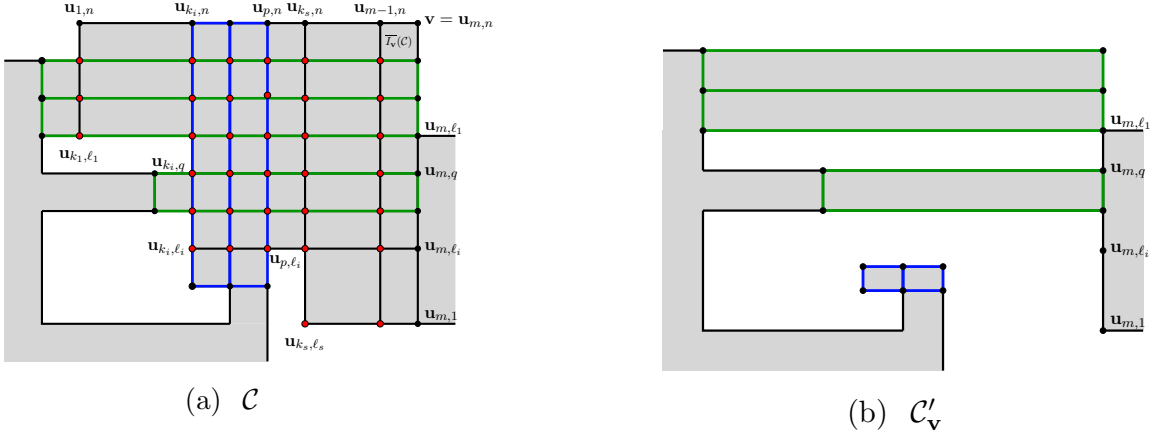
\begin{figure}[htbp]
\centering
\begin{minipage}{0.43\textwidth}
\centering
\resizebox{\linewidth}{!}{%
\begin{tikzpicture}[every node/.style={font=\small}]
  \fill[black!16] (-0.5,4)--(-0.5,3)--(0.5,3)--(0.5,4)--(-0.5,4)--cycle;
   \fill[black!16] (0,2.5)--(1.5,2.5)--(1.5,2)--(0,2)--(0,2.5)--cycle
   (2.5,0.5)--(2.5,1)--(3,1)--(3,0.5)--(2.5,0.5)--cycle
   (5,0.5)--(5.5,0.5)--(5.5,3)--(5,3)--(5,0.5)--cycle
   (2.5,1.5)--(3,1.5)--(3,1)--(2.5,1)--(2.5,1.5)--cycle ;
    \fill[black!16] (-0.5,3)--(0,3)--(0,0)--(-0.5,0)--(-0.5,3)--cycle;
    \fill[black!16] (0,0)--(3,0)--(3,0.5)--(0,0.5)--(0,0)--cycle;
   \fill[black!16] (2,1.5)--(2,1)--(2.5,1)--(2.5,1.5)--(2,1.5)--cycle;

  \draw[fill=black!16,line width=0.8pt]
  (0.5,4.5)--(0.5,3)--(2,3)--(2,1.5)--(3.5,1.5)--(3.5,0.5)--(5,0.5)--(5,4.5)--(0.5,4.5)--cycle
  (1.5,2.5)--(1.5,2)--(5,2)--(5,2.5)--(1.5,2.5)--cycle;

 \draw[line width=0.8pt] (0,3)--(0.5,3) (0.5,4)--(0,4) (2,3)--(2,4.5) (3.5,1.5)--(3.5,4.5) (2,3)--(5,3) (3.5,1.5)--(5,1.5)
   (2.5,0.5)--(2.5,1) (3,1)--(3,0.5) (5,0.5)--(5.5,0.5) (5.5,3)--(5,3)  (0,2.5)--(1.5,2.5) (1.5,2)--(0,2)  (4.5,4.5)--(4.5,4)--(5,4)
   (0.5,4)--(4.5,4)--(4.5,0.5)      (-0.5,4)--(0,4)  (0,3)--(0,2.5) (0,2)--(0,0.5)--(2.5,0.5) (3,0.5)--(3,0) (2,1.5)--(2,1)--(2.5,1)--(2.5,1.5);

\textcolor[rgb]{0.00,0.59,0.00}{ \draw[line width=1pt]
 (1.5,2.5)--(1.5,2)--(5,2)--(5,2.5)--(1.5,2.5);
   \draw[line width=1pt]
 (0,4)--(0,3)--(5,3)--(5,4)--(0,4);
    \draw[line width=1pt]
 (0,3.5)--(5,3.5);}

\textcolor[rgb]{0.00,0.07,1.00}{ \draw[line width=1pt]
(2.5,1)--(3,1)--(3,4.5)--(2.5,4.5)--(2.5,1);
\draw[line width=1pt]
(2,1)--(2.5,1)--(2.5,4.5)--(2,4.5)--(2,1);
}

  \node[above, font=\small,scale=0.6] at (0.5,4.5) {$\mathbf{u}_{1,n}$};
  \node[above, font=\small,scale=0.6] at (2,4.5) {$\mathbf{u}_{k_i,n}$};
  \node[above, font=\small,scale=0.6] at (3.5,4.5) {$\mathbf{u}_{k_s,n}$};
  \node[above, font=\small,scale=0.6] at (4.5,4.5) {$\mathbf{u}_{m-1,n}$};
  \node[right, font=\small,scale=0.6] at (5,4.5) {$\mathbf{v}=\mathbf{u}_{m,n}$};

  \node[right, font=\small,scale=0.6] at (5,3.1) {$\mathbf{u}_{m,\ell_1}$};
  \node[right, font=\small,scale=0.6] at (5,1.5) {$\mathbf{u}_{m,\ell_i}$};
  \node[right, font=\small,scale=0.6] at (5,0.6) {$\mathbf{u}_{m,1}$};

  \node[below, font=\small,scale=0.6] at (0.5,2.9) {$\mathbf{u}_{k_1,\ell_1}$};
  \node[left, font=\small,scale=0.6] at (2,1.5) {$\mathbf{u}_{k_i,\ell_i}$};
  \node[below, font=\small,scale=0.6] at (3.5,0.4) {$\mathbf{u}_{k_s,\ell_s}$};

  \node[below, font=\small,scale=0.4] at (4.75,4.4) {$\overline{I_{\mathbf{v}}}(\mathcal {C})$};

 \node[right, font=\small,scale=0.6] at (5,2.5) {$\mathbf{u}_{m,q}$};
 \node[left, font=\small,scale=0.6] at (2,2.62) {$\mathbf{u}_{k_i,q}$};
 \node[above, font=\small,scale=0.6] at (3,4.5) {$\mathbf{u}_{p,n}$};
 \node[below, font=\small,scale=0.6] at (3.24,1.5) {$\mathbf{u}_{p,\ell_i}$};

 \draw[fill=black] (0.5,4.5) circle (1pt);
  \draw[fill=black] (2,4.5) circle (1pt);
  \draw[fill=black] (2.5,4.5) circle (1pt);
  \draw[fill=black] (3,4.5) circle (1pt);
  \draw[fill=black] (3.5,4.5) circle (1pt);
  \draw[fill=black] (4.5,4.5) circle (1pt);
  \draw[fill=black] (5,4.5) circle (1pt);
  \draw[fill=black] (5,4) circle (1pt);
\draw[fill=black] (5,3.5) circle (1pt);
  \draw[fill=black] (5,3) circle (1pt);
  \draw[fill=black] (5,2.5) circle (1pt);
  \draw[fill=black] (5,2) circle (1pt);
  \draw[fill=black] (5,1.5) circle (1pt);
  \draw[fill=black] (5,0.5) circle (1pt);

 \draw[fill=black] (0,4) circle (1.2pt);
   \draw[fill=red] (0.5,4) circle (1.2pt);
  \draw[fill=red] (2,4) circle (1.2pt);
   \draw[fill=red] (2.5,4) circle (1.2pt);
    \draw[fill=red] (3,4) circle (1.2pt);
  \draw[fill=red] (3.5,4) circle (1.2pt);
  \draw[fill=red] (4.5,4) circle (1.2pt);

   \draw[fill=black] (0,3.5) circle (1.2pt);
   \draw[fill=red] (0.5,3.5) circle (1.2pt);
  \draw[fill=red] (2,3.5) circle (1.2pt);
   \draw[fill=red] (2.5,3.5) circle (1.2pt);
    \draw[fill=red] (3,3.54) circle (1.2pt);
  \draw[fill=red] (3.5,3.5) circle (1.2pt);
  \draw[fill=red] (4.5,3.5) circle (1.2pt);

  \draw[fill=black] (0,3) circle (1pt);
   \draw[fill=red] (0.5,3) circle (1.2pt);
  \draw[fill=red] (2,3) circle (1.2pt);
   \draw[fill=red] (2.5,3) circle (1.2pt);
    \draw[fill=red] (3,3) circle (1.2pt);
  \draw[fill=red] (3.5,3) circle (1.2pt);
  \draw[fill=red] (4.5,3) circle (1.2pt);

  \draw[fill=black] (1.5,2.5) circle (1pt);
  \draw[fill=red] (2,2.5) circle (1.2pt);
   \draw[fill=red] (2.5,2.5) circle (1.2pt);
    \draw[fill=red] (3,2.5) circle (1.2pt);
  \draw[fill=red] (3.5,2.5) circle (1.2pt);
  \draw[fill=red] (4.5,2.5) circle (1.2pt);

\draw[fill=black] (1.5,2) circle (1pt);
  \draw[fill=red] (2,2) circle (1.2pt);
   \draw[fill=red] (2.5,2) circle (1.2pt);
    \draw[fill=red] (3,2) circle (1.2pt);
  \draw[fill=red] (3.5,2) circle (1.2pt);
  \draw[fill=red] (4.5,2) circle (1.2pt);

  \draw[fill=red] (2,1.5) circle (1.2pt);
   \draw[fill=red] (2.5,1.5) circle (1.2pt);
    \draw[fill=red] (3,1.5) circle (1.2pt);
  \draw[fill=red] (3.5,1.5) circle (1.2pt);
  \draw[fill=red] (4.5,1.5) circle (1.2pt);

  \draw[fill=black] (2.5,1) circle (1pt);
  \draw[fill=black] (3,1) circle (1pt);

 \draw[fill=black] (2,1) circle (1.2pt);
  \draw[fill=red] (3.5,0.5) circle (1.2pt);
  \draw[fill=red] (4.5,0.5) circle (1.2pt);
   \node[below, font=\small,scale=1] at (2.4,-0.2) {(a)\hspace{0.15cm} $\mathcal {C}$};

\end{tikzpicture}}
\end{minipage}\hfill
\begin{minipage}{0.43\textwidth}
\centering
\resizebox{\linewidth}{!}{%
\begin{tikzpicture}[every node/.style={font=\small}]
  \fill[black!16] (-0.5,4)--(-0.5,3)--(5,3)--(5,4)--(-0.5,4)--cycle;
  \fill[black!16] (0,2.5)--(0,2)--(5,2)--(5,2.5)--(0,2.5)--cycle;
 \fill[black!16] (2.5,0.5)--(3,0.5)--(3,1.3)--(2.5,1.3)--(2.5,0.5)--cycle;
 \fill[black!16] (-0.5,3)--(0,3)--(0,0)--(-0.5,0)--(-0.5,3)--cycle;
  \fill[black!16] (0,0)--(3,0)--(3,0.5)--(0,0.5)--(0,0)--cycle;
  \fill[black!16] (5,3)--(5,0.5)--(5.5,0.5)--(5.5,3)--(5,3)--cycle;

  \draw[fill=black!16,line width=0.8pt]
  (1.5,2.5)--(1.5,2)--(5,2)--(5,2.5)--(1.5,2.5)--cycle;

 \draw[line width=0.8pt] (2.5,1)--(2.5,0.5)  (3,1)--(3,0.5) (0,2.5)--(1.5,2.5) (0,2)--(1.5,2)  (-0.5,4)--(0,4)  (0,3)--(0,2.5) (0,2)--(0,0.5)--(2.5,0.5) (3,0.5)--(3,0)
 (5.5,0.5)--(5,0.5)--(5,3)--(5.5,3);

\textcolor[rgb]{0.00,0.59,0.00}{ \draw[line width=1pt]
 (1.5,2.5)--(1.5,2)--(5,2)--(5,2.5)--(1.5,2.5);
   \draw[line width=1pt]
 (0,4)--(0,3)--(5,3)--(5,4)--(0,4);}
     \textcolor[rgb]{0.00,0.59,0.00}{
       \draw[line width=1pt]
 (0,3.5)--(5,3.5);}
\fill[black!16] (2,1)--(2.5,1)--(2.5,1.3)--(2,1.3)--(2,1)--cycle;

\textcolor[rgb]{0.00,0.07,1.00}{ \draw[line width=1pt]
 (2.5,1)--(3,1)--(3,1.3)--(2.5,1.3)--(2.5,1);
\draw[line width=1pt]
(2,1)--(2.5,1)--(2.5,1.3)--(2,1.3)--(2,1);}

\textcolor[rgb]{1.00,1.00,1.00}{   \node[above, font=\small,scale=0.6] at (0.5,4.5) {$\mathbf{u}_{1,n}$};}

 \node[right, font=\small,scale=0.6] at (5,2.5) {$\mathbf{u}_{m,q}$};
  \node[right, font=\small,scale=0.6] at (5,3.1) {$\mathbf{u}_{m,\ell_1}$};
  \node[right, font=\small,scale=0.6] at (5,1.5) {$\mathbf{u}_{m,\ell_i}$};
  \node[right, font=\small,scale=0.6] at (5,0.6) {$\mathbf{u}_{m,1}$};

\draw[fill=black] (2,1.3) circle (1pt);
\draw[fill=black] (2,1) circle (1pt);

\draw[fill=black] (0,3.5) circle (1pt);
\draw[fill=black] (5,3.5) circle (1pt);
  \draw[fill=black] (5,2.5) circle (1pt);
  \draw[fill=black] (5,2) circle (1pt);
  \draw[fill=black] (5,1.5) circle (1pt);
   \draw[fill=black] (5,0.5) circle (1pt);
 \draw[fill=black] (5,4) circle (1pt);
  \draw[fill=black] (5,3) circle (1pt);
 \draw[fill=black] (0,4) circle (1pt);
  \draw[fill=black] (0,3) circle (1pt);

  \draw[fill=black] (1.5,2.5) circle (1pt);

\draw[fill=black] (1.5,2) circle (1pt);

  \draw[fill=black] (2.5,1) circle (1pt);
  \draw[fill=black] (3,1) circle (1pt);

  \draw[fill=black] (3,1.3) circle (1pt);
  \draw[fill=black] (2.5,1.3) circle (1pt);

   \node[below, font=\small,scale=1] at (2.4,-0.2)  {(b)\hspace{0.15cm} $\mathcal {C}'_{\mathbf{v}}$};
\end{tikzpicture}}
\end{minipage}\hfill

\caption{An illustration of Setup \ref{setup}}\label{fig9}
\end{figure}

From Setup \ref{setup}, the following observation and remark can be obtained directly.
\begin{ob} \label{ob-2}
 Under Setup \ref{setup},  let $\mathbf{v}=(v_1,v_2)$ and $\mathbf{w}=(w_1,w_2) \in V(\mathcal {C})$ with $w_1 < v_1$ and $w_2 < v_2$. If $\mathbf{w}\in U_{\mathbf{v}}(\mathcal {C})$, then $(p,q)\in U_{\mathbf{v}}(\mathcal {C})$ for each $(p,q)\in V(\mathcal {C})$ with $w_1\leq p<v_1$ and $w_2\leq q< v_2$. In other words, if $\mathbf{w}\notin U_{\mathbf{v}}(\mathcal {C})$, then $(k,\ell)\notin U_{\mathbf{v}}(\mathcal {C})$ for each $(k,\ell)\in V(\mathcal {C})$ with $k\leq w_1$ and $\ell\leq w_2$.
\end{ob}


\begin{remark} \label{T1}
By Setup \ref{setup}, it is easy to observe that each proper interval of $\mathcal {C}_{\mathbf{v}}$ is an inner interval of $\mathcal {C}$. Let $[\mathbf{a},\mathbf{b}]$ be an inner interval of $\mathcal {C}$  with $\mathbf{a} = (a_1, a_2)$, $\mathbf{b} = (b_1, b_2)$,  and anti-diagonal corners $\mathbf{c} = (a_1, b_2)$, $\mathbf{d} = (b_1, a_2)$.  Under Setup \ref{setup}, let $\mathbf{v}=(v_1, v_2)$. Then the following statements are true$:$ \\[-0.4cm]
\begin{itemize}
\item[(1)] If $[\mathbf{a},\mathbf{b}]\in \mathcal {C}$ with $\mathbf{c}\notin U_{\mathbf{v}}(\mathcal {C})$ and
$\mathbf{d}\in U_{\mathbf{v}}(\mathcal {C})$, then $[\mathbf{a}, (v_1, b_2)]\in \mathcal {T}_1$$;$ \\[-0.4cm] \\
if $[\mathbf{a},\mathbf{b}]\in \mathcal {C}$ with $\mathbf{d}\notin U_{\mathbf{v}}(\mathcal {C})$ and
$\mathbf{c}\in U_{\mathbf{v}}(\mathcal {C})$, then $[\mathbf{a}, (b_1, v_2)]\in \mathcal {T}_2$. \\[-0.4cm]
\item[(2)]
$[\mathbf{a},\mathbf{b}]\in \mathcal {T}_1$ if and only if
 $\mathbf{c}\notin U_{\mathbf{v}}(\mathcal {C})$, $b_1=v_1$, and $\emptyset\neq(\text{\em c}([\mathbf{a},\mathbf{b}])\setminus \{\mathbf{a},\mathbf{b}, \mathbf{c}, \mathbf{d}\})\cap V(\mathcal {C})\subseteq U_{\mathbf{v}}(\mathcal {C})$$;$ \\[-0.4cm]\\
  $[\mathbf{a},\mathbf{b}]\in \mathcal {T}_2$ if and only if
 $\mathbf{d}\notin U_{\mathbf{v}}(\mathcal {C})$, $b_2=v_2$, and
$\emptyset\neq(\text{\em c}([\mathbf{a},\mathbf{b}])\setminus \{\mathbf{a},\mathbf{b}, \mathbf{c}, \mathbf{d}\})\cap V(\mathcal {C})\subseteq U_{\mathbf{v}}(\mathcal {C})$.\\[-0.4cm]
\end{itemize}
\end{remark}

\begin{ex}\label{exm1}
Let $\mathcal {C}$ be a natural polyocollection,  and let  $\mathbf{v}\in V(\mathcal {C})$ be its uppermost and rightmost vertex, as shown in Figure \ref{fig10} (a). Then $U_{\mathbf{v}}(\mathcal {C})$ consists precisely of all the vertices marked in red in Figure \ref{fig10} (a). Moreover, $\mathcal {C}_{\mathbf{v}}$ and $\mathcal {C}'_{\mathbf{v}}$ are shown in (b) and (c)  of Figure \ref{fig10} respectively. One can easily observe that $\mathcal {C}_{\mathbf{v}}$ is a polyocollection but  is not natural, whereas $\mathcal {C}'_{\mathbf{v}}$  can be regarded as  a natural polyocollection and it is algebraically isomorphic to $\mathcal {C}_{\mathbf{v}}$.
\end{ex}

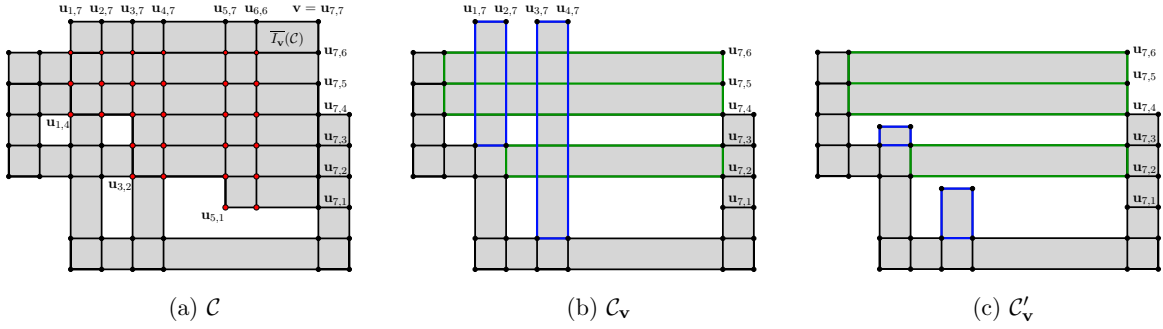
\begin{figure}[htbp]
\centering
\begin{minipage}{0.3\textwidth}
\centering
\resizebox{\linewidth}{!}{%
\begin{tikzpicture}[every node/.style={font=\small}]
 \draw[fill=black!16,line width=0.8pt]
  (1,3.5)--(1,2.5)--(2,2.5)--(2,1.5)--(3.5,1.5)--(3.5,1)--(5,1)--(5,3.5)--(1,3.5)--cycle
  (0,3)--(0,2.5)--(5,2.5)--(5,3)--(0,3)--cycle
  (0,3)--(0,1.5)--(0.5,1.5)--(0.5,3)--(0,3)--cycle
  (0,2)--(0,1.5)--(5,1.5)--(5,2)--(0,2)--cycle
  (1,3.5)--(1,0)--(1.5,0)--(1.5,3.5)--(1,3.5)--cycle
  (2,3.5)--(2,0)--(2.5,0)--(2.5,3.5)--(2,3.5)--cycle
  (1,0.5)--(1,0)--(5.5,0)--(5.5,0.5)--(1,0.5)--cycle
  (5,0)--(5.5,0)--(5.5,2.5)--(5,2.5)--(5,0)--cycle
  (3.5,1)--(3.5,3.5) (4,1)--(4,3.5) (5,2)--(5.5,2)  (5,1.5)--(5.5,1.5) (5,1)--(5.5,1);

   \draw[fill=black!16,line width=0.8pt]
   (1,4)--(1,3.5)--(5,3.5)--(5,4)--(1,4)--cycle;
   \draw[fill=black!5,line width=0.8pt]
    (1.5,3.5)--(1.5,4)  (2,3.5)--(2,4)   (2.5,3.5)--(2.5,4)  (3.5,3.5)--(3.5,4)  (4,3.5)--(4,4)  (5,3.5)--(5,4);

   \draw[fill=black!16,line width=0.8pt]
   (0,3.5)--(0,3)--(1,3)--(1,3.5)--(0,3.5)--cycle;

   \draw[fill=black!16,line width=0.8pt]
   (0.5,3.5)--(0.5,3);

  \node[above, font=\small,scale=0.6] at (1,4) {$\mathbf{u}_{1,7}$};
  \node[above, font=\small,scale=0.6] at (1.5,4) {$\mathbf{u}_{2,7}$};
  \node[above, font=\small,scale=0.6] at (2,4) {$\mathbf{u}_{3,7}$};
  \node[above, font=\small,scale=0.6] at (2.5,4) {$\mathbf{u}_{4,7}$};
  \node[above, font=\small,scale=0.6] at (3.5,4) {$\mathbf{u}_{5,7}$};
  \node[above, font=\small,scale=0.6] at (4,4) {$\mathbf{u}_{6,6}$};
  \node[above, font=\small,scale=0.6] at (5,4) {$\mathbf{v}=\mathbf{u}_{7,7}$};
  \node[right, font=\small,scale=0.6] at (5,3.5) {$\mathbf{u}_{7,6}$};
  \node[right, font=\small,scale=0.6] at (5,3) {$\mathbf{u}_{7,5}$};
  \node[right, font=\small,scale=0.6] at (5,2.6) {$\mathbf{u}_{7,4}$};
  \node[right, font=\small,scale=0.6] at (5,2.1) {$\mathbf{u}_{7,3}$};
  \node[right, font=\small,scale=0.6] at (5,1.6) {$\mathbf{u}_{7,2}$};
  \node[right, font=\small,scale=0.6] at (5,1.1) {$\mathbf{u}_{7,1}$};

  \node[below, font=\small,scale=0.6] at (0.8,2.5) {$\mathbf{u}_{1,4}$};
  \node[below, font=\small,scale=0.6] at (1.8,1.5) {$\mathbf{u}_{3,2}$};
   \node[below, font=\small,scale=0.6] at (3.3,1) {$\mathbf{u}_{5,1}$};

  \node[below, font=\small,scale=0.6] at (4.5,3.94) {$\overline{I_{\mathbf{v}}}(\mathcal {C})$};

 \draw[fill=black] (1,4) circle (1pt);
  \draw[fill=black] (1.5,4) circle (1pt);
  \draw[fill=black] (2,4) circle (1pt);
  \draw[fill=black] (2.5,4) circle (1pt);
 \draw[fill=black] (3.5,4) circle (1pt);
  \draw[fill=black] (4,4) circle (1pt);
  \draw[fill=black] (5,4) circle (1pt);

\draw[fill=black] (0.5,3.5) circle (1pt);
\draw[fill=black] (0,3.5) circle (1pt);

  \draw[fill=red] (1,3.5) circle (1pt);
  \draw[fill=red] (1.5,3.5) circle (1pt);
  \draw[fill=red] (2,3.5) circle (1pt);
  \draw[fill=red] (2.5,3.5) circle (1pt);
  \draw[fill=red] (3.5,3.5) circle (1pt);
  \draw[fill=red] (4,3.5) circle (1pt);
  \draw[fill=black] (5,3.5) circle (1pt);

  \draw[fill=black] (5,3) circle (1pt);
  \draw[fill=black] (5,2.5) circle (1pt);
  \draw[fill=black] (5,2) circle (1pt);
  \draw[fill=black] (5,1.5) circle (1pt);
  \draw[fill=black] (5,1) circle (1pt);

  \draw[fill=red] (1,3) circle (1.2pt);
  \draw[fill=red] (1.5,3) circle (1.2pt);
  \draw[fill=red] (2,3) circle (1.2pt);
  \draw[fill=red] (2.5,3) circle (1.2pt);
  \draw[fill=red] (3.5,3) circle (1.2pt);
  \draw[fill=red] (4,3) circle (1.2pt);

    \draw[fill=red] (1,2.5) circle (1.2pt);
  \draw[fill=red] (1.5,2.5) circle (1.2pt);
  \draw[fill=red] (2,2.5) circle (1.2pt);
  \draw[fill=red] (2.5,2.5) circle (1.2pt);
  \draw[fill=red] (3.5,2.5) circle (1.2pt);
  \draw[fill=red] (4,2.5) circle (1.2pt);

    \draw[fill=red] (2,2) circle (1.2pt);
  \draw[fill=red] (2.5,2) circle (1.2pt);
  \draw[fill=red] (3.5,2) circle (1.2pt);
  \draw[fill=red] (4,2) circle (1.2pt);

    \draw[fill=red] (2,1.5) circle (1.2pt);
  \draw[fill=red] (2.5,1.5) circle (1.2pt);
  \draw[fill=red] (3.5,1.5) circle (1.2pt);
  \draw[fill=red] (4,1.5) circle (1.2pt);

    \draw[fill=red] (3.5,1) circle (1.2pt);
  \draw[fill=red] (4,1) circle (1.2pt);

  \draw[fill=black] (0,3) circle (1pt);
  \draw[fill=black] (0.5,3) circle (1pt);

  \draw[fill=black] (0,2.5) circle (1pt);
  \draw[fill=black] (0.5,2.5) circle (1pt);

    \draw[fill=black] (0,2) circle (1pt);
  \draw[fill=black] (0.5,2) circle (1pt);
      \draw[fill=black] (1,2) circle (1pt);
  \draw[fill=black] (1.5,2) circle (1pt);

   \draw[fill=black] (0,1.5) circle (1pt);
  \draw[fill=black] (0.5,1.5) circle (1pt);
   \draw[fill=black] (1,1.5) circle (1pt);
  \draw[fill=black] (1.5,1.5) circle (1pt);

    \draw[fill=black] (1,0.5) circle (1pt);
  \draw[fill=black] (1.5,0.5) circle (1pt);
    \draw[fill=black] (2,0.5) circle (1pt);
  \draw[fill=black] (2.5,0.5) circle (1pt);
      \draw[fill=black] (5,0.5) circle (1pt);
  \draw[fill=black] (5.5,0.5) circle (1pt);

      \draw[fill=black] (1,0) circle (1pt);
  \draw[fill=black] (1.5,0) circle (1pt);
    \draw[fill=black] (2,0) circle (1pt);
  \draw[fill=black] (2.5,0) circle (1pt);
      \draw[fill=black] (5,0) circle (1pt);
  \draw[fill=black] (5.5,0) circle (1pt);
  \draw[fill=black] (5.5,1) circle (1pt);
  \draw[fill=black] (5.5,1.5) circle (1pt);
  \draw[fill=black] (5.5,2) circle (1pt);
  \draw[fill=black] (5.5,2.5) circle (1pt);

   \node[below, font=\small,scale=1] at (3,-0.3) {(a)~~~~~~~~~ $\mathcal {C}$};

\end{tikzpicture}}
\end{minipage}\hfill
\begin{minipage}{0.3\textwidth}
\centering
\resizebox{\linewidth}{!}{%
\begin{tikzpicture}[every node/.style={font=\small}]
\draw[fill=black!16,line width=0.8pt]

  (0,3)--(0,2.5)--(5,2.5)--(5,3)--(0,3)--cycle
  (0,3)--(0,1.5)--(0.5,1.5)--(0.5,3)--(0,3)--cycle
  (0,2)--(0,1.5)--(5,1.5)--(5,2)--(0,2)--cycle
  (1,4)--(1,0)--(1.5,0)--(1.5,4)--(1,4)--cycle
  (2,4)--(2,0)--(2.5,0)--(2.5,4)--(2,4)--cycle
  (1,0.5)--(1,0)--(5.5,0)--(5.5,0.5)--(1,0.5)--cycle
  (5,0)--(5.5,0)--(5.5,2.5)--(5,2.5)--(5,0)--cycle
  (5,2)--(5.5,2) (5,1.5)--(5.5,1.5) (5,1)--(5.5,1);

\draw[fill=black!16,line width=0.8pt]
  (0.5,3.5)--(0.5,2.5)--(5,2.5)--(5,3.5)--(0.5,3.5)--cycle;

  \draw[fill=black!16,line width=0.8pt]
   (0,3.5)--(0,3)--(1,3)--(1,3.5)--(0,3.5)--cycle;

   \draw[fill=black!16,line width=0.8pt]
   (0.5,3.5)--(0.5,3);

  \textcolor[rgb]{0.00,0.59,0.00}{
   \draw[line width=1pt]
  (0.5,3.5)--(0.5,2.5)--(5,2.5)--(5,3.5)--(0.5,3.5)
  (1.5,2)--(1.5,1.5)--(5,1.5)--(5,2)--(1.5,2)  (0.5,3)--(5,3);}

 \textcolor[rgb]{0.00,0.07,1.00}{
  \draw[line width=1pt]
  (1,2)--(1.5,2)--(1.5,4)--(1,4)--(1,2)
  (2,0.5)--(2.5,0.5)--(2.5,4)--(2,4)--(2,0.5);}

  \node[above, font=\small,scale=0.6] at (1,4) {$\mathbf{u}_{1,7}$};
  \node[above, font=\small,scale=0.6] at (1.5,4) {$\mathbf{u}_{2,7}$};
  \node[above, font=\small,scale=0.6] at (2,4) {$\mathbf{u}_{3,7}$};
  \node[above, font=\small,scale=0.6] at (2.5,4) {$\mathbf{u}_{4,7}$};

  \node[right, font=\small,scale=0.6] at (5,3.5) {$\mathbf{u}_{7,6}$};

  \node[right, font=\small,scale=0.6] at (5,3) {$\mathbf{u}_{7,5}$};
  \node[right, font=\small,scale=0.6] at (5,2.6) {$\mathbf{u}_{7,4}$};
  \node[right, font=\small,scale=0.6] at (5,2.1) {$\mathbf{u}_{7,3}$};
  \node[right, font=\small,scale=0.6] at (5,1.6) {$\mathbf{u}_{7,2}$};
  \node[right, font=\small,scale=0.6] at (5,1.1) {$\mathbf{u}_{7,1}$};

   \draw[fill=black] (1,4) circle (1pt);
  \draw[fill=black] (1.5,4) circle (1pt);
  \draw[fill=black] (2,4) circle (1pt);
  \draw[fill=black] (2.5,4) circle (1pt);

\draw[fill=black] (0.5,3.5) circle (1pt);
\draw[fill=black] (0,3.5) circle (1pt);
\draw[fill=black] (5,3.5) circle (1pt);

  \draw[fill=black] (5,3) circle (1pt);
  \draw[fill=black] (5,2.5) circle (1pt);
  \draw[fill=black] (5,2) circle (1pt);
  \draw[fill=black] (5,1.5) circle (1pt);
  \draw[fill=black] (5,1) circle (1pt);

   \draw[fill=black] (0,3) circle (1pt);
  \draw[fill=black] (0.5,3) circle (1pt);

  \draw[fill=black] (0,2.5) circle (1pt);
  \draw[fill=black] (0.5,2.5) circle (1pt);

    \draw[fill=black] (0,2) circle (1pt);
  \draw[fill=black] (0.5,2) circle (1pt);
      \draw[fill=black] (1,2) circle (1pt);
  \draw[fill=black] (1.5,2) circle (1pt);

   \draw[fill=black] (0,1.5) circle (1pt);
  \draw[fill=black] (0.5,1.5) circle (1pt);
   \draw[fill=black] (1,1.5) circle (1pt);
  \draw[fill=black] (1.5,1.5) circle (1pt);

    \draw[fill=black] (1,0.5) circle (1pt);
  \draw[fill=black] (1.5,0.5) circle (1pt);
    \draw[fill=black] (2,0.5) circle (1pt);
  \draw[fill=black] (2.5,0.5) circle (1pt);
      \draw[fill=black] (5,0.5) circle (1pt);
  \draw[fill=black] (5.5,0.5) circle (1pt);

      \draw[fill=black] (1,0) circle (1pt);
  \draw[fill=black] (1.5,0) circle (1pt);
    \draw[fill=black] (2,0) circle (1pt);
  \draw[fill=black] (2.5,0) circle (1pt);
      \draw[fill=black] (5,0) circle (1pt);
  \draw[fill=black] (5.5,0) circle (1pt);
  \draw[fill=black] (5.5,1) circle (1pt);
  \draw[fill=black] (5.5,1.5) circle (1pt);
  \draw[fill=black] (5.5,2) circle (1pt);
  \draw[fill=black] (5.5,2.5) circle (1pt);

   \node[below, font=\small,scale=1] at (3,-0.3) {(b)~~~~~~~~~ $\mathcal {C}_{\mathbf{v}}$};
\end{tikzpicture}}
\end{minipage}\hfill
\begin{minipage}{0.3\textwidth}
\centering
\resizebox{\linewidth}{!}{%
\begin{tikzpicture}[every node/.style={font=\small}]

\textcolor[rgb]{1.00,1.00,1.00}{\node[above, font=\small,scale=0.6] at (1,3.5) {$\mathbf{u}_{1,6}$};}

 \draw[fill=black!16,line width=0.8pt]
  (0,3)--(0,2.5)--(5,2.5)--(5,3)--(0,3)--cycle
  (0,3)--(0,1.5)--(0.5,1.5)--(0.5,3)--(0,3)--cycle
  (0,2)--(0,1.5)--(5,1.5)--(5,2)--(0,2)--cycle
  (1,2.3)--(1,0)--(1.5,0)--(1.5,2.3)--(1,2.3)--cycle
  (2,1.3)--(2,0)--(2.5,0)--(2.5,1.3)--(2,1.3)--cycle
  (1,0.5)--(1,0)--(5.5,0)--(5.5,0.5)--(1,0.5)--cycle
  (5,0)--(5.5,0)--(5.5,2.5)--(5,2.5)--(5,0)--cycle
  (5,2)--(5.5,2) (5,1.5)--(5.5,1.5) (5,1)--(5.5,1);

    \draw[fill=black!16,line width=0.8pt]
   (0,3.5)--(0,3)--(0.5,3)--(0.5,3.5)--(0,3.5)--cycle;

\draw[fill=black!16,line width=0.8pt]
  (0.5,3.5)--(0.5,2.5)--(5,2.5)--(5,3.5)--(0.5,3.5)--cycle;

   \draw[fill=black!16,line width=0.8pt]
   (0.5,3.5)--(0.5,3);

  \textcolor[rgb]{0.00,0.59,0.00}{
   \draw[line width=1pt]
  (0.5,3.5)--(0.5,2.5)--(5,2.5)--(5,3.5)--(0.5,3.5)
  (1.5,2)--(1.5,1.5)--(5,1.5)--(5,2)--(1.5,2)  (0.5,3)--(5,3);}

 \textcolor[rgb]{0.00,0.07,1.00}{
  \draw[line width=1pt]
  (1,2)--(1.5,2)--(1.5,2.3)--(1,2.3)--(1,2)
  (2,0.5)--(2.5,0.5)--(2.5,1.3)--(2,1.3)--(2,0.5);
  }

    \draw[fill=black] (1,2.3) circle (1pt);
  \draw[fill=black] (1.5,2.3) circle (1pt);

  \draw[fill=black] (2.5,1.3) circle (1pt);
  \draw[fill=black] (2,1.3) circle (1pt);

  \node[right, font=\small,scale=0.6] at (5,3.5) {$\mathbf{u}_{7,6}$};
  \node[right, font=\small,scale=0.6] at (5,3.1) {$\mathbf{u}_{7,5}$};
  \node[right, font=\small,scale=0.6] at (5,2.6) {$\mathbf{u}_{7,4}$};
  \node[right, font=\small,scale=0.6] at (5,2.1) {$\mathbf{u}_{7,3}$};
  \node[right, font=\small,scale=0.6] at (5,1.6) {$\mathbf{u}_{7,2}$};
  \node[right, font=\small,scale=0.6] at (5,1.1) {$\mathbf{u}_{7,1}$};

  \draw[fill=black] (5,3.5) circle (1pt);
  \draw[fill=black] (5,3) circle (1pt);
  \draw[fill=black] (5,2.5) circle (1pt);
  \draw[fill=black] (5,2) circle (1pt);
  \draw[fill=black] (5,1.5) circle (1pt);
  \draw[fill=black] (5,1) circle (1pt);

  \draw[fill=black] (0,3.5) circle (1pt);

   \draw[fill=black] (0,3) circle (1pt);
  \draw[fill=black] (0.5,3) circle (1pt);
  \draw[fill=black] (0.5,3.5) circle (1pt);

  \draw[fill=black] (0,2.5) circle (1pt);
  \draw[fill=black] (0.5,2.5) circle (1pt);

    \draw[fill=black] (0,2) circle (1pt);
  \draw[fill=black] (0.5,2) circle (1pt);
      \draw[fill=black] (1,2) circle (1pt);
  \draw[fill=black] (1.5,2) circle (1pt);

   \draw[fill=black] (0,1.5) circle (1pt);
  \draw[fill=black] (0.5,1.5) circle (1pt);
   \draw[fill=black] (1,1.5) circle (1pt);
  \draw[fill=black] (1.5,1.5) circle (1pt);

    \draw[fill=black] (1,0.5) circle (1pt);
  \draw[fill=black] (1.5,0.5) circle (1pt);
    \draw[fill=black] (2,0.5) circle (1pt);
  \draw[fill=black] (2.5,0.5) circle (1pt);
      \draw[fill=black] (5,0.5) circle (1pt);
  \draw[fill=black] (5.5,0.5) circle (1pt);

      \draw[fill=black] (1,0) circle (1pt);
  \draw[fill=black] (1.5,0) circle (1pt);
    \draw[fill=black] (2,0) circle (1pt);
  \draw[fill=black] (2.5,0) circle (1pt);
      \draw[fill=black] (5,0) circle (1pt);
  \draw[fill=black] (5.5,0) circle (1pt);
  \draw[fill=black] (5.5,1) circle (1pt);
  \draw[fill=black] (5.5,1.5) circle (1pt);
  \draw[fill=black] (5.5,2) circle (1pt);
  \draw[fill=black] (5.5,2.5) circle (1pt);

\textcolor[rgb]{1.00,1.00,1.00}{  \node[above, font=\small,scale=0.6] at (1,4) {$\mathbf{u}_{1,7}$};}

   \node[below, font=\small,scale=1] at (3,-0.3) {(c)~~~~~~~~~ $\mathcal {C}'_{\mathbf{v}}$};
\end{tikzpicture}}
\end{minipage}
\caption{An illustration of Example \ref{exm1}}\label{fig10}
\end{figure}

The following Lemma shows that the results of the above example hold for all natural polyocollections.

\begin{lem} \label{cv-1}
Under Setup \ref{setup}, the following statements hold$:$\\[-0.4cm]
\begin{itemize}
\item[(1)] $\mathcal {C}_{\mathbf{v}}$ is a polyocollection$;$\\[-0.4cm]
\item[(2)] $\mathcal {C}'_{\mathbf{v}}$ can be regarded as a natural polyocollection and it is  algebraically isomorphic to $\mathcal {C}_{\mathbf{v}}$.\\[-0.4cm]
\end{itemize}
\end{lem}

\begin{proof} (1) By the definition of polyocollection, it suffices to prove $|F\cap G|\leq 1$ for each two distinct edge $F, G\in E(\mathcal {C}_{\mathbf{v}})$.
Note that $\mathcal {C}$ is a natural polyocollection. If $F, G\in E(\mathcal {C}_{\mathbf{v}}) \cap E(\mathcal {C})$, then it is obvious that $|F\cap G|\leq 1$. Now, without loss of generality, assume that  $F\in E(\mathcal {C}_{\mathbf{v}})\setminus E(\mathcal {C})$, then $F$ is a horizontal edge interval of $\mathcal {T}_1$ or a vertical edge intervals of $\mathcal {T}_2$. Set $F=[\mathbf{a},\mathbf{b}]$. By Remark \ref{T1} (2), one has $(\text{c}(F)\setminus \{\mathbf{a},\mathbf{b}\})\cap V(\mathcal {C})\subseteq U_{\mathbf{v}}(\mathcal {C})$. Consider the following cases:

(i) $G \in E(\mathcal {C}_{\mathbf{v}})\setminus E(\mathcal {C})$. In this case, $G$ is parallel or perpendicular to $F$. Hence $|F\cap G|=0$ or $|F\cap G|=1$, as desired.

(ii) $G \in E(\mathcal {C}_{\mathbf{v}})\cap E(\mathcal {C})$. Note that $\text{c}(G) \cap U_{\mathbf{v}}(\mathcal {C}) = \emptyset$ and $(\text{c}(F)\setminus \{\mathbf{a},\mathbf{b}\})\cap V(\mathcal {C})\subseteq U_{\mathbf{v}}(\mathcal {C})$. It follows that $|F\cap G| \leq 1$, as desired.

(2) We claim that for each pair of distinct $I,J\in \mathcal{C}'_{\mathbf{v}}$, $\text{int}(I)\cap \text{int}(J)=\emptyset$ holds true. In fact, note that $\mathcal {C}'_{\mathbf{v}}=\big(\mathcal {C}\setminus U_{\mathbf{v}}(\mathcal {C})\big)\cup \mathcal {T}_1\cup \mathcal {T}'_2$. If $I$ or $J$ is in $\mathcal {C}\setminus U_{\mathbf{v}}(\mathcal {C})$, it is clear that $\text{int}(I)\cap \text{int}(J)=\emptyset$; if both of $I$ and $J$ are in $\mathcal {T}_1$ or in $\mathcal {T}'_2$, it is easy to see that $\text{int}(I)\cap \text{int}(J)=\emptyset$; if one of $I$ and $J$ is in $\mathcal {T}_1$, the other is in $\mathcal {T}'_2$, then it follows from the construction of the proper intervals in $\mathcal {T}'_2$ that $\text{int}(I)\cap \text{int}(J)=\emptyset$.

It is direct to check that $\mathcal {C}_{\mathbf{v}}$ is algebraically isomorphic to $\mathcal {C}'_{\mathbf{v}}$, which is algebraically isomorphic to $2\mathcal{C}'_{\mathbf{v}}$. Note that $2\mathcal{C}'_{\mathbf{v}}$ is a collection of proper intervals on $\mathbb{Z}^2$. It follows that $2\mathcal{C}'_{\mathbf{v}}$ is a natural polyocollection, and $\mathcal {C}'_{\mathbf{v}}$ can be regarded as a natural polyocollection.
\end{proof}

Now, we could view $\mathcal {T}_1$ and $\mathcal {T}_2$ as two subpolyocollections of $\mathcal {C}_{\mathbf{v}}$.  In the following, we present some auxiliary results required for the proof of the part (3) of Main theorem.

\begin{lem} \label{cv-2}
Under Setup \ref{setup},
$\mathcal {I}(\mathcal {C}_{\mathbf{v}})=\big\{[\mathbf{a},\mathbf{b}] \in \mathcal {I}(\mathcal {C}) \mid \mathbf{a}, \mathbf{b}\notin U_{\mathbf{v}}(\mathcal {C})\big\}.$
\end{lem}

\begin{proof}
Since each basic interval of $\mathcal {C}_{\mathbf{v}}$ is an inner interval of $\mathcal {C}$ and its four corners are not in $ U_{\mathbf{v}}(\mathcal {C})$ by Remark \ref{T1}, $\mathcal {I}(\mathcal {C}_{\mathbf{v}})\subseteq \big\{[\mathbf{a},\mathbf{b}] \in \mathcal {I}(\mathcal {C}) \mid \mathbf{a}, \mathbf{b}\notin U_{\mathbf{v}}(\mathcal {C})\big\}$.

Conversely, let $\mathbf{v}=(v_1,v_2)$ and let $[\mathbf{a},\mathbf{b}]$ be an  inner interval of $\mathcal {C}$ with  $\mathbf{a}, \mathbf{b}\notin U_{\mathbf{v}}(\mathcal {C})$. Assume that $\mathbf{a}=(a_1, a_2)$ and $\mathbf{b}=(b_1,b_2)$, then $\mathbf{c}=(a_1, b_2)$ and $\mathbf{d}=(b_1, a_2)$ are anti-diagonal corners of $[\mathbf{a},\mathbf{b}]$.
If $\text{c}([\mathbf{a},\mathbf{b}]) \cap U_{\mathbf{v}}(\mathcal {C}) = \emptyset$, then $[\mathbf{a},\mathbf{b}]\in \mathcal {I}(\mathcal {C}_{\mathbf{v}})$. If $\text{c}([\mathbf{a},\mathbf{b}]) \cap U_{\mathbf{v}}(\mathcal {C}) \neq \emptyset$, then consider the following two cases.\\[-0.4cm]

(i) $b_2 < v_2$. Note that $\mathbf{b}\notin  U_{\mathbf{v}}(\mathcal {C})$. It follows from Observation \ref{ob-2} that $b_1\geq v_1$. Note that $\mathbf{a}\notin  U_{\mathbf{v}}(\mathcal {C})$. There exists some $\mathbf{w} =(w_1, w_2)\in V(\mathcal {C})$ with $a_1\leq w_1 < v_1$ and $b_2 \leq w_2 < v_2$, such that the cell $C^{(1)}_\mathbf{w}$ is not contained in any basic interval of $\mathcal{C}$. Let $e=\max\big\{i\mid a_1\leq i<v_1, (i,b_2) \in V(\mathcal {C}) \,\text{and}\, [(i,b_2),\mathbf{v}]\notin \mathcal {I}(\mathcal {C})\big\}$, as shown in Figure \ref{fig11} (a).
It is clear that $e \geq a_1$. By Observation \ref{ob-2}, one has
$$\text{c}([\mathbf{a},\mathbf{b}])\cap U_{\mathbf{v}}(\mathcal {C})=\big\{(i,j)\in V(\mathcal {C}) \mid e<i<v_1, a_2\leq j \leq b_2\big\}.$$
Therefore,  $[(e,a_2),(v_1,b_2)]\in \mathcal {I}(\mathcal {T}_1)$ by Remark \ref{T1} (2), $[\mathbf{a},(e,b_2)]\in \mathcal {I}(\mathcal {C}\setminus U_{\mathbf{v}}(\mathcal {C}))$ when $e> a_1$, and  $[(v_1,a_2),\mathbf{b})]\in \mathcal {I}(\mathcal {C}\setminus U_{\mathbf{v}}(\mathcal {C}))$ when $b_1> v_1$. Furthermore, by formula (\ref{formula2}) and the equation
$$\text{c}([\mathbf{a},\mathbf{b}])=\text{c}([\mathbf{a},(e,b_2)])\cup \text{c}([(e,a_2),(v_1,b_2)])\cup \text{c}([(v_1,a_2),\mathbf{b})]),$$
we obtain that $[\mathbf{a},\mathbf{b}]\in \mathcal {I}(\mathcal {C}_{\mathbf{v}})$.\\[-0.4cm]

(ii) $b_2 = v_2$. It follows from $\mathbf{a}\notin  U_{\mathbf{v}}(\mathcal {C})$ that $b_1<v_1$. By a similar discussion as above, we could denote $e'=\max\big\{j\mid a_2\leq j< v_2, (b_1,j) \in V(\mathcal {C}) \,\text{and}\, [(b_1,j),\mathbf{v}]\notin \mathcal {I}(\mathcal {C})\big\}$, as shown in Figure \ref{fig11} (b).
It is clear that $e' \geq a_2$. By Observation \ref{ob-2}, one has
$$\text{c}([\mathbf{a},\mathbf{b}])\cap  U_{\mathbf{v}}(\mathcal {C})=\big\{(i,j)\in V(\mathcal {C})\mid  a_1\leq i\leq b_1, e'<j < b_2\big\}.$$
Therefore, $[(a_1,e'),\mathbf{b}]\in \mathcal {I}(\mathcal {T}_2)$ by Remark \ref{T1} (2), and $[\mathbf{a},(b_1,e')]\in \mathcal {I}(\mathcal {C}\setminus U_{\mathbf{v}}(\mathcal {C}))$ when $e'> a_2$. Furthermore, by formula (\ref{formula2}) and the equation
$$\text{c}([\mathbf{a},\mathbf{b}])=\text{c}([(a_1,e'),\mathbf{b}])\cup \text{c}([\mathbf{a},(b_1,e')]),$$
we obtain that $[\mathbf{a},\mathbf{b}]\in \mathcal {I}(\mathcal {C}_{\mathbf{v}})$.
\end{proof}

\begin{figure}[htbp]
\centering
\begin{minipage}{0.40\textwidth}
\centering
\resizebox{\linewidth}{!}{%
\begin{tikzpicture}[every node/.style={font=\small}]
      \draw[fill=black!16,line width=0.8pt]
    (0,0)--(5,0)--(5,3)--(4,3)--(4,4)--(2,4)--(2,3)--(0,3)--(0,0)--cycle;
  \draw[line width=0.8pt] (1,0)--(1,3)  (2,0)--(2,4) (3,0)--(3,4) (4,0)--(4,4) (0,1)--(5,1)  (0,2)--(5,2) (2,3)--(4,3);
  \draw[fill=black] (0,0) circle (1pt);
  \draw[fill=black] (1,0) circle (1pt);
  \draw[fill=black] (2,0) circle (1pt);
  \draw[fill=black] (3,0) circle (1pt);
  \draw[fill=black] (4,0) circle (1pt);
  \draw[fill=black] (5,0) circle (1pt);

  \draw[fill=black] (0,1) circle (1pt);
  \draw[fill=black] (1,1) circle (1pt);
  \draw[fill=black] (2,1) circle (1pt);
  \draw[fill=black] (3,1) circle (1pt);
  \draw[fill=black] (4,1) circle (1pt);
  \draw[fill=black] (5,1) circle (1pt);

  \draw[fill=black] (0,2) circle (1pt);
  \draw[fill=black] (1,2) circle (1pt);
  \draw[fill=black] (2,2) circle (1pt);
  \draw[fill=black] (3,2) circle (1pt);
  \draw[fill=black] (4,2) circle (1pt);
  \draw[fill=black] (5,2) circle (1pt);

  \draw[fill=black] (0,3) circle (1pt);
  \draw[fill=black] (1,3) circle (1pt);
  \draw[fill=black] (2,3) circle (1pt);
  \draw[fill=black] (3,3) circle (1pt);
  \draw[fill=black] (4,3) circle (1pt);
  \draw[fill=black] (5,3) circle (1pt);

  \draw[fill=black] (2,4) circle (1pt);
  \draw[fill=black] (3,4) circle (1pt);
  \draw[fill=black] (4,4) circle (1pt);

 \node[below, font=\small,scale=0.6] at (0,0) {$\mathbf{a}=(a_1,a_2)$};
 \node[below, font=\small,scale=0.6] at (1,0) {$(e,a_2)$};
 \node[below, font=\small,scale=0.6] at (4,0) {$(v_1,a_2)$};
 \node[above, font=\small,scale=0.6] at (1,3) {$(e,b_2)$};
 \node[right, font=\small,scale=0.6] at (4,3.2) {$(v_1,b_2)$};
 \node[above, font=\small,scale=0.6] at (4,4) {$\mathbf{v}=(v_1,v_2)$};
 \node[right, font=\small,scale=0.6] at (5,0) {$\mathbf{d}$};
 \node[right, font=\small,scale=0.6] at (5,3) {$\mathbf{b}=(b_1,b_2)$};

\textcolor[rgb]{1.00,1.00,1.00}{   \node[left, font=\small,scale=0.6] at (-1.5,0) {$1$};}

 \node[below, font=\small,scale=1] at (2.4,-0.5) {(a) $b_2<v_2$};

\end{tikzpicture}}
\end{minipage}\hfill
\begin{minipage}{0.40\textwidth}
\centering
\resizebox{\linewidth}{!}{%
\begin{tikzpicture}[every node/.style={font=\small}]
         \draw[fill=black!16,line width=0.8pt]
    (0,0)--(4,0)--(4,2)--(5,2)--(5,4)--(0,4)--(0,0)--cycle;
  \draw[line width=0.8pt] (1,0)--(1,4)  (2,0)--(2,4) (4,0)--(4,4) (0,1)--(4,1)  (0,2)--(4,2) (0,3)--(5,3);
  \draw[fill=black] (0,0) circle (1pt);
  \draw[fill=black] (1,0) circle (1pt);
  \draw[fill=black] (2,0) circle (1pt);
  \draw[fill=black] (4,0) circle (1pt);

  \draw[fill=black] (0,1) circle (1pt);
  \draw[fill=black] (1,1) circle (1pt);
  \draw[fill=black] (2,1) circle (1pt);
  \draw[fill=black] (4,1) circle (1pt);

  \draw[fill=black] (0,2) circle (1pt);
  \draw[fill=black] (1,2) circle (1pt);
  \draw[fill=black] (2,2) circle (1pt);
  \draw[fill=black] (4,2) circle (1pt);

  \draw[fill=black] (0,3) circle (1pt);
  \draw[fill=black] (1,3) circle (1pt);
  \draw[fill=black] (2,3) circle (1pt);
  \draw[fill=black] (4,3) circle (1pt);
  \draw[fill=black] (5,3) circle (1pt);

  \draw[fill=black] (0,4) circle (1pt);
  \draw[fill=black] (1,4) circle (1pt);
  \draw[fill=black] (2,4) circle (1pt);
  \draw[fill=black] (4,4) circle (1pt);
  \draw[fill=black] (5,4) circle (1pt);

 \node[below, font=\small,scale=0.6] at (0,0) {$\mathbf{a}=(a_1,a_2)$};
 \node[right, font=\small,scale=0.6] at (4,1) {$(b_1,e')$};
 \node[left, font=\small,scale=0.6] at (0,1) {$(a_1,e')$};
 \node[right, font=\small,scale=0.6] at (4,3.3) {$(v_1,b_2)$};
 \node[right, font=\small,scale=0.6] at (5,4) {$\mathbf{v}=(v_1,v_2)$};
 \node[right, font=\small,scale=0.6] at (4,0) {$\mathbf{d}$};
 \node[above, font=\small,scale=0.6] at (4,4) {$\mathbf{b}=(b_1,b_2)$};

 \node[below, font=\small,scale=1] at (2.4,-0.5) {(b) $b_2=v_2$};

\textcolor[rgb]{1.00,1.00,1.00}{\node[left, font=\small,scale=0.6] at (6.5,0) {$1$};}
\end{tikzpicture}}
\end{minipage}\hfill
\caption{The positional relationship between $[\mathbf{a},\mathbf{b}]$ and $\mathbf{v}$}\label{fig11}
\end{figure}
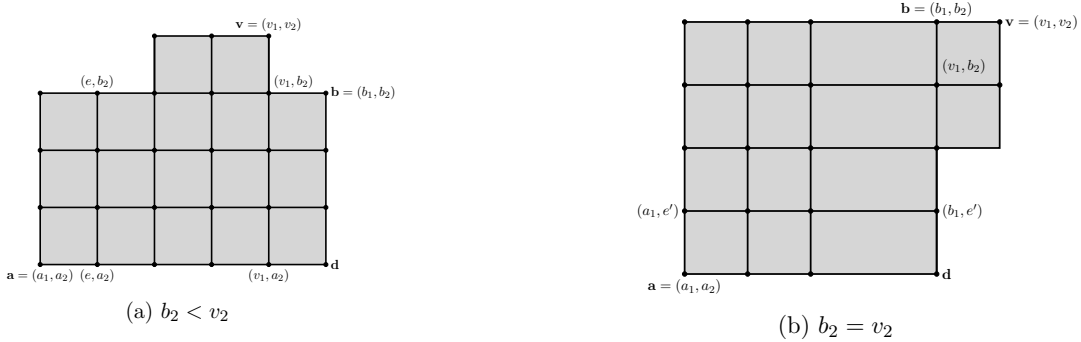

\begin{prop} \label{induction-1}
Under Setup \ref{setup}, if $\mathcal {C}$ is a natural polyocollection of type $\mathcal{Q}_1$, then the following statements hold$:$ \\[-0.3cm]
\begin{itemize}
\item[(1)] $(\text{\em in}_{<_{\text{\em lex}}}(I_{\mathcal {C}}),x_{\mathbf{v}})=(\text{\em in}_{<_{\text{\em lex}}}(I_{\mathcal {C}\setminus \mathbf{v}}),x_{\mathbf{v}})$. Moreover, $\mathcal {C}\setminus \mathbf{v}$ belongs to type $\mathcal{Q}_1$. \\[-0.3cm]
\item[(2)] $(\text{\em in}_{<_{\text{\em lex}}}(I_{\mathcal {C}}):x_{\mathbf{v}})=\text{\em in}_{<_{\text{\em lex}}}(I_{\mathcal {C}_{\mathbf{v}}})+\big(x_{\mathbf{u}}\mid \mathbf{u}\in U_{\mathbf{v}}(\mathcal {C}) \big)$. Moreover, $\mathcal {C}_{\mathbf{v}}$ belongs to  type $\mathcal{Q}_1$.\\[-0.3cm]
\end{itemize}
\end{prop}

\begin{proof}
 (1) The inclusion ``$\supseteq$" is obvious. Recall the notations of Setup \ref{setup}, we can get that $\mathbf{v}$ is a upper right single vertex of $\mathcal {C}$. For any $\mathbf{a},\mathbf{b}\in V(\mathcal {C}\setminus \mathbf{v})$, it follows from Observation \ref{ob-1} that  $[\mathbf{a},\mathbf{b}]\in \mathcal {I}(\mathcal {C}\setminus \mathbf{v})$ if and only if $[\mathbf{a},\mathbf{b}]\in \mathcal {I}(\mathcal {C})$. Using Remark \ref{graph}, we have
\begin{eqnarray*}
(\text{in}_{<_{\text{lex}}}(I_{\mathcal {C}}),x_{\mathbf{v}})\!\!&=&\!\!\big(x_{\mathbf{a}}x_{\mathbf{b}}\mid [\mathbf{a},\mathbf{b}]\in \mathcal {I}(\mathcal {C})  \text{ with } \mathbf{b}\neq \mathbf{v}\big)+\big(x_{\mathbf{a}}x_{\mathbf{v}}\mid [\mathbf{a},\mathbf{v}]\in \mathcal {I}(\mathcal {C})\big)+(x_{\mathbf{v}})\\
\!\!&=&\!\!\big(x_{\mathbf{a}}x_{\mathbf{b}}\mid [\mathbf{a},\mathbf{b}]\in \mathcal {I}(\mathcal {C}\setminus \mathbf{v})\big)+(x_{\mathbf{v}})\\
\!\!&\subseteq&\!\!\text{in}_{<_{\text{lex}}}(I_{\mathcal {C}\setminus \mathbf{v}})+(x_{\mathbf{v}})\\
\!\!&\subseteq&\!\!(\text{in}_{<_{\text{lex}}}(I_{\mathcal {C}}),x_{\mathbf{v}}).
\end{eqnarray*}
Hence, $\big(x_{\mathbf{a}}x_{\mathbf{b}}\mid [\mathbf{a},\mathbf{b}]\in \mathcal {I}(\mathcal {C}\setminus \mathbf{v})\big)+(x_{\mathbf{v}})=
\text{in}_{<_{\text{lex}}}(I_{\mathcal {C}\setminus \mathbf{v}})+(x_{\mathbf{v}})$.
Note that each element in the minimal monomial generator set of $\text{in}_{<_{\text{lex}}}(I_{\mathcal {C}\setminus \mathbf{v}})$ is not divisible by the variable $x_{\mathbf{v}}$, we deduce $\big(x_{\mathbf{a}}x_{\mathbf{b}}\mid [\mathbf{a},\mathbf{b}]\in \mathcal {I}(\mathcal {C}\setminus \mathbf{v})\big)=\text{in}_{<_{\text{lex}}}(I_{\mathcal {C}\setminus \mathbf{v}})$. According to Theorem \ref{condition},  $\mathcal {C}\setminus \mathbf{v}$ belongs to type  $\mathcal{Q}_1$, as desired.

(2)  According to Remark \ref{graph} and Lemma \ref{cv-2}, we have
\begin{eqnarray*}
\text{in}_{<_{\text{lex}}}(I_{\mathcal {C}}):x_{\mathbf{v}} &=&(x_{\mathbf{a}}x_{\mathbf{b}}\mid [\mathbf{a},\mathbf{b}]\in \mathcal {I}(\mathcal {C}))+(x_{\mathbf{a}}\mid [\mathbf{a},\mathbf{v}]\in \mathcal {I}(\mathcal {C}))\\
&=&(x_{\mathbf{a}}x_{\mathbf{b}}\mid [\mathbf{a},\mathbf{b}]\in \mathcal {I}(\mathcal {C}) \text{ with } \mathbf{a}, \mathbf{b}\notin U_{\mathbf{v}}(\mathcal {C}))+(x_{\mathbf{a}}\mid \mathbf{a}\in U_{\mathbf{v}}(\mathcal {C}))\\
&=&(x_{\mathbf{a}}x_{\mathbf{b}}\mid [\mathbf{a},\mathbf{b}]\in \mathcal {I}(\mathcal {C}_{\mathbf{v}}))+(x_{\mathbf{a}}\mid \mathbf{a}\in U_{\mathbf{v}}(\mathcal {C}))\\
&\subseteq&\text{in}_{<_{\text{lex}}}(I_{\mathcal {C}_{\mathbf{v}}})+(x_{\mathbf{a}}\mid \mathbf{a}\in U_{\mathbf{v}}(\mathcal {C})).
\end{eqnarray*}
Note that $\mathcal {I}(\mathcal {C}_{\mathbf{v}})\subseteq \mathcal {I}(\mathcal {C})$, then $I_{\mathcal {C}_{\mathbf{v}}}\subseteq I_{\mathcal {C}}$. So  $\text{in}_{<_{\text{lex}}}(I_{\mathcal {C}_{\mathbf{v}}})\subseteq \text{in}_{<_{\text{lex}}}(I_{\mathcal {C}}) \subseteq \text{in}_{<_{\text{lex}}}(I_{\mathcal {C}}):x_{\mathbf{v}}$. Moreover, $(x_{\mathbf{a}}\mid \mathbf{a}\in U_{\mathbf{v}}(\mathcal {C}))\subseteq \text{in}_{<_{\text{lex}}}(I_{\mathcal {C}}):x_{\mathbf{v}}$ is obvious. Thus, the reverse inclusion ``$\supseteq$" holds. Furthermore,
$$(x_{\mathbf{a}}x_{\mathbf{b}}\mid [\mathbf{a},\mathbf{b}]\in \mathcal {I}(\mathcal {C}_{\mathbf{v}}))+(x_{\mathbf{a}}\mid \mathbf{a}\in U_{\mathbf{v}}(\mathcal {C}))
=\text{in}_{<_{\text{lex}}}(I_{\mathcal {C}_{\mathbf{v}}})+(x_{\mathbf{a}}\mid \mathbf{a}\in U_{\mathbf{v}}(\mathcal {C})).$$
Note that each element in the minimal monomial generator set of $\text{in}_{<_{\text{lex}}}(I_{\mathcal {C}_{\mathbf{v}}})$ is not divisible by any variable in $\big\{x_{\mathbf{a}}\mid \mathbf{a}\in U_{\mathbf{v}}(\mathcal {C})\big\}$, we deduce $(x_{\mathbf{a}}x_{\mathbf{b}}\mid [\mathbf{a},\mathbf{b}]\in \mathcal {I}(\mathcal {C}_{\mathbf{v}}))=\text{in}_{<_{\text{lex}}}(I_{\mathcal {C}_{\mathbf{v}}})$. According to Theorem \ref{condition},  $\mathcal {C}_{\mathbf{v}}$ belongs to type  $\mathcal{Q}_1$, as desired.
\end{proof}

\begin{prop} \label{induction-2}
Under Setup \ref{setup}, if $\mathcal {C}$ belongs to type $\mathcal{Q}_1$, then
$$|\mathcal {C}|=|\mathcal {C}_{\mathbf{v}}|+|U_{\mathbf{v}}(\mathcal {C})|.$$
\end{prop}

\begin{proof}
Let $T(\mathcal {C})$ and $T(\mathcal {C}_{\mathbf{v}})$ be the sets of the lower left corners of basic intervals of $\mathcal {C}$ and $\mathcal {C}_{\mathbf{v}}$ respectively. Note that $|T(\mathcal {C})|=|\mathcal {C}|$ and $|T(\mathcal {C}_{\mathbf{v}})|=|\mathcal {C}_{\mathbf{v}}|$. It suffices to show that $T(\mathcal {C})=T(\mathcal {C}_{\mathbf{v}})\cup U_{\mathbf{v}}(\mathcal {C})$ and $T(\mathcal {C}_{\mathbf{v}})\cap U_{\mathbf{v}}(\mathcal {C})=\emptyset$.

By the definitions of $\mathcal {C}_{\mathbf{v}}$ and $U_{\mathbf{v}}(\mathcal {C})$, we have
$T(\mathcal {C})\supseteq T(\mathcal {C}_{\mathbf{v}})\cup U_{\mathbf{v}}(\mathcal {C})$ and $T(\mathcal {C}_{\mathbf{v}})\cap U_{\mathbf{v}}(\mathcal {C})=\emptyset$. Now, let $\mathbf{a}\in T(\mathcal {C})\setminus U_{\mathbf{v}}(\mathcal {C})$, we will show that  $\mathbf{a}\in T(\mathcal {C}_{\mathbf{v}})$.

Assume that $I_{\mathbf{a}}(\mathcal {C})=[\mathbf{a},\mathbf{b}]$ and its upper left corner and lower right corner are $\mathbf{c}$ and $\mathbf{d}$ respectively. If $[\mathbf{a},\mathbf{b}]\in \mathcal {C}\setminus U_{\mathbf{v}}(\mathcal {C})$, then  $[\mathbf{a},\mathbf{b}]\in \mathcal {C}_{\mathbf{v}}$ by the definition of $\mathcal {C}_{\mathbf{v}}$. Consequently, $\mathbf{a}\in T(\mathcal {C}_{\mathbf{v}})$. Otherwise, $I_{\mathbf{a}}(\mathcal {C})\notin \mathcal {C}\setminus U_{\mathbf{v}}(\mathcal {C})$. This yields that at least one of $\mathbf{b}$, $\mathbf{c}$ and $\mathbf{d}$ belongs to $U_{\mathbf{v}}(\mathcal {C})$. We claim that $\mathbf{c}$ or $\mathbf{d}$ is in $U_{\mathbf{v}}(\mathcal {C})$. Otherwise, if $\mathbf{c}, \mathbf{d}\notin U_{\mathbf{v}}(\mathcal {C})$, then $\mathbf{b}\in U_{\mathbf{v}}(\mathcal {C})$. Hence $[\mathbf{b}, \mathbf{v}]$ is an inner interval of $\mathcal {C}$. Applying Theorem \ref{condition} to the inner intervals $[\mathbf{a},\mathbf{b}]$ and $[\mathbf{b}, \mathbf{v}]$, we obtain that $[\mathbf{c}, \mathbf{v}]$ or $[\mathbf{d}, \mathbf{v}]$ is an inner interval of $\mathcal {C}$, which contradicts the assumption  $\mathbf{c}, \mathbf{d}\notin U_{\mathbf{v}}(\mathcal {C})$.
Consider the following two cases.

(i) $\mathbf{c}\in U_{\mathbf{v}}(\mathcal {C})$. Note that $\mathbf{a}\notin U_{\mathbf{v}}(\mathcal {C})$. It follows from Observation \ref{ob-2} that $\mathbf{b}\in U_{\mathbf{v}}(\mathcal {C})$ and $\mathbf{d}\notin U_{\mathbf{v}}(\mathcal {C})$. Set $\mathbf{b}=(b_1,b_2)$ and $\mathbf{v}=(v_1, v_2)$. According to Remark \ref{T1} (1), we have $[\mathbf{a},(b_1,v_2)]\in \mathcal {T}_2\subseteq \mathcal {C}_{\mathbf{v}}$. Hence,  $\mathbf{a}\in T(\mathcal {C}_{\mathbf{v}})$.

(ii) $\mathbf{d}\in U_{\mathbf{v}}(\mathcal {C})$. Similar to (i), $\mathbf{a}\in T(\mathcal {C}_{\mathbf{v}})$.
\end{proof}


\begin{prop} \label{induction-3}
Under Setup \ref{setup}, if $\mathcal {C}$ belongs to type $\mathcal{Q}_1$, then $$\widetilde{r}_{\mathcal {C}}(t)=t\cdot\widetilde{r}_{\mathcal {C}_{\mathbf{v}}}(t)+\widetilde{r}_{\mathcal {C}\setminus \mathbf{v}}(t).$$
\end{prop}

\begin{proof} It is obvious that $\widetilde{r}_0({\mathcal {C}})=\widetilde{r}_0(\mathcal {C}\setminus {\mathbf{v}})=1$ and $\widetilde{r}_1({\mathcal {C}})=\widetilde{r}_1({\mathcal {C}} \setminus {\mathbf{v}})+\widetilde{r}_0(\mathcal {C}_{\mathbf{v}})=|\mathcal {C}|$. For each $k\in\{2,\ldots,r(\mathcal {C})\}$, we should show $\widetilde{r}_k({\mathcal {C}})=\widetilde{r}_k({\mathcal {C}} \setminus {\mathbf{v}})+\widetilde{r}_{k-1}(\mathcal {C}_{\mathbf{v}})$. Since  $\mathcal {C}\setminus \mathbf{v}$ is obtained from $\mathcal {C}$ by removing the interval $\overline{I_{\mathbf{v}}}(\mathcal {C})$, every $k$-rook configuration in $\mathcal {C}\setminus \mathbf{v}$ is also a $k$-rook configuration in $\mathcal {C}$, i.e., $\mathcal {R}_k(\mathcal {C}\setminus \mathbf{v})\subseteq \mathcal {R}_k(\mathcal {C})$. Consequently, the inclusion induces a map $\varphi_k:  \widetilde{\mathcal {R}}_k(\mathcal {C}\setminus \mathbf{v}) \to  \widetilde{\mathcal {R}}_k(\mathcal {C})$ given by $\widetilde{F} \mapsto \overline{F}$ for all $F \in \mathcal {R}_k(\mathcal {C}\setminus \mathbf{v})$, where $\widetilde{F}$ and $\overline{F}$ denote equivalence classes containing $F$ in $\widetilde{\mathcal {R}}_k(\mathcal {C}\setminus \mathbf{v})$ and $\widetilde{\mathcal {R}}_k(\mathcal {C})$, respectively. Let $\mathcal {R}'_k(\mathcal {C}) =\{F \in \mathcal {R}_k(\mathcal {C}) \,|\, F\overset{\scriptscriptstyle\mathcal{\mathcal {C}}}{\nsim} G \,\text{for each}\, G \in \mathcal {R}_k(\mathcal {C}\setminus \mathbf{v})\}$.
 It is direct to check that $\mathcal {R}'_k(\mathcal {C})/ \overset{\scriptscriptstyle\mathcal {C}}{\thicksim}\;= \widetilde{\mathcal {R}}_k(\mathcal {C}) \setminus \;\text{Im}\varphi_k$. In the following, we show that $|\text{Im}\varphi_k|=|\widetilde{\mathcal {R}}_k(\mathcal {C}\setminus \mathbf{v})|= \widetilde{r}_k({\mathcal {C}} \setminus {\mathbf{v}})$ and $|\mathcal {R}'_k(\mathcal {C})/ \overset{\scriptscriptstyle\mathcal {C}}{\thicksim}|= |\widetilde{\mathcal {R}}_{k-1}(\mathcal {C}_{\mathbf{v}})|=\widetilde{r}_{k-1}(\mathcal {C}_{\mathbf{v}})$. As a consequence,  $\widetilde{r}_k({\mathcal {C}})=|\widetilde{\mathcal {R}}_k(\mathcal {C})|=|\text{Im}\varphi_k|+|\mathcal {R}'_k(\mathcal {C})/ \overset{\scriptscriptstyle\mathcal {C}}{\thicksim}|=\widetilde{r}_k({\mathcal {C}} \setminus {\mathbf{v}})+\widetilde{r}_{k-1}(\mathcal {C}_{\mathbf{v}})$, which completes the proof.

{\bf Step 1:} We show that $\varphi_k$ is an injection. As a consequence, $|\text{Im}\varphi_k|=|\widetilde{\mathcal {R}}_k(\mathcal {C}\setminus \mathbf{v})|= \widetilde{r}_k({\mathcal {C}} \setminus {\mathbf{v}})$.

Suppose, to the contrary, that  $\varphi_k$ is not injective. Then there exists a pair of $k$-rook configurations in $\mathcal {C}\setminus \mathbf{v}$ that are not equivalent in $\mathcal {C}\setminus \mathbf{v}$ but equivalent in $\mathcal {C}$. Denote by $m_k(\mathcal {C}) = \min{\{n_{\mathcal {C}}(F_1, F_2)\,|\, F_1, F_2 \in \mathcal {R}_k(\mathcal {C\setminus \mathbf{v}}), F_1\overset{\scriptscriptstyle\mathcal{\mathcal {C}\setminus \mathbf{v}}}{\nsim} F_2, F_1\overset{\scriptscriptstyle\mathcal{C}}{\thicksim} F_2\}}$, and denote by $M_k(\mathcal {C}) = \{(F_1, F_2)\,|\, F_1, F_2 \in \mathcal {R}_k(\mathcal {C\setminus \mathbf{v}}), F_1\overset{\scriptscriptstyle\mathcal{\mathcal {C}\setminus \mathbf{v}}}{\nsim} F_2, n_{\mathcal {C}}(F_1, F_2) = m_k(\mathcal {C})\}$. Assume that $m_k(\mathcal {C})=t$. Choose a pair $(F, G) \in M_k(\mathcal {C})$.
 Consider the sequence of switches
\[
F  \to L_1 \to L_2 \to \cdots \to L_{t-1} \to L_t=G.
\]
If there exists some $1\leq j\leq t-1$ such that $L_j$ is a $k$-rook configuration in $\mathcal {C}\setminus \mathbf{v}$, then $F\overset{\scriptscriptstyle\mathcal{\mathcal {C}\setminus \mathbf{v}}}{\nsim} L_j$ or $L_j\overset{\scriptscriptstyle\mathcal{\mathcal {C}\setminus \mathbf{v}}}{\nsim} G$ by the assumption that $F\overset{\scriptscriptstyle\mathcal{\mathcal {C}\setminus \mathbf{v}}}{\nsim} G$. Assume without loss of generality that $F\overset{\scriptscriptstyle\mathcal{\mathcal {C}\setminus \mathbf{v}}}{\nsim} L_j$, then $n_{\mathcal {C}}(F, L_j)\leq j<t$, a contradiction. Thus, for all $1\leq j\leq t-1$, we have $L_j \in \mathcal {R}_k(\mathcal {C})\setminus \mathcal {R}_k(\mathcal {C}\setminus \mathbf{v})$, i.e., $L_j$ contains a rook placed on the basic interval $\overline{I_{\mathbf{v}}}(\mathcal {C})$. Consequently, $F$ contains two rooks $R_1$ and $R_2$ which are placed on the anti-diagonal basic intervals of an inner interval $[\mathbf{u},\mathbf{v}]$ of $\mathcal {C}$ for some $\mathbf{u}\in U_{\mathbf{v}}(\mathcal {C})$, and $L_1$ is obtained from $F$ by switching $R_1, R_2$ for $R'_1, R'_2$, where $R'_1$ and $R'_2$ are placed on $\overline{I_{\mathbf{v}}}(\mathcal {C})$ and $I_{\mathbf{u}}(\mathcal {C})$ respectively. Assume that $F=\{R_1,R_2, R_3,\ldots R_k\}$, then $L_1=\{R'_1,R'_2, R_3,\ldots R_k\}$. Consider the following three cases.

(1) $L_2$ is obtained from $L_1$ by switching $R_p, R_q$ for $R'_p, R'_q$ for some $3\leq p,q\leq  k$ with $p\neq q$. In this case, we construct a new sequence of switchs from $F$ to $G$:
\[
F  \to L'_1 \to L_2 \to \cdots \to L_{t-1} \to L_{t}=G,
\]
where $L'_1$ is obtained from $F$ by switching $R_p, R_q$ for $R'_p, R'_q$ and $L_2$ is obtained from $L'_1$ by switching $R_1, R_2$ for $R'_1, R'_2$. One can see that $L'_1\in \mathcal {R}_k(\mathcal {C}\setminus \mathbf{v})$ and $F\overset{\scriptscriptstyle\mathcal{\mathcal {C}\setminus \mathbf{v}}}{\sim} L'_1$. By $F\overset{\scriptscriptstyle\mathcal{\mathcal {C}\setminus \mathbf{v}}}{\nsim} G$, we have $L'_1\overset{\scriptscriptstyle\mathcal{\mathcal {C}\setminus \mathbf{v}}}{\nsim} G$. Note that $n_{\mathcal {C}}(L'_1, G)\leq t-1<t$, a contradiction.

(2) $L_2$ is obtained from $L_1$ by switching $R'_1, R_p$ for $R''_1, R''_p$ for some $3\leq p\leq  k$. It is clear that $L_2\in \mathcal {R}_k(\mathcal {C}\setminus \mathbf{v})$. Hence $L_2 = G$. Assume without loss of generality that $p=3$, then $G=\{R''_1,R'_2,R''_3, R_4,\ldots, R_k\}$, where $R''_1$ and $R''_3$ are in anti-diagonal position of $[\mathbf{a},\mathbf{v}]$ for some $\mathbf{a}\in U_{\mathbf{v}}(\mathcal {C})$ and $\mathbf{a}\neq \mathbf{u}$. Note that $R_3$ is placed on $I_{\mathbf{a}}(\mathcal {C})$.
The configuration of $R'_1, R'_2, R_3$ in $L_1$ must be one of the four cases as shown in Figure \ref{fig13}. Correspondingly, the configuration of $R_1, R_2, R_3$ in $F$ must be one of the four cases as shown in Figure \ref{fig14} and the configuration of  $R''_1, R'_2, R''_3$ in $G$ must be one of the four cases as shown in Figure \ref{fig15}. For Cases 1, 2 and 4 of $R_1, R_2, R_3$ in $F$,  one can construct a new sequence of switches:
\[
F  \;\xrightarrow{\text{switch $R_1, R_3$ for $R''_1, R^{\ast}_3$ }} \: \{R''_1, R_2, R^{\ast}_3, R_4,\ldots, R_k\}  \; \xrightarrow{{\text{switch $R_2, R^{\ast}_3$ for $R'_2, R''_3$ }}} \: G.
\]
It is easy to see that $\{R''_1, R_2, R^{\ast}_3, R_4,\ldots, R_k\}\in \mathcal {R}_k(\mathcal {C}\setminus \mathbf{v})$. Thus, $F\overset{\scriptscriptstyle\mathcal{\mathcal {C}\setminus \mathbf{v}}}{\sim} G$, a contradiction.  For Case 3 of $R_1, R_2, R_3$ in $F$,  it is easy observe that
\[
F  \;\xrightarrow{\text{switch $R_2,R_3$ for $R^{\ast}_2, R''_3$ }} \: \{R_1,R^{\ast}_2, R''_3,  R_4,\ldots, R_k\}  \; \xrightarrow{{\text{switch $R_1,R^{\ast}_2$ for $R''_1,R'_2$ }}} \: G,
\]
and $\{R_1,R^{\ast}_2, R''_3,  R_4,\ldots, R_k\}\in \mathcal {R}_k(\mathcal {C}\setminus \mathbf{v})$. Thus, $F\overset{\scriptscriptstyle\mathcal{\mathcal {C}\setminus \mathbf{v}}}{\sim} G$, a contradiction.
\begin{figure}[htbp]
\centering
\begin{minipage}{0.2\textwidth}
\centering
\resizebox{\linewidth}{!}{%
\begin{tikzpicture}[every node/.style={font=\small}]
  \draw[fill=black!16,line width=0.8pt]
    (0,0)--(3,0)--(3,3)--(0,3)--(0,0)--cycle;
  \draw[line width=0.8pt] (1,3)--(1,1)--(3,1) (0,0.5)--(0.5,0.5)--(0.5,0)  (1,1.5)--(1.5,1.5)--(1.5,1)  (2.5,3)--(2.5,2.5)--(3,2.5) ;

  \draw[fill=black] (0,0) circle (1pt);
  \draw[fill=black] (1,1) circle (1pt);
  \draw[fill=black] (3,3) circle (1pt);

 \node[above, font=\small,scale=0.6] at (3,3) {$\mathbf{v}$};
 \node[below, font=\small,scale=0.6] at (1,1) {$\mathbf{u}$};
 \node[below, font=\small,scale=0.6] at (0,0) {$\mathbf{a}$};

 \node[above, font=\small,scale=0.6] at (2.75,2.5) {$R'_1$};
 \node[below, font=\small,scale=0.6] at (1.25,1.4) {$R'_2$};
 \node[below, font=\small,scale=0.6] at (0.25,0.4) {$R_3$};

 \node[below, font=\small,scale=1] at (1.4,-0.25) {(a) Case 1, $L_1$};


\end{tikzpicture}}
\end{minipage}\hfill
\begin{minipage}{0.2\textwidth}
\centering
\resizebox{\linewidth}{!}{%
\begin{tikzpicture}[every node/.style={font=\small}]
\draw[fill=black!16,line width=0.8pt]
    (0,0)--(3,0)--(3,3)--(0,3)--(0,0)--cycle;
  \draw[line width=0.8pt] (1,3)--(1,1)--(3,1) (0,0.5)--(0.5,0.5)--(0.5,0)  (1,1.5)--(1.5,1.5)--(1.5,1)  (2.5,3)--(2.5,2.5)--(3,2.5) ;

  \draw[fill=black] (0,0) circle (1pt);
  \draw[fill=black] (1,1) circle (1pt);
  \draw[fill=black] (3,3) circle (1pt);

 \node[above, font=\small,scale=0.6] at (3,3) {$\mathbf{v}$};
 \node[below, font=\small,scale=0.6] at (1,1) {$\mathbf{a}$};
 \node[below, font=\small,scale=0.6] at (0,0) {$\mathbf{u}$};

 \node[above, font=\small,scale=0.6] at (2.75,2.5) {$R'_1$};
 \node[below, font=\small,scale=0.6] at (1.25,1.4) {$R_3$};
 \node[below, font=\small,scale=0.6] at (0.25,0.4) {$R'_2$};

 \node[below, font=\small,scale=1] at (1.4,-0.25) {(b) Case 2, $L_1$};

\end{tikzpicture}}
\end{minipage}\hfill
\begin{minipage}{0.2\textwidth}
\centering
\resizebox{\linewidth}{!}{%
\begin{tikzpicture}[every node/.style={font=\small}]
  \draw[fill=black!16,line width=0.8pt]
    (3,3)--(0,3)--(0,1.5)--(3,1.5)--(3,3)--cycle;
 \draw[fill=black!16,line width=0.8pt]
    (3,3)--(1.5,3)--(1.5,0)--(3,0)--(3,3)--cycle;

  \draw[line width=0.8pt]  (1.5,1.5)--(3,1.5)  (0,2)--(0.5,2)--(0.5,1.5)  (1.5,0.5)--(2,0.5)--(2,0)   (2.5,3)--(2.5,2.5)--(3,2.5) ;

  \draw[fill=black] (3,3) circle (1pt);
  \draw[fill=black] (0,1.5) circle (1pt);
  \draw[fill=black] (1.5,0) circle (1pt);

 \node[above, font=\small,scale=0.6] at (3,3) {$\mathbf{v}$};
 \node[below, font=\small,scale=0.6] at (0,1.5) {$\mathbf{u}$};
 \node[below, font=\small,scale=0.6] at (1.5,0) {$\mathbf{a}$};

 \node[above, font=\small,scale=0.6] at (2.75,2.5) {$R'_1$};
 \node[below, font=\small,scale=0.6] at (0.25,1.9) {$R'_2$};
 \node[below, font=\small,scale=0.6] at (1.75,0.4) {$R_3$};

 \node[below, font=\small,scale=1] at (1.4,-0.25) {(c) Case 3, $L_1$};


\end{tikzpicture}}
\end{minipage}\hfill
\begin{minipage}{0.2\textwidth}
\centering
\resizebox{\linewidth}{!}{%
\begin{tikzpicture}[every node/.style={font=\small}]
  \draw[fill=black!16,line width=0.8pt]
    (3,3)--(0,3)--(0,1.5)--(3,1.5)--(3,3)--cycle;
 \draw[fill=black!16,line width=0.8pt]
    (3,3)--(1.5,3)--(1.5,0)--(3,0)--(3,3)--cycle;

  \draw[line width=0.8pt]  (1.5,1.5)--(3,1.5)  (0,2)--(0.5,2)--(0.5,1.5)  (1.5,0.5)--(2,0.5)--(2,0)  (2.5,3)--(2.5,2.5)--(3,2.5) ;

  \draw[fill=black] (3,3) circle (1pt);
  \draw[fill=black] (0,1.5) circle (1pt);
  \draw[fill=black] (1.5,0) circle (1pt);

 \node[above, font=\small,scale=0.6] at (3,3) {$\mathbf{v}$};
 \node[below, font=\small,scale=0.6] at (0,1.5) {$\mathbf{a}$};
 \node[below, font=\small,scale=0.6] at (1.5,0) {$\mathbf{u}$};

 \node[above, font=\small,scale=0.6] at (2.75,2.5) {$R'_1$};
 \node[below, font=\small,scale=0.6] at (0.25,1.9) {$R_3$};
 \node[below, font=\small,scale=0.6] at (1.75,0.4) {$R'_2$};

 \node[below, font=\small,scale=1] at (1.4,-0.25) {(d) Case 4, $L_1$};


\end{tikzpicture}}
\end{minipage}\hfill

\caption{Configurations of $R'_1, R'_2, R_3$ in $L_1$}\label{fig13}
\end{figure}
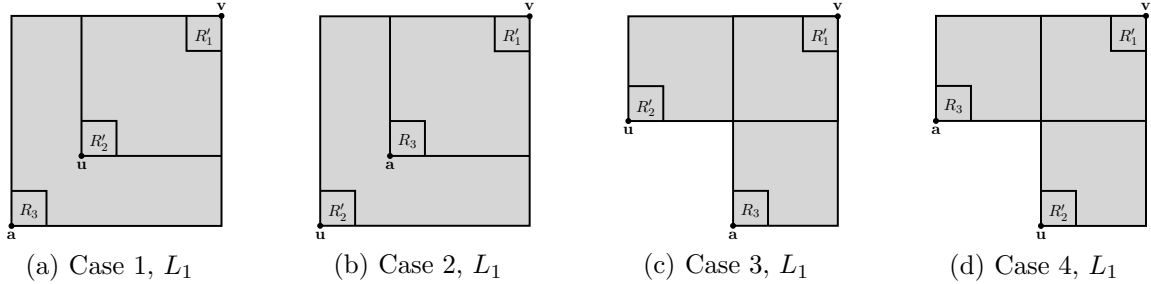

\begin{figure}[htbp]
\centering
\begin{minipage}{0.2\textwidth}
\centering
\resizebox{\linewidth}{!}{%
\begin{tikzpicture}[every node/.style={font=\small}]
\draw[fill=black!16,line width=0.8pt]
    (0,0)--(3,0)--(3,3)--(0,3)--(0,0)--cycle;
  \draw[line width=0.8pt] (1,3)--(1,1)--(3,1) (0,0.5)--(0.5,0.5)--(0.5,0)  (2.5,1)--(2.5,1.5)--(3,1.5)  (1.5,3)--(1.5,2.5)--(1,2.5) ;

  \draw[fill=black] (0,0) circle (1pt);
  \draw[fill=black] (1,1) circle (1pt);
  \draw[fill=black] (3,3) circle (1pt);

 \node[above, font=\small,scale=0.6] at (3,3) {$\mathbf{v}$};
 \node[below, font=\small,scale=0.6] at (1,1) {$\mathbf{a}$};
 \node[below, font=\small,scale=0.6] at (0,0) {$\mathbf{u}$};

 \node[above, font=\small,scale=0.6] at (1.25,2.5) {$R_1$};
 \node[below, font=\small,scale=0.6] at (2.75,1.4) {$R_2$};
 \node[below, font=\small,scale=0.6] at (0.25,0.4) {$R_3$};

 \node[below, font=\small,scale=1] at (1.4,-0.25) {(a) Case 1, $F$};


\end{tikzpicture}}
\end{minipage}\hfill
\begin{minipage}{0.2\textwidth}
\centering
\resizebox{\linewidth}{!}{%
\begin{tikzpicture}[every node/.style={font=\small}]
\draw[fill=black!16,line width=0.8pt]
    (0,0)--(3,0)--(3,3)--(0,3)--(0,0)--cycle;
  \draw[line width=0.8pt] (1,3)--(1,1)--(3,1) (0.5,3)--(0.5,2.5)--(0,2.5)  (1,1.5)--(1.5,1.5)--(1.5,1)  (2.5,0)--(2.5,0.5)--(3,0.5) ;

  \draw[fill=black] (0,0) circle (1pt);
  \draw[fill=black] (1,1) circle (1pt);
  \draw[fill=black] (3,3) circle (1pt);

 \node[above, font=\small,scale=0.6] at (3,3) {$\mathbf{v}$};
 \node[below, font=\small,scale=0.6] at (1,1) {$\mathbf{a}$};
 \node[below, font=\small,scale=0.6] at (0,0) {$\mathbf{u}$};

 \node[above, font=\small,scale=0.6] at (0.25,2.5) {$R_1$};
 \node[below, font=\small,scale=0.6] at (1.25,1.4) {$R_3$};
 \node[below, font=\small,scale=0.6] at (2.75,0.4) {$R_2$};

 \node[below, font=\small,scale=1] at (1.4,-0.25) {(b) Case 2, $F$};

\end{tikzpicture}}
\end{minipage}\hfill
\begin{minipage}{0.2\textwidth}
\centering
\resizebox{\linewidth}{!}{%
\begin{tikzpicture}[every node/.style={font=\small}]
  \draw[fill=black!16,line width=0.8pt]
    (3,3)--(0,3)--(0,1.5)--(3,1.5)--(3,3)--cycle;
 \draw[fill=black!16,line width=0.8pt]
    (3,3)--(1.5,3)--(1.5,0)--(3,0)--(3,3)--cycle;

  \draw[line width=0.8pt]  (1.5,1.5)--(3,1.5)   (0,2.5)--(0.5,2.5)--(0.5,3)  (1.5,0.5)--(2,0.5)--(2,0)   (2.5,1.5)--(2.5,2)--(3,2) ;

  \draw[fill=black] (3,3) circle (1pt);
  \draw[fill=black] (0,1.5) circle (1pt);
  \draw[fill=black] (1.5,0) circle (1pt);

 \node[above, font=\small,scale=0.6] at (3,3) {$\mathbf{v}$};
 \node[below, font=\small,scale=0.6] at (0,1.5) {$\mathbf{u}$};
 \node[below, font=\small,scale=0.6] at (1.5,0) {$\mathbf{a}$};

 \node[above, font=\small,scale=0.6] at (0.25,2.5) {$R_1$};
 \node[below, font=\small,scale=0.6] at (2.75,1.9) {$R_2$};
 \node[below, font=\small,scale=0.6] at (1.75,0.4) {$R_3$};

 \node[below, font=\small,scale=1] at (1.4,-0.25) {(c) Case 3, $F$};


\end{tikzpicture}}
\end{minipage}\hfill
\begin{minipage}{0.2\textwidth}
\centering
\resizebox{\linewidth}{!}{%
\begin{tikzpicture}[every node/.style={font=\small}]
  \draw[fill=black!16,line width=0.8pt]
    (3,3)--(0,3)--(0,1.5)--(3,1.5)--(3,3)--cycle;
 \draw[fill=black!16,line width=0.8pt]
    (3,3)--(1.5,3)--(1.5,0)--(3,0)--(3,3)--cycle;

  \draw[line width=0.8pt]  (1.5,1.5)--(3,1.5)  (0,2)--(0.5,2)--(0.5,1.5)  (2.5,0)--(2.5,0.5)--(3,0.5)  (2,3)--(2,2.5)--(1.5,2.5) ;

  \draw[fill=black] (3,3) circle (1pt);
  \draw[fill=black] (0,1.5) circle (1pt);
  \draw[fill=black] (1.5,0) circle (1pt);

 \node[above, font=\small,scale=0.6] at (3,3) {$\mathbf{v}$};
 \node[below, font=\small,scale=0.6] at (0,1.5) {$\mathbf{a}$};
 \node[below, font=\small,scale=0.6] at (1.5,0) {$\mathbf{u}$};

 \node[above, font=\small,scale=0.6] at (1.75,2.5) {$R_1$};
 \node[below, font=\small,scale=0.6] at (0.25,1.9) {$R_3$};
 \node[below, font=\small,scale=0.6] at (2.75,0.4) {$R_2$};

 \node[below, font=\small,scale=1] at (1.4,-0.25) {(d) Case 4, $F$};


\end{tikzpicture}}
\end{minipage}\hfill

\caption{Configurations of $R_1, R_2, R_3$ in $F$}\label{fig14}
\end{figure}
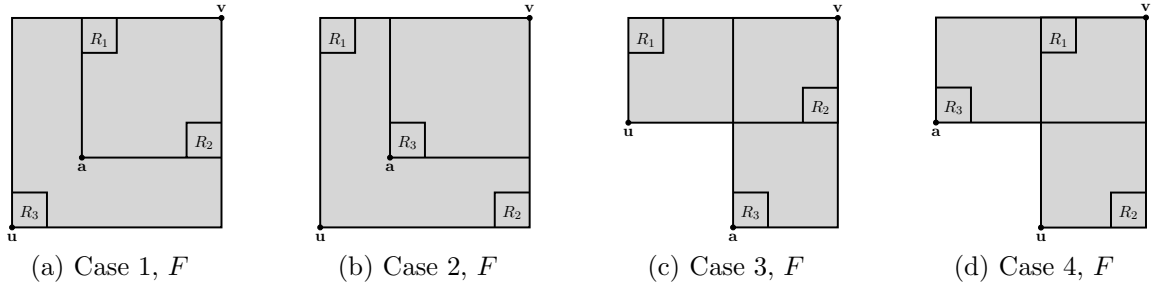

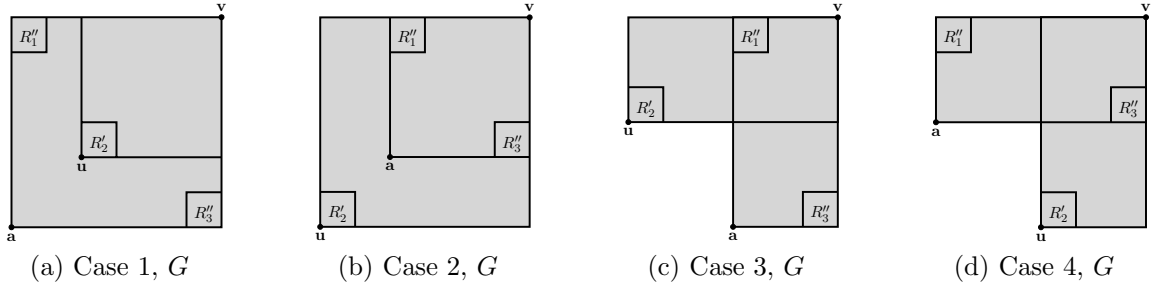
\begin{figure}[htbp]
\centering
\begin{minipage}{0.2\textwidth}
\centering
\resizebox{\linewidth}{!}{%
\begin{tikzpicture}[every node/.style={font=\small}]
  \draw[fill=black!16,line width=0.8pt]
    (0,0)--(3,0)--(3,3)--(0,3)--(0,0)--cycle;
  \draw[line width=0.8pt] (1,3)--(1,1)--(3,1)   (0,2.5)--(0.5,2.5)--(0.5,3)   (1,1.5)--(1.5,1.5)--(1.5,1)  (2.5,0)--(2.5,0.5)--(3,0.5) ;

  \draw[fill=black] (0,0) circle (1pt);
  \draw[fill=black] (1,1) circle (1pt);
  \draw[fill=black] (3,3) circle (1pt);

 \node[above, font=\small,scale=0.6] at (3,3) {$\mathbf{v}$};
 \node[below, font=\small,scale=0.6] at (1,1) {$\mathbf{u}$};
 \node[below, font=\small,scale=0.6] at (0,0) {$\mathbf{a}$};

 \node[above, font=\small,scale=0.6] at (0.25,2.5) {$R''_1$};
 \node[below, font=\small,scale=0.6] at (1.25,1.4) {$R'_2$};
 \node[below, font=\small,scale=0.6] at (2.75,0.4) {$R''_3$};

 \node[below, font=\small,scale=1] at (1.4,-0.25) {(a) Case 1, $G$};


\end{tikzpicture}}
\end{minipage}\hfill
\begin{minipage}{0.2\textwidth}
\centering
\resizebox{\linewidth}{!}{%
\begin{tikzpicture}[every node/.style={font=\small}]
\draw[fill=black!16,line width=0.8pt]
    (0,0)--(3,0)--(3,3)--(0,3)--(0,0)--cycle;
  \draw[line width=0.8pt] (1,3)--(1,1)--(3,1) (0,0.5)--(0.5,0.5)--(0.5,0)  (2.5,1)--(2.5,1.5)--(3,1.5)  (1.5,3)--(1.5,2.5)--(1,2.5);

  \draw[fill=black] (0,0) circle (1pt);
  \draw[fill=black] (1,1) circle (1pt);
  \draw[fill=black] (3,3) circle (1pt);

 \node[above, font=\small,scale=0.6] at (3,3) {$\mathbf{v}$};
 \node[below, font=\small,scale=0.6] at (1,1) {$\mathbf{a}$};
 \node[below, font=\small,scale=0.6] at (0,0) {$\mathbf{u}$};

 \node[above, font=\small,scale=0.6] at (1.25,2.5) {$R''_1$};
 \node[below, font=\small,scale=0.6] at (2.75,1.4) {$R''_3$};
 \node[below, font=\small,scale=0.6] at (0.25,0.4) {$R'_2$};

 \node[below, font=\small,scale=1] at (1.4,-0.25) {(b) Case 2, $G$};

\end{tikzpicture}}
\end{minipage}\hfill
\begin{minipage}{0.2\textwidth}
\centering
\resizebox{\linewidth}{!}{%
\begin{tikzpicture}[every node/.style={font=\small}]
  \draw[fill=black!16,line width=0.8pt]
    (3,3)--(0,3)--(0,1.5)--(3,1.5)--(3,3)--cycle;
 \draw[fill=black!16,line width=0.8pt]
    (3,3)--(1.5,3)--(1.5,0)--(3,0)--(3,3)--cycle;

  \draw[line width=0.8pt]  (1.5,1.5)--(3,1.5)  (0,2)--(0.5,2)--(0.5,1.5)  (2.5,0)--(2.5,0.5)--(3,0.5)  (2,3)--(2,2.5)--(1.5,2.5);

  \draw[fill=black] (3,3) circle (1pt);
  \draw[fill=black] (0,1.5) circle (1pt);
  \draw[fill=black] (1.5,0) circle (1pt);

 \node[above, font=\small,scale=0.6] at (3,3) {$\mathbf{v}$};
 \node[below, font=\small,scale=0.6] at (0,1.5) {$\mathbf{u}$};
 \node[below, font=\small,scale=0.6] at (1.5,0) {$\mathbf{a}$};

 \node[above, font=\small,scale=0.6] at (1.75,2.5) {$R''_1$};
 \node[below, font=\small,scale=0.6] at (0.25,1.9) {$R'_2$};
 \node[below, font=\small,scale=0.6] at (2.75,0.4) {$R''_3$};

 \node[below, font=\small,scale=1] at (1.4,-0.25) {(c) Case 3, $G$};


\end{tikzpicture}}
\end{minipage}\hfill
\begin{minipage}{0.2\textwidth}
\centering
\resizebox{\linewidth}{!}{%
\begin{tikzpicture}[every node/.style={font=\small}]
  \draw[fill=black!16,line width=0.8pt]
    (3,3)--(0,3)--(0,1.5)--(3,1.5)--(3,3)--cycle;
 \draw[fill=black!16,line width=0.8pt]
    (3,3)--(1.5,3)--(1.5,0)--(3,0)--(3,3)--cycle;

  \draw[line width=0.8pt]  (1.5,1.5)--(3,1.5)  (0,2.5)--(0.5,2.5)--(0.5,3)  (1.5,0.5)--(2,0.5)--(2,0)   (2.5,1.5)--(2.5,2)--(3,2) ;

  \draw[fill=black] (3,3) circle (1pt);
  \draw[fill=black] (0,1.5) circle (1pt);
  \draw[fill=black] (1.5,0) circle (1pt);

 \node[above, font=\small,scale=0.6] at (3,3) {$\mathbf{v}$};
 \node[below, font=\small,scale=0.6] at (0,1.5) {$\mathbf{a}$};
 \node[below, font=\small,scale=0.6] at (1.5,0) {$\mathbf{u}$};

 \node[above, font=\small,scale=0.6] at (0.25,2.5) {$R''_1$};
 \node[below, font=\small,scale=0.6] at (2.75,1.9) {$R''_3$};
 \node[below, font=\small,scale=0.6] at (1.75,0.4) {$R'_2$};

 \node[below, font=\small,scale=1] at (1.4,-0.25) {(d) Case 4, $G$};


\end{tikzpicture}}
\end{minipage}\hfill

\caption{Configurations of $R''_1, R'_2, R''_3$ in $G$}\label{fig15}
\end{figure}

(3) $L_2$ is obtained from $L_1$ by switching $R'_2, R_p$ for $R''_2, R''_p$ for some $3\leq p\leq  k$. Assume without loss of genrality that $p=3$, then  $L_2=\{R'_1,R''_2,R''_3, R_4,\ldots, R_k\}$.
 Assume that $R_3$ is place on a basic interval $I_{\mathbf{a}}(\mathcal {C})$ for some $\mathbf{a}\in \mathcal {C}\setminus \mathbf{v}$. Based on the above discussion, the configuration of $R'_1, R'_2, R_3$ in $L_1$ must be one of the six cases as shown in Figure \ref{fig16}. Correspondingly, the configuration of  $R_1, R_2, R_3$ in $F$ must be one of the six cases as shown in Figure \ref{fig17} and the configuration of  $R'_1, R''_2, R''_3$ in $L_2$ must be one of the six cases as shown in Figure \ref{fig18}. For Case 1 of $R'_1, R'_2, R_3$ in $L_1$, applying Theorem \ref{condition}, we obtain $[\mathbf{a},\mathbf{u}_1]$ or $[\mathbf{a},\mathbf{u}_2]$ is an inner interval of $\mathcal {C}$, where $\mathbf{u}_1$ and $\mathbf{u}_2$ are shown in  Figure \ref{fig16} (a). Without loss of generality, we assume that  $[\mathbf{a},\mathbf{u}_1]$ is an inner interval of $\mathcal {C}$. For Cases 1, 2 and 3 of $R_1, R_2, R_3$ in $F$  one can construct a new sequence of switches:
\[
F  \; \xrightarrow{\text{switch $R_1, R_3$ for $R^{\ast}_1, R''_3$ }}  \; L'_1   \; \xrightarrow{{\text{switch $R^{\ast}_1, R_2$ for $R'_1, R''_2$}}}  \; L_2  \;\to\cdots\to \; L_t=G,
\]
where $L'_1=\{R^{\ast}_1, R_2, R''_3, R_4,\ldots, R_k\}$.
It is easy to see that $L'_1\in \mathcal {R}_k(\mathcal {C}\setminus \mathbf{v})$.
By $F\overset{\scriptscriptstyle\mathcal {C}\setminus \mathbf{v}}{\nsim} G$, we have $L'_1\overset{\scriptscriptstyle\mathcal {C}\setminus \mathbf{v}}{\nsim} G$. Note that $n_{\mathcal {C}}(L'_1, G)\leq t-1<t$, a contradiction.  For Cases 4, 5 and 6 of $R_1, R_2, R_3$ in $F$,  one can construct a new sequence of switches:
\[
F  \;\xrightarrow{\text{switch $R_2,R_3$ for $R''_2,R^{\ast}_3$}}  \;L''_1 \; \xrightarrow{{\text{switch $R_1,R^{\ast}_3$ for $R'_1,R''_3$}}} \; L_2 \; \to\cdots\to \;L_t= G,
\]
where $L''_1=\{R_1,R''_2, R^{\ast}_3,  R_4,\ldots, R_k\}$. It is easy to see that $L''_1\in \mathcal {R}_k(\mathcal {C}\setminus \mathbf{v})$.
By $F\overset{\scriptscriptstyle\mathcal {C}\setminus \mathbf{v}}{\nsim} G$, we have $L''_1\overset{\scriptscriptstyle\mathcal {C}\setminus \mathbf{v}}{\nsim} G$. Note that $n_{\mathcal {C}}(L''_1, G)\leq t-1<t$, a contradiction.

All possibilities lead to contradictions, completing the proof that $\varphi_k$ is an injection.

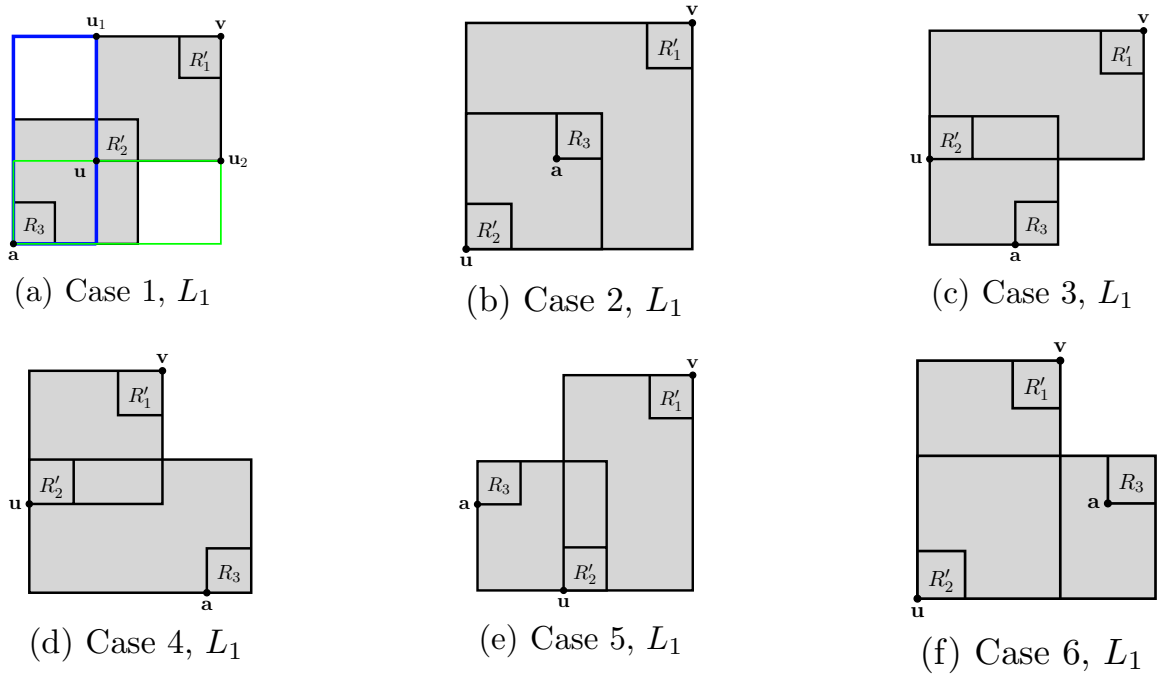
\begin{figure}[htbp]
\centering
\begin{minipage}{0.22\textwidth}
\centering
\resizebox{\linewidth}{!}{%
\begin{tikzpicture}[every node/.style={font=\small}]
  \draw[fill=black!16,line width=0.8pt]
    (0,0)--(1.5,0)--(1.5,1.5)--(0,1.5)--(0,0)--cycle;
  \draw[fill=black!16,line width=0.8pt]
    (1,1)--(2.5,1)--(2.5,2.5)--(1,2.5)--(1,1)--cycle;
  \draw[line width=0.8pt] (0,0.5)--(0.5,0.5)--(0.5,0)  (1,1.5)--(1.5,1.5)--(1.5,1)  (2,2.5)--(2,2)--(2.5,2) ;

 \textcolor[rgb]{0.00,0.07,1.00}{ \draw[line width=1.2pt] (1,0)--(0,0)--(0,2.5)--(1,2.5)--(1,0); }
 \textcolor[rgb]{0.00,1.00,0.00}{ \draw[line width=0.6pt] (0,1)--(0,0)--(2.5,0)--(2.5,1)--(0,1); }

  \draw[fill=black] (0,0) circle (1pt);
  \draw[fill=black] (1,1) circle (1pt);
  \draw[fill=black] (2.5,2.5) circle (1pt);
    \draw[fill=black] (1,2.5) circle (1pt);
        \draw[fill=black] (2.5,1) circle (1pt);

 \node[above, font=\small,scale=0.6] at (2.5,2.5) {$\mathbf{v}$};
  \node[above, font=\small,scale=0.6] at (1,2.5) {$\mathbf{u}_1$};
    \node[right, font=\small,scale=0.6] at (2.5,1) {$\mathbf{u}_2$};
 \node[below, font=\small,scale=0.6] at (0.8,1) {$\mathbf{u}$};
 \node[below, font=\small,scale=0.6] at (0,0) {$\mathbf{a}$};

 \node[above, font=\small,scale=0.6] at (2.25,2) {$R'_1$};
 \node[below, font=\small,scale=0.6] at (1.25,1.4) {$R'_2$};
 \node[below, font=\small,scale=0.6] at (0.25,0.4) {$R_3$};

 \node[below, font=\small,scale=1] at (1.2,-0.25) {(a) Case 1, $L_1$};


\end{tikzpicture}}
\end{minipage}\hfill
\begin{minipage}{0.22\textwidth}
\centering
\resizebox{\linewidth}{!}{%
\begin{tikzpicture}[every node/.style={font=\small}]
  \draw[fill=black!16,line width=0.8pt]
    (0,0)--(2.5,0)--(2.5,2.5)--(0,2.5)--(0,0)--cycle;
  \draw[fill=black!16,line width=0.8pt]
   (0,0)--(1.5,0)--(1.5,1.5)--(0,1.5)--(0,0)--cycle;
  \draw[line width=0.8pt] (0,0.5)--(0.5,0.5)--(0.5,0)  (1,1.5)--(1,1)--(1.5,1)  (2,2.5)--(2,2)--(2.5,2) ;

  \draw[fill=black] (0,0) circle (1pt);
  \draw[fill=black] (1,1) circle (1pt);
  \draw[fill=black] (2.5,2.5) circle (1pt);

 \node[above, font=\small,scale=0.6] at (2.5,2.5) {$\mathbf{v}$};
 \node[below, font=\small,scale=0.6] at (1,1) {$\mathbf{a}$};
 \node[below, font=\small,scale=0.6] at (0,0) {$\mathbf{u}$};

 \node[above, font=\small,scale=0.6] at (2.25,2) {$R'_1$};
 \node[below, font=\small,scale=0.6] at (1.25,1.4) {$R_3$};
 \node[below, font=\small,scale=0.6] at (0.25,0.4) {$R'_2$};

 \node[below, font=\small,scale=1] at (1.2,-0.25) {(b) Case 2, $L_1$};

\end{tikzpicture}}
\end{minipage}\hfill
\begin{minipage}{0.22\textwidth}
\centering
\resizebox{\linewidth}{!}{%
\begin{tikzpicture}[every node/.style={font=\small}]
  \draw[fill=black!16,line width=0.8pt]
    (0,1)--(2.5,1)--(2.5,2.5)--(0,2.5)--(0,1)--cycle;
  \draw[fill=black!16,line width=0.8pt]
   (0,0)--(1.5,0)--(1.5,1.5)--(0,1.5)--(0,0)--cycle;
  \draw[line width=0.8pt] (0,1)--(2.5,1) (1,0)--(1,0.5)--(1.5,0.5)  (0.5,1.5)--(0.5,1)  (2,2.5)--(2,2)--(2.5,2) ;

  \draw[fill=black] (1,0) circle (1pt);
  \draw[fill=black] (0,1) circle (1pt);
  \draw[fill=black] (2.5,2.5) circle (1pt);

 \node[above, font=\small,scale=0.6] at (2.5,2.5) {$\mathbf{v}$};
 \node[left, font=\small,scale=0.6] at (0,1) {$\mathbf{u}$};
 \node[below, font=\small,scale=0.6] at (1,0) {$\mathbf{a}$};

 \node[above, font=\small,scale=0.6] at (2.25,2) {$R'_1$};
 \node[below, font=\small,scale=0.6] at (0.25,1.4) {$R'_2$};
 \node[below, font=\small,scale=0.6] at (1.25,0.4) {$R_3$};

 \node[below, font=\small,scale=1] at (1.2,-0.25) {(c) Case 3, $L_1$};


\end{tikzpicture}}
\end{minipage}\hfill

\begin{minipage}{0.22\textwidth}
\centering
\resizebox{\linewidth}{!}{%
\begin{tikzpicture}[every node/.style={font=\small}]
  \draw[fill=black!16,line width=0.8pt]
    (0,1)--(1.5,1)--(1.5,2.5)--(0,2.5)--(0,1)--cycle;
  \draw[fill=black!16,line width=0.8pt]
   (0,0)--(2.5,0)--(2.5,1.5)--(0,1.5)--(0,0)--cycle;
  \draw[line width=0.8pt] (0,1)--(1.5,1)--(1.5,1.5) (2,0)--(2,0.5)--(2.5,0.5)  (0.5,1.5)--(0.5,1)  (1,2.5)--(1,2)--(1.5,2) ;

  \draw[fill=black] (2,0) circle (1pt);
  \draw[fill=black] (0,1) circle (1pt);
  \draw[fill=black] (1.5,2.5) circle (1pt);

 \node[above, font=\small,scale=0.6] at (1.5,2.5) {$\mathbf{v}$};
 \node[left, font=\small,scale=0.6] at (0,1) {$\mathbf{u}$};
 \node[below, font=\small,scale=0.6] at (2,0) {$\mathbf{a}$};

 \node[above, font=\small,scale=0.6] at (1.25,2) {$R'_1$};
 \node[below, font=\small,scale=0.6] at (0.25,1.4) {$R'_2$};
 \node[below, font=\small,scale=0.6] at (2.25,0.4) {$R_3$};

 \node[below, font=\small,scale=1] at (1.2,-0.25) {(d) Case 4, $L_1$};


\end{tikzpicture}}
\end{minipage}\hfill
\begin{minipage}{0.22\textwidth}
\centering
\resizebox{\linewidth}{!}{%
\begin{tikzpicture}[every node/.style={font=\small}]
  \draw[fill=black!16,line width=0.8pt]
    (0,0)--(1.5,0)--(1.5,1.5)--(0,1.5)--(0,0)--cycle;
  \draw[fill=black!16,line width=0.8pt]
   (1,0)--(2.5,0)--(2.5,2.5)--(1,2.5)--(1,0)--cycle;
  \draw[line width=0.8pt] (1.5,0)--(1.5,1.5)--(1,1.5) (1,0.5)--(1.5,0.5)  (0,1)--(0.5,1)--(0.5,1.5)  (2,2.5)--(2,2)--(2.5,2) ;

  \draw[fill=black] (1,0) circle (1pt);
  \draw[fill=black] (0,1) circle (1pt);
  \draw[fill=black] (2.5,2.5) circle (1pt);

 \node[above, font=\small,scale=0.6] at (2.5,2.5) {$\mathbf{v}$};
 \node[below, font=\small,scale=0.6] at (1,0) {$\mathbf{u}$};
 \node[left, font=\small,scale=0.6] at (0,1) {$\mathbf{a}$};

 \node[above, font=\small,scale=0.6] at (2.25,2) {$R'_1$};
 \node[below, font=\small,scale=0.6] at (0.25,1.4) {$R_3$};
 \node[below, font=\small,scale=0.6] at (1.25,0.4) {$R'_2$};

 \node[below, font=\small,scale=1] at (1.2,-0.25) {(e) Case 5, $L_1$};


\end{tikzpicture}}
\end{minipage}\hfill
\begin{minipage}{0.22\textwidth}
\centering
\resizebox{\linewidth}{!}{%
\begin{tikzpicture}[every node/.style={font=\small}]
  \draw[fill=black!16,line width=0.8pt]
    (0,0)--(1.5,0)--(1.5,2.5)--(0,2.5)--(0,0)--cycle;
  \draw[fill=black!16,line width=0.8pt]
   (0,0)--(2.5,0)--(2.5,1.5)--(0,1.5)--(0,0)--cycle;
  \draw[line width=0.8pt] (1.5,1.5)--(1.5,0) (0,0.5)--(0.5,0.5)--(0.5,0)  (2,1.5)--(2,1)--(2.5,1) (1,2.5)--(1,2)--(1.5,2);

  \draw[fill=black] (0,0) circle (1pt);
  \draw[fill=black] (1.5,2.5) circle (1pt);
  \draw[fill=black] (2,1) circle (1pt);

 \node[above, font=\small,scale=0.6] at (1.5,2.5) {$\mathbf{v}$};
 \node[below, font=\small,scale=0.6] at (0,0) {$\mathbf{u}$};
 \node[left, font=\small,scale=0.6] at (2,1) {$\mathbf{a}$};

 \node[above, font=\small,scale=0.6] at (1.25,2) {$R'_1$};
 \node[below, font=\small,scale=0.6] at (2.25,1.4) {$R_3$};
 \node[below, font=\small,scale=0.6] at (0.25,0.4) {$R'_2$};

 \node[below, font=\small,scale=1] at (1.2,-0.25) {(f) Case 6, $L_1$};


\end{tikzpicture}}
\end{minipage}\hfill
\caption{Configurations of $R'_1, R'_2, R_3$ in $L_1$}\label{fig16}
\end{figure}

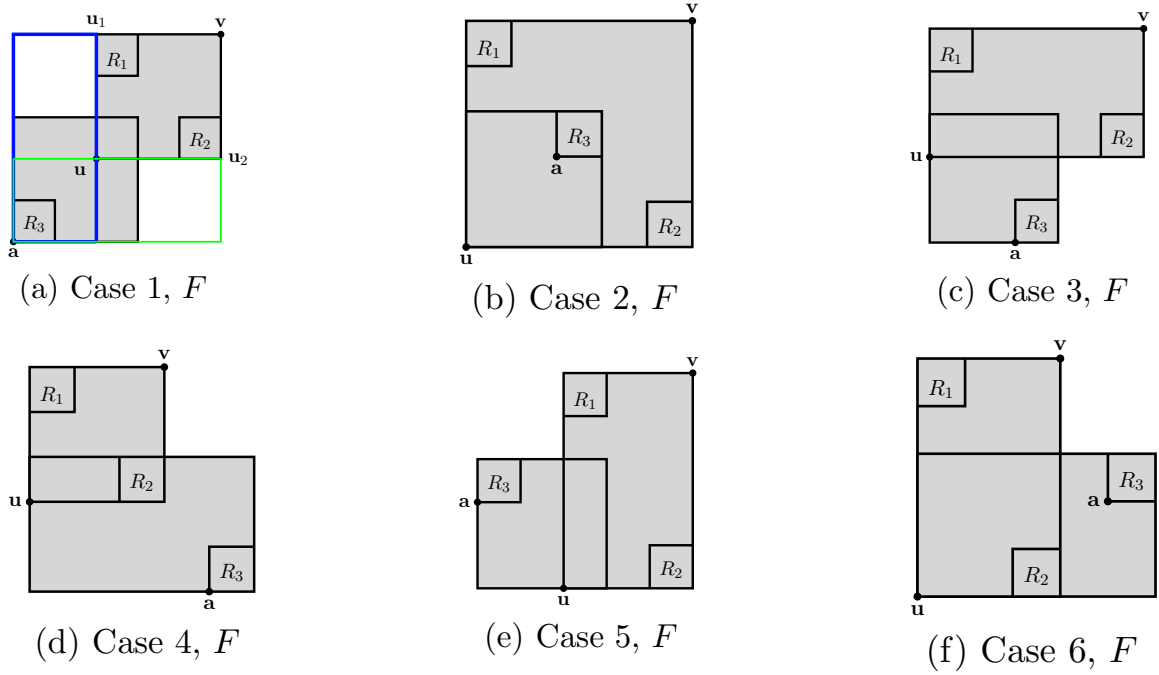
\begin{figure}[htbp]
\centering
\begin{minipage}{0.22\textwidth}
\centering
\resizebox{\linewidth}{!}{%
\begin{tikzpicture}[every node/.style={font=\small}]
  \draw[fill=black!16,line width=0.8pt]
    (0,0)--(1.5,0)--(1.5,1.5)--(0,1.5)--(0,0)--cycle;
  \draw[fill=black!16,line width=0.8pt]
    (1,1)--(2.5,1)--(2.5,2.5)--(1,2.5)--(1,1)--cycle;
  \draw[line width=0.8pt] (1,1.5)--(1.5,1.5)--(1.5,1) (0,0.5)--(0.5,0.5)--(0.5,0)  (2,1)--(2,1.5)--(2.5,1.5)  (1.5,2.5)--(1.5,2)--(1,2) ;

  \draw[fill=black] (0,0) circle (1pt);
    \node[above, font=\small,scale=0.6] at (1,2.5) {$\mathbf{u}_1$};
    \node[right, font=\small,scale=0.6] at (2.5,1) {$\mathbf{u}_2$};
  \draw[fill=black] (1,1) circle (1pt);
  \draw[fill=black] (2.5,2.5) circle (1pt);

 \textcolor[rgb]{0.00,0.07,1.00}{ \draw[line width=1.2pt] (1,0)--(0,0)--(0,2.5)--(1,2.5)--(1,0); }
 \textcolor[rgb]{0.00,1.00,0.00}{ \draw[line width=0.6pt] (0,1)--(0,0)--(2.5,0)--(2.5,1)--(0,1); }

 \node[above, font=\small,scale=0.6] at (2.5,2.5) {$\mathbf{v}$};
 \node[below, font=\small,scale=0.6] at (0.8,1) {$\mathbf{u}$};
 \node[below, font=\small,scale=0.6] at (0,0) {$\mathbf{a}$};

 \node[above, font=\small,scale=0.6] at (1.25,2) {$R_1$};
 \node[below, font=\small,scale=0.6] at (2.25,1.4) {$R_2$};
 \node[below, font=\small,scale=0.6] at (0.25,0.4) {$R_3$};

 \node[below, font=\small,scale=1] at (1.2,-0.25) {(a) Case 1, $F$};


\end{tikzpicture}}
\end{minipage}\hfill
\begin{minipage}{0.22\textwidth}
\centering
\resizebox{\linewidth}{!}{%
\begin{tikzpicture}[every node/.style={font=\small}]
  \draw[fill=black!16,line width=0.8pt]
    (0,0)--(2.5,0)--(2.5,2.5)--(0,2.5)--(0,0)--cycle;
  \draw[fill=black!16,line width=0.8pt]
   (0,0)--(1.5,0)--(1.5,1.5)--(0,1.5)--(0,0)--cycle;
  \draw[line width=0.8pt]  (2,0)--(2,0.5)--(2.5,0.5)  (1,1.5)--(1,1)--(1.5,1)  (0,2)--(0.5,2)--(0.5,2.5) ;

  \draw[fill=black] (0,0) circle (1pt);
  \draw[fill=black] (1,1) circle (1pt);
  \draw[fill=black] (2.5,2.5) circle (1pt);

 \node[above, font=\small,scale=0.6] at (2.5,2.5) {$\mathbf{v}$};
 \node[below, font=\small,scale=0.6] at (1,1) {$\mathbf{a}$};
 \node[below, font=\small,scale=0.6] at (0,0) {$\mathbf{u}$};

 \node[above, font=\small,scale=0.6] at (0.25,2) {$R_1$};
 \node[below, font=\small,scale=0.6] at (1.25,1.4) {$R_3$};
 \node[below, font=\small,scale=0.6] at (2.25,0.4) {$R_2$};

 \node[below, font=\small,scale=1] at (1.2,-0.25) {(b) Case 2, $F$};

\end{tikzpicture}}
\end{minipage}\hfill
\begin{minipage}{0.22\textwidth}
\centering
\resizebox{\linewidth}{!}{%
\begin{tikzpicture}[every node/.style={font=\small}]
  \draw[fill=black!16,line width=0.8pt]
    (0,1)--(2.5,1)--(2.5,2.5)--(0,2.5)--(0,1)--cycle;
  \draw[fill=black!16,line width=0.8pt]
   (0,0)--(1.5,0)--(1.5,1.5)--(0,1.5)--(0,0)--cycle;
  \draw[line width=0.8pt]  (0,1)--(1.5,1) (1,0)--(1,0.5)--(1.5,0.5)  (2,1)--(2,1.5)--(2.5,1.5)  (0,2)--(0.5,2)--(0.5,2.5);

  \draw[fill=black] (1,0) circle (1pt);
  \draw[fill=black] (0,1) circle (1pt);
  \draw[fill=black] (2.5,2.5) circle (1pt);

 \node[above, font=\small,scale=0.6] at (2.5,2.5) {$\mathbf{v}$};
 \node[left, font=\small,scale=0.6] at (0,1) {$\mathbf{u}$};
 \node[below, font=\small,scale=0.6] at (1,0) {$\mathbf{a}$};

 \node[above, font=\small,scale=0.6] at (0.25,2) {$R_1$};
 \node[below, font=\small,scale=0.6] at (2.25,1.4) {$R_2$};
 \node[below, font=\small,scale=0.6] at (1.25,0.4) {$R_3$};

 \node[below, font=\small,scale=1] at (1.2,-0.25) {(c) Case 3, $F$};


\end{tikzpicture}}
\end{minipage}\hfill

\begin{minipage}{0.22\textwidth}
\centering
\resizebox{\linewidth}{!}{%
\begin{tikzpicture}[every node/.style={font=\small}]
  \draw[fill=black!16,line width=0.8pt]
    (0,1)--(1.5,1)--(1.5,2.5)--(0,2.5)--(0,1)--cycle;
  \draw[fill=black!16,line width=0.8pt]
   (0,0)--(2.5,0)--(2.5,1.5)--(0,1.5)--(0,0)--cycle;
  \draw[line width=0.8pt] (0,1)--(1.5,1)--(1.5,1.5) (2,0)--(2,0.5)--(2.5,0.5)  (1,1.5)--(1,1)  (0,2)--(0.5,2)--(0.5,2.5);

  \draw[fill=black] (2,0) circle (1pt);
  \draw[fill=black] (0,1) circle (1pt);
  \draw[fill=black] (1.5,2.5) circle (1pt);

 \node[above, font=\small,scale=0.6] at (1.5,2.5) {$\mathbf{v}$};
 \node[left, font=\small,scale=0.6] at (0,1) {$\mathbf{u}$};
 \node[below, font=\small,scale=0.6] at (2,0) {$\mathbf{a}$};

 \node[above, font=\small,scale=0.6] at (0.25,2) {$R_1$};
 \node[below, font=\small,scale=0.6] at (1.25,1.4) {$R_2$};
 \node[below, font=\small,scale=0.6] at (2.25,0.4) {$R_3$};

 \node[below, font=\small,scale=1] at (1.2,-0.25) {(d) Case 4, $F$};


\end{tikzpicture}}
\end{minipage}\hfill
\begin{minipage}{0.22\textwidth}
\centering
\resizebox{\linewidth}{!}{%
\begin{tikzpicture}[every node/.style={font=\small}]
  \draw[fill=black!16,line width=0.8pt]
    (0,0)--(1.5,0)--(1.5,1.5)--(0,1.5)--(0,0)--cycle;
  \draw[fill=black!16,line width=0.8pt]
   (1,0)--(2.5,0)--(2.5,2.5)--(1,2.5)--(1,0)--cycle;
  \draw[line width=0.8pt] (1.5,0)--(1.5,1.5)--(1,1.5) (2,0)--(2,0.5)--(2.5,0.5) (0,1)--(0.5,1)--(0.5,1.5)  (1.5,2.5)--(1.5,2)--(1,2) ;

  \draw[fill=black] (1,0) circle (1pt);
  \draw[fill=black] (0,1) circle (1pt);
  \draw[fill=black] (2.5,2.5) circle (1pt);

 \node[above, font=\small,scale=0.6] at (2.5,2.5) {$\mathbf{v}$};
 \node[below, font=\small,scale=0.6] at (1,0) {$\mathbf{u}$};
 \node[left, font=\small,scale=0.6] at (0,1) {$\mathbf{a}$};

 \node[above, font=\small,scale=0.6] at (1.25,2) {$R_1$};
 \node[below, font=\small,scale=0.6] at (0.25,1.4) {$R_3$};
 \node[below, font=\small,scale=0.6] at (2.25,0.4) {$R_2$};

 \node[below, font=\small,scale=1] at (1.2,-0.25) {(e) Case 5, $F$};


\end{tikzpicture}}
\end{minipage}\hfill
\begin{minipage}{0.22\textwidth}
\centering
\resizebox{\linewidth}{!}{%
\begin{tikzpicture}[every node/.style={font=\small}]
  \draw[fill=black!16,line width=0.8pt]
    (0,0)--(1.5,0)--(1.5,2.5)--(0,2.5)--(0,0)--cycle;
  \draw[fill=black!16,line width=0.8pt]
   (0,0)--(2.5,0)--(2.5,1.5)--(0,1.5)--(0,0)--cycle;
  \draw[line width=0.8pt] (1.5,1.5)--(1.5,0) (1,0)--(1,0.5)--(1.5,0.5)  (2,1.5)--(2,1)--(2.5,1) (0,2)--(0.5,2)--(0.5,2.5);

  \draw[fill=black] (0,0) circle (1pt);
  \draw[fill=black] (1.5,2.5) circle (1pt);
  \draw[fill=black] (2,1) circle (1pt);

 \node[above, font=\small,scale=0.6] at (1.5,2.5) {$\mathbf{v}$};
 \node[below, font=\small,scale=0.6] at (0,0) {$\mathbf{u}$};
 \node[left, font=\small,scale=0.6] at (2,1) {$\mathbf{a}$};

 \node[above, font=\small,scale=0.6] at (0.25,2) {$R_1$};
 \node[below, font=\small,scale=0.6] at (2.25,1.4) {$R_3$};
 \node[below, font=\small,scale=0.6] at (1.25,0.4) {$R_2$};

 \node[below, font=\small,scale=1] at (1.2,-0.25) {(f) Case 6, $F$};


\end{tikzpicture}}
\end{minipage}\hfill
\caption{Configurations of $R_1, R_2, R_3$ in $F$}\label{fig17}
\end{figure}

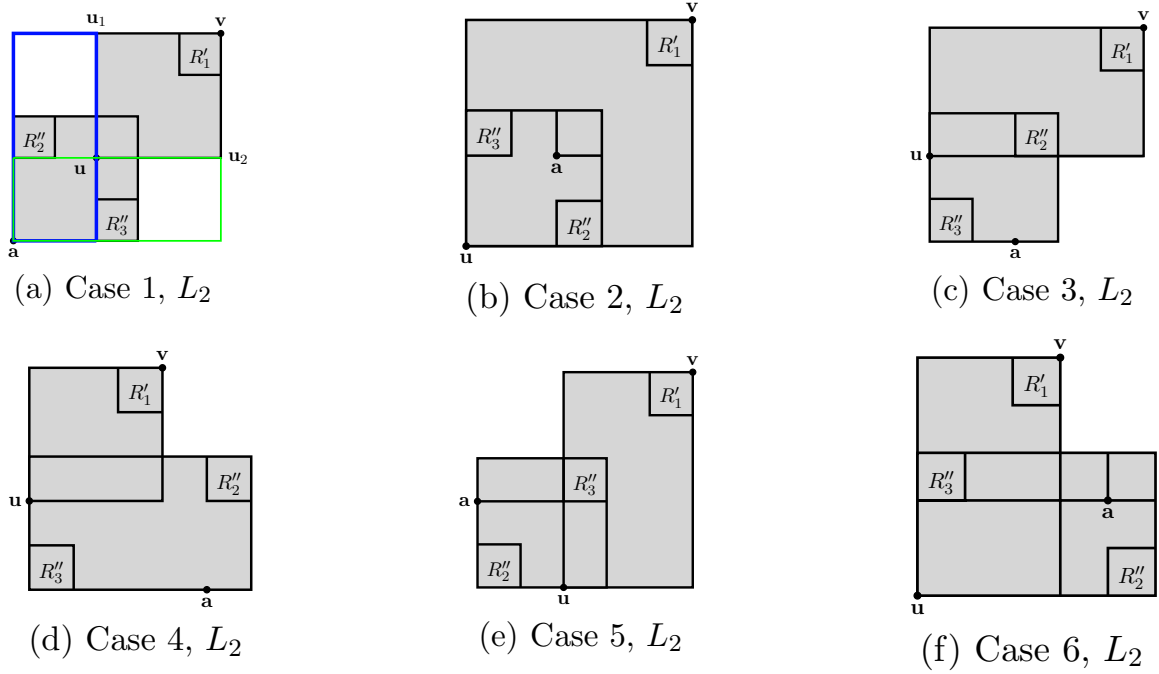
\begin{figure}[htbp]
\centering
\begin{minipage}{0.22\textwidth}
\centering
\resizebox{\linewidth}{!}{%
\begin{tikzpicture}[every node/.style={font=\small}]
  \draw[fill=black!16,line width=0.8pt]
    (0,0)--(1.5,0)--(1.5,1.5)--(0,1.5)--(0,0)--cycle;
  \draw[fill=black!16,line width=0.8pt]
    (1,1)--(2.5,1)--(2.5,2.5)--(1,2.5)--(1,1)--cycle;
  \draw[line width=0.8pt]  (1,1.5)--(1.5,1.5)--(1.5,1) (1,0)--(1,0.5)--(1.5,0.5) (0,1)--(0.5,1)--(0.5,1.5)  (2,2.5)--(2,2)--(2.5,2) ;

  \draw[fill=black] (0,0) circle (1pt);
  \draw[fill=black] (1,1) circle (1pt);
  \draw[fill=black] (2.5,2.5) circle (1pt);

 \textcolor[rgb]{0.00,0.07,1.00}{ \draw[line width=1.2pt] (1,0)--(0,0)--(0,2.5)--(1,2.5)--(1,0); }
 \textcolor[rgb]{0.00,1.00,0.00}{ \draw[line width=0.6pt] (0,1)--(0,0)--(2.5,0)--(2.5,1)--(0,1); }

 \node[above, font=\small,scale=0.6] at (2.5,2.5) {$\mathbf{v}$};
   \node[above, font=\small,scale=0.6] at (1,2.5) {$\mathbf{u}_1$};
    \node[right, font=\small,scale=0.6] at (2.5,1) {$\mathbf{u}_2$};
 \node[below, font=\small,scale=0.6] at (0.8,1) {$\mathbf{u}$};
 \node[below, font=\small,scale=0.6] at (0,0) {$\mathbf{a}$};

 \node[above, font=\small,scale=0.6] at (2.25,2) {$R'_1$};
 \node[below, font=\small,scale=0.6] at (0.25,1.4) {$R''_2$};
 \node[below, font=\small,scale=0.6] at (1.25,0.4) {$R''_3$};

 \node[below, font=\small,scale=1] at (1.2,-0.25) {(a) Case 1, $L_2$};


\end{tikzpicture}}
\end{minipage}\hfill
\begin{minipage}{0.22\textwidth}
\centering
\resizebox{\linewidth}{!}{%
\begin{tikzpicture}[every node/.style={font=\small}]
  \draw[fill=black!16,line width=0.8pt]
    (0,0)--(2.5,0)--(2.5,2.5)--(0,2.5)--(0,0)--cycle;
  \draw[fill=black!16,line width=0.8pt]
   (0,0)--(1.5,0)--(1.5,1.5)--(0,1.5)--(0,0)--cycle;
  \draw[line width=0.8pt] (1,0)--(1,0.5)--(1.5,0.5) (0,1)--(0.5,1)--(0.5,1.5)   (1,1.5)--(1,1)--(1.5,1)  (2,2.5)--(2,2)--(2.5,2) ;

  \draw[fill=black] (0,0) circle (1pt);
  \draw[fill=black] (1,1) circle (1pt);
  \draw[fill=black] (2.5,2.5) circle (1pt);

 \node[above, font=\small,scale=0.6] at (2.5,2.5) {$\mathbf{v}$};
 \node[below, font=\small,scale=0.6] at (1,1) {$\mathbf{a}$};
 \node[below, font=\small,scale=0.6] at (0,0) {$\mathbf{u}$};

 \node[above, font=\small,scale=0.6] at (2.25,2) {$R'_1$};
 \node[below, font=\small,scale=0.6] at (0.25,1.4) {$R''_3$};
 \node[below, font=\small,scale=0.6] at (1.25,0.4) {$R''_2$};

 \node[below, font=\small,scale=1] at (1.2,-0.25) {(b) Case 2, $L_2$};

\end{tikzpicture}}
\end{minipage}\hfill
\begin{minipage}{0.22\textwidth}
\centering
\resizebox{\linewidth}{!}{%
\begin{tikzpicture}[every node/.style={font=\small}]
  \draw[fill=black!16,line width=0.8pt]
    (0,1)--(2.5,1)--(2.5,2.5)--(0,2.5)--(0,1)--cycle;
  \draw[fill=black!16,line width=0.8pt]
   (0,0)--(1.5,0)--(1.5,1.5)--(0,1.5)--(0,0)--cycle;
  \draw[line width=0.8pt] (0,1)--(2.5,1)   (0,0.5)--(0.5,0.5)--(0.5,0)  (1,1.5)--(1,1)--(1.5,1)     (2,2.5)--(2,2)--(2.5,2) ;

  \draw[fill=black] (1,0) circle (1pt);
  \draw[fill=black] (0,1) circle (1pt);
  \draw[fill=black] (2.5,2.5) circle (1pt);

 \node[above, font=\small,scale=0.6] at (2.5,2.5) {$\mathbf{v}$};
 \node[left, font=\small,scale=0.6] at (0,1) {$\mathbf{u}$};
 \node[below, font=\small,scale=0.6] at (1,0) {$\mathbf{a}$};

 \node[above, font=\small,scale=0.6] at (2.25,2) {$R'_1$};
 \node[below, font=\small,scale=0.6] at (1.25,1.4) {$R''_2$};
 \node[below, font=\small,scale=0.6] at (0.25,0.4) {$R''_3$};

 \node[below, font=\small,scale=1] at (1.2,-0.25) {(c) Case 3, $L_2$};


\end{tikzpicture}}
\end{minipage}\hfill

\begin{minipage}{0.22\textwidth}
\centering
\resizebox{\linewidth}{!}{%
\begin{tikzpicture}[every node/.style={font=\small}]
  \draw[fill=black!16,line width=0.8pt]
    (0,1)--(1.5,1)--(1.5,2.5)--(0,2.5)--(0,1)--cycle;
  \draw[fill=black!16,line width=0.8pt]
   (0,0)--(2.5,0)--(2.5,1.5)--(0,1.5)--(0,0)--cycle;
  \draw[line width=0.8pt] (0,1)--(1.5,1)--(1.5,1.5) (0,0.5)--(0.5,0.5)--(0.5,0)  (2,1.5)--(2,1)--(2.5,1)  (1,2.5)--(1,2)--(1.5,2) ;

  \draw[fill=black] (2,0) circle (1pt);
  \draw[fill=black] (0,1) circle (1pt);
  \draw[fill=black] (1.5,2.5) circle (1pt);

 \node[above, font=\small,scale=0.6] at (1.5,2.5) {$\mathbf{v}$};
 \node[left, font=\small,scale=0.6] at (0,1) {$\mathbf{u}$};
 \node[below, font=\small,scale=0.6] at (2,0) {$\mathbf{a}$};

 \node[above, font=\small,scale=0.6] at (1.25,2) {$R'_1$};
 \node[below, font=\small,scale=0.6] at (2.25,1.4) {$R''_2$};
 \node[below, font=\small,scale=0.6] at (0.25,0.4) {$R''_3$};

 \node[below, font=\small,scale=1] at (1.2,-0.25) {(d) Case 4, $L_2$};


\end{tikzpicture}}
\end{minipage}\hfill
\begin{minipage}{0.22\textwidth}
\centering
\resizebox{\linewidth}{!}{%
\begin{tikzpicture}[every node/.style={font=\small}]
  \draw[fill=black!16,line width=0.8pt]
    (0,0)--(1.5,0)--(1.5,1.5)--(0,1.5)--(0,0)--cycle;
  \draw[fill=black!16,line width=0.8pt]
   (1,0)--(2.5,0)--(2.5,2.5)--(1,2.5)--(1,0)--cycle;
  \draw[line width=0.8pt] (0,1)--(1,1) (1.5,0)--(1.5,1.5)--(1,1.5)  (2,2.5)--(2,2)--(2.5,2) ;

  \draw[line width=0.8pt]  (0,0.5)--(0.5,0.5)--(0.5,0)  (1,1.5)--(1,1)--(1.5,1);

  \draw[fill=black] (1,0) circle (1pt);
  \draw[fill=black] (0,1) circle (1pt);
  \draw[fill=black] (2.5,2.5) circle (1pt);

 \node[above, font=\small,scale=0.6] at (2.5,2.5) {$\mathbf{v}$};
 \node[below, font=\small,scale=0.6] at (1,0) {$\mathbf{u}$};
 \node[left, font=\small,scale=0.6] at (0,1) {$\mathbf{a}$};

 \node[above, font=\small,scale=0.6] at (2.25,2) {$R'_1$};
 \node[below, font=\small,scale=0.6] at (1.25,1.4) {$R''_3$};
 \node[below, font=\small,scale=0.6] at (0.25,0.4) {$R''_2$};

 \node[below, font=\small,scale=1] at (1.2,-0.25) {(e) Case 5, $L_2$};


\end{tikzpicture}}
\end{minipage}\hfill
\begin{minipage}{0.22\textwidth}
\centering
\resizebox{\linewidth}{!}{%
\begin{tikzpicture}[every node/.style={font=\small}]
  \draw[fill=black!16,line width=0.8pt]
    (0,0)--(1.5,0)--(1.5,2.5)--(0,2.5)--(0,0)--cycle;
  \draw[fill=black!16,line width=0.8pt]
   (0,0)--(2.5,0)--(2.5,1.5)--(0,1.5)--(0,0)--cycle;
  \draw[line width=0.8pt] (0,1)--(2.5,1) (2,1.5)--(2,1)  (1.5,1.5)--(1.5,0) (2.5,0.5)--(2,0.5)--(2,0)  (0.5,1.5)--(0.5,1)--(0,1) (1,2.5)--(1,2)--(1.5,2);

  \draw[fill=black] (0,0) circle (1pt);
  \draw[fill=black] (1.5,2.5) circle (1pt);
  \draw[fill=black] (2,1) circle (1pt);

 \node[above, font=\small,scale=0.6] at (1.5,2.5) {$\mathbf{v}$};
 \node[below, font=\small,scale=0.6] at (0,0) {$\mathbf{u}$};
 \node[below, font=\small,scale=0.6] at (2,1) {$\mathbf{a}$};

 \node[above, font=\small,scale=0.6] at (1.25,2) {$R'_1$};
 \node[below, font=\small,scale=0.6] at (0.25,1.4) {$R''_3$};
 \node[below, font=\small,scale=0.6] at (2.25,0.4) {$R''_2$};

 \node[below, font=\small,scale=1] at (1.2,-0.25) {(f) Case 6, $L_2$};


\end{tikzpicture}}
\end{minipage}\hfill
\caption{Configurations of $R'_1, R''_2, R''_3$ in $L_2$}\label{fig18}
\end{figure}

{\bf Step 2:} We show that there exists a bijection from $\mathcal {R}'_k(\mathcal {C})/ \overset{\scriptscriptstyle\mathcal {C}}{\thicksim}$ to $\widetilde{\mathcal {R}}_{k-1}(\mathcal {C}_{\mathbf{v}})$. As a consequence, $|\mathcal {R}'_k(\mathcal {C})/ \overset{\scriptscriptstyle\mathcal {C}}{\thicksim}|= |\widetilde{\mathcal {R}}_{k-1}(\mathcal {C}_{\mathbf{v}})|=\widetilde{r}_{k-1}(\mathcal {C}_{\mathbf{v}})$.

It is easy to see that $F\in \mathcal {R}'_k(\mathcal {C})$ if and only if $F$ is a $k$-rook configuration in $\mathcal {R}_k(\mathcal {C})$ with a rook $R_1$ placed on the basic interval $\overline{I_{\mathbf{v}}}(\mathcal {C})$ and  other rooks placed on the  set $\big\{I_{\mathbf{u}}(\mathcal {C})\mid \mathbf{u}\in T(\mathcal {C})\setminus U_{\mathbf{v}}(\mathcal {C}) \big\}$, where $T(\mathcal {C})$ is the set of the lower left corners of basic intervals of $\mathcal {C}$.
In fact, if there exists some $R_2\in F\setminus\{R_1\}$ such that $R_2$ is placed on $I_{\mathbf{a}}(\mathcal {C})$ for some $\mathbf{a}\in U_{\mathbf{v}}(\mathcal {C})$.
Then $R_1$ and $R_2$ are  in switching position, and $F$ is equivalent to a $k$-rook configuration in $\mathcal {C}\setminus \mathbf{v}$, a contradiction.


Let $\mathcal {R}''_{i}(\mathcal {C})$ (resp. $\mathcal {R}^*_{i}(\mathcal {C})$) be the collection of all $i$-rook configurations (resp. $i$-rook placements) in $\mathcal {C}$, each rook of which is placed on $\big\{I_{\mathbf{u}}(\mathcal {C})\mid \mathbf{u}\in T(\mathcal {C})\setminus U_{\mathbf{v}}(\mathcal {C}) \big\}$. Note that $F \in \mathcal {R}'_k(\mathcal {C})$ if and only if $F \setminus \{R_1\} \in \mathcal {R}''_{k-1}(\mathcal {C})$, and $F \overset{\scriptscriptstyle\mathcal {C}}{\thicksim} G$ if and only if $F \setminus \{R_1\} \overset{\scriptscriptstyle\mathcal {C}}{\thicksim} G \setminus \{R_1\}$ for each pair of  $F, G \in \mathcal {R}'_k(\mathcal {C})$, where $R_1$ is a rook placed on $\overline{I_{\mathbf{v}}}(\mathcal {C})$. It follows that there exists a bijection from $\mathcal {R}'_k(\mathcal {C})/ \overset{\scriptscriptstyle\mathcal {C}}{\thicksim}$ to $\mathcal {R}''_{k-1}(\mathcal {C})/ \overset{\scriptscriptstyle\mathcal {C}}{\thicksim}$.
Thus, it suffices to show that there exists a bijection from $\mathcal {R}''_{k-1}(\mathcal {C})/ \overset{\scriptscriptstyle\mathcal {C}}{\thicksim}$ to $\widetilde{\mathcal {R}}_{k-1}(\mathcal {C}_{\mathbf{v}})$.

According to the proof of Proposition \ref{induction-2}, $T(\mathcal {C})\setminus U_{\mathbf{v}}(\mathcal {C})$ is the set of the lower left corners of basic intervals of $\mathcal {C}_{\mathbf{v}}$. Hence there exists a bijection $\psi$: $\big\{I_{\mathbf{u}}(\mathcal {C})\mid \mathbf{u}\in T(\mathcal {C})\setminus U_{\mathbf{v}}(\mathcal {C}) \big\} \to \mathcal {C}_{\mathbf{v}}$ given by $I_{\mathbf{u}}(\mathcal {C}) \mapsto I_{\mathbf{u}}(\mathcal {C}_{\mathbf{v}})$. This could induce a natural bijection $\psi_{i}$:  $\mathcal {R}^*_{i}(\mathcal {C}) \to \mathcal {R}^*_{i}(\mathcal {C}_{\mathbf{v}})$ given by $F \mapsto \psi_{i}(F)$, where $F$ is an $i$-rook placement placed on $I_{\mathbf{u}_1}(\mathcal {C}), \ldots, I_{\mathbf{u}_i}(\mathcal {C})$ and $\psi_{i}(F)$ is an $i$-rook placement placed on $I_{\mathbf{u}_1}(\mathcal {C}_{\mathbf{v}}), \ldots, I_{\mathbf{u}_i}(\mathcal {C}_{\mathbf{v}})$.
It follows from Setup \ref{setup} and Remark \ref{T1} that two rooks $R_1, R_2$ are in a same row (resp. in a same column) of  $\mathcal {C}$ if and only if $\psi_{1}(R_1), \psi_{1}(R_2)$ are in a same row (resp. in a same column) of  $\mathcal {C}_{\mathbf{v}}$. As a consequence, $F$ is a $(k-1)$-rook configuration of $\mathcal {R}''_{k-1}(\mathcal {C})$ if and only if $\psi_{k-1}(F)$ is a $(k-1)$-rook configuration of $\mathcal {R}_{k-1}(\mathcal {C}_{\mathbf{v}})$. Hence $\psi_{k-1}$ induces a bijection from $\mathcal {R}''_{k-1}(\mathcal {C})$ to $\mathcal {R}_{k-1}(\mathcal {C}_{\mathbf{v}})$.
Furthermore, it follows from Setup \ref{setup}, Remark \ref{T1} and Lemma \ref{cv-2} that $[\mathbf{u}, urc(I_{\mathbf{w}}(\mathcal {C}))]$ is an inner interval of $\mathcal {C}$ containing rooks $R_1, R_2$ if and only if $[\mathbf{u}, urc(I_{\mathbf{w}}(\mathcal {C}_{\mathbf{v}}))]$ is an inner interval of $\mathcal {C}_{\mathbf{v}}$ containing rooks $\psi_{1}(R_1), \psi_{1}(R_2)$. As a consequence, for each pair of $F, G \in \mathcal {R}''_{k-1}(\mathcal {C})$, $F \overset{\scriptscriptstyle\mathcal {C}}{\thicksim} G$ if and only if $\psi_{k-1}(F) \overset{\scriptscriptstyle\mathcal {C}_{\mathbf{v}}}{\thicksim} \psi_{k-1}(G)$. Hence $\psi_{k-1}$ induces a bijection from $\mathcal {R}''_{k-1}(\mathcal {C})/ \overset{\scriptscriptstyle\mathcal {C}}{\thicksim}$ to $\widetilde{\mathcal {R}}_{k-1}(\mathcal {C}_{\mathbf{v}})$.
\end{proof}

It follows from \cite[Remark 2.1]{2025nrr} and Lemma \ref{sum} that the following result is true.
\begin{lem}\label{weakly-connected}
Let $\mathcal {C}$ be a polyocollection with weakly connected components $\mathcal {C}_1,\mathcal {C}_2,\ldots,\mathcal {C}_s$. If $H_{\mathbb{K}[\mathcal {C}_i]}(t)=\frac{\widetilde{r}_{\mathcal {C}_i}(t)}{(1-t)^{\text{\em dim}\mathbb{K}[\mathcal {C}_i]}}$ for $i=1,2,\ldots,s$, then $H_{\mathbb{K}[\mathcal {C}]}(t)=\frac{\widetilde{r}_{\mathcal {C}}(t)}{(1-t)^{|V(\mathcal {C})|-|\mathcal {C}|}}$.
\end{lem}

Now, we could give (3) of Main theorem.

\begin{thm} \label{hs-r}
Let $\mathcal {C}$ be a natural polyocollection of type $\mathcal{Q}_1$ or of type $\mathcal{Q}_2$. Then the following statements are true:\\[-0.3cm]
\begin{itemize}
\item[(1)]  The $h$-polynomial of $\mathbb{K}[\mathcal {C}]$ is $\widetilde{r}_{\mathcal {C}}(t)$$;$ \\[-0.3cm]
\item[(2)] $\text{\em reg}(\mathbb{K}[\mathcal {C}])=r(\mathcal {C})$.\\[-0.3cm]
\end{itemize}
\end{thm}

\begin{proof}
We first consider the case that $\mathcal {C}$ is a natural polyocollection of type $\mathcal{Q}_1$.
From Lemma \ref{weakly-connected}, we can assume that $\mathcal {C}$ is weakly connected. Since $\mathbb{K}[\mathcal {C}]$ is Cohen-Macaulay by Corollary \ref{cor1}, then, using Lemma \ref{h-r}, it suffices to show that $H_{\mathbb{K}[\mathcal {C}]}(t)=\frac{\widetilde{r}_{\mathcal {C}}(t)}{(1-t)^{|V(\mathcal {C})|-|\mathcal {C}|}}$.

We use induction on $|\mathcal {C}|$. If $|\mathcal {C}|=1$, then the result is obvious. Now assume that  $|\mathcal {C}|\geq 2$ and the result is true for $|\mathcal {C}|-1$.  From Lemmas \ref{intial} and  \ref{exact}, it follows that
\[H_{\mathbb{K}[\mathcal {C}]}(t)=H_{S_{\mathcal {C}}/\text{in}_{<_{\text{lex}}}(I_{\mathcal {C}})}(t)=t\cdot H_{S_{\mathcal {C}}/(\text{in}_{<_{\text{lex}}}(I_{\mathcal {C}}):x_{\mathbf{v}})}(t)+H_{S_{\mathcal {C}}/(\text{in}_{<_{\text{lex}}}(I_{\mathcal {C}}),x_{\mathbf{v}})}(t),\tag{$\ddag\ddag$} \label{formula4}\]
where $\mathbf{v}\in V(\mathcal {C})$  is the uppermost and rightmost vertex of $\mathcal {C}$ and $<_{\text{lex}}$ is the monomial order on $S_{\mathcal {C}}$ defined in Section \ref{sec1}.

By Proposition \ref{induction-1}, $(\text{in}_{<_{\text{lex}}}(I_{\mathcal {C}}),x_{\mathbf{v}})=\text{in}_{<_{\text{lex}}}(I_{\mathcal {C}\setminus \mathbf{v}})+(x_{\mathbf{v}})$ and $\mathcal {C}\setminus \mathbf{v}$ is a natural polyocollection of type $\mathcal{Q}_1$.  Note that  $\mathcal {C}\setminus \mathbf{v}$ can be obtained from $\mathcal {C}$ by removing the interval $\overline{I_{\mathbf{v}}}(\mathcal {C})$, then $|\mathcal {C}|=|\mathcal {C}\setminus \mathbf{v}|+1$.
 Using Lemmas \ref{intial}, \ref{sum} and induction hypothesis on $|\mathcal {C}|$, we can get that
\begin{eqnarray*}
H_{S_{\mathcal {P}}/(\text{in}_{<_{\text{lex}}}(I_{\mathcal {C}}),x_{\mathbf{v}})}(t)&=&H_{S_{1}/(x_{\mathbf{v}})}(t)\cdot H_{S_{\mathcal {C}\setminus \mathbf{v}}/\text{in}_{<_{\text{lex}}}(I_{\mathcal {C}\setminus \mathbf{v}})}(t)\\
&=&\frac{1}{(1-t)^{|V(\mathcal {C})|-|V(\mathcal {C}\setminus \mathbf{v})|-1}}\cdot H_{\mathbb{K}[\mathcal {C}\setminus \mathbf{v}]}(t)\\
&=&\frac{1}{(1-t)^{|V(\mathcal {C})|-|V(\mathcal {C}\setminus \mathbf{v})|-1}}\cdot\frac{\widetilde{r}_{\mathcal {C}\setminus \mathbf{v}}(t)}{(1-t)^{|V(\mathcal {C}\setminus \mathbf{v})|-|\mathcal {C}\setminus \mathbf{v}|}}\\
&=&\frac{\widetilde{r}_{\mathcal {C}\setminus \mathbf{v}}(t)}{(1-t)^{|V(\mathcal {C})|-|\mathcal {C}|}},
\end{eqnarray*}
where $S_1=\mathbb{K}[x_{\mathbf{a}}\mid \mathbf{a}\in V(\mathcal {C})\setminus V(\mathcal {C}\setminus \mathbf{v})]$.

By Lemma \ref{cv-1} and Proposition \ref{induction-1}, $$(\text{in}_{<_{\text{lex}}}(I_{\mathcal {C}}):x_{\mathbf{v}})=\text{in}_{<_{\text{lex}}}(I_{\mathcal {C}_{\mathbf{v}}})+\big(x_{\mathbf{u}}\mid \mathbf{u}\in U_{\mathbf{v}} (\mathcal {C})\big)$$
 and $\mathcal {C}_{\mathbf{v}}$ is polyocollection of type $\mathcal{Q}_1$, where $U_{\mathbf{v}} (\mathcal {C})$ and $\mathcal {C}_{\mathbf{v}}$ are defined as Setup \ref{setup}.  Using Lemmas \ref{intial} and \ref{sum}, we can get that
\begin{eqnarray*}
\text{($\ddag\ddag\ddag$)} \hspace{1cm} H_{S_{\mathcal {C}}/(\text{in}_{<_{\text{lex}}}(I_{\mathcal {C}}):x_{\mathbf{v}})}(t)&=&H_{S_{2}/(x_{\mathbf{u}}\mid \mathbf{u}\in U_{\mathbf{v}} (\mathcal {C}) )}(t)\cdot H_{S_{\mathcal {C}_{\mathbf{v}}}/\text{in}_{<_{\text{lex}}}(I_{\mathcal {C}_{\mathbf{v}}})}(t)\\
&=&\frac{1}{(1-t)^{|V(\mathcal {C})|-|V(\mathcal {C}_{\mathbf{v}})|-|U_{\mathbf{v}} (\mathcal {C})|}}\cdot H_{\mathbb{K}[\mathcal {C}_{\mathbf{v}}]}(t),
\end{eqnarray*}
where $S_2=\mathbb{K}[x_{\mathbf{u}}\mid \mathbf{u}\in V(\mathcal {C})\setminus V(\mathcal {C}_{\mathbf{v}})]$. Recall that $\mathcal {C}_{\mathbf{v}}$ is  algebraically isomorphic to $\mathcal {C}'_{\mathbf{v}}$ defined as Setup \ref{setup} and $\mathcal {C}'_{\mathbf{v}}$ can be regarded as a natural polyocollection of type $\mathcal{Q}_1$ by Lemma \ref{cv-2}, then by hypothesis on $|\mathcal {C}|$,
$$H_{\mathbb{K}[\mathcal {C}_{\mathbf{v}}]}(t)= H_{\mathbb{K}[\mathcal {C}'_{\mathbf{v}}]}(t)
=\frac{\widetilde{r}_{\mathcal {C}'_{\mathbf{v}}}(t)}{(1-t)^{|V(\mathcal {C}'_{\mathbf{v}})|-|\mathcal {C}'_{\mathbf{v}}|}}
=\frac{\widetilde{r}_{\mathcal {C}_{\mathbf{v}}}(t)}{(1-t)^{|V(\mathcal {C}_{\mathbf{v}})|-|\mathcal {C}_{\mathbf{v}}|}}.$$
Since $|\mathcal {C}|=|\mathcal {C}_{\mathbf{v}}|+|U_{\mathbf{v}} (\mathcal {C})|$ by Proposition \ref{induction-2}, then from formula ($\ddag\ddag\ddag$) it follows that
$$H_{S_{\mathcal {C}}/(\text{in}_{<_{\text{lex}}}(I_{\mathcal {C}}):x_{\mathbf{v}})}(t)=\frac{\widetilde{r}_{\mathcal {C}_{\mathbf{v}}}(t)}{(1-t)^{|V(\mathcal {C})|-|\mathcal {C}|}}.$$
Furthermore, from formula (\ref{formula4}) and Proposition \ref{induction-3}  it follows that
$$H_{\mathbb{K}[\mathcal {C}]}(t)=\frac{t\cdot\widetilde{r}_{\mathcal {C}_{\mathbf{v}}}(t)+\widetilde{r}_{\mathcal {C}\setminus \mathbf{v}}(t)}{(1-t)^{|V(\mathcal {C})|-|\mathcal {C}|}}=\frac{\widetilde{r}_{\mathcal {C}}(t)}{(1-t)^{|V(\mathcal {C})|-|\mathcal {C}|}},$$
as desired.

If $\mathcal {C}$ is is a natural polyocollection of type $\mathcal{Q}_2$, then its horizontal flip is a natural polyocollection of type $\mathcal{Q}_1$, which verifies the conclusion proved above.
\end{proof}

Theorem \ref{hs-r} gives positive answers for the switching rook polynomial conjecture and the rook number problem for collections of cells of type $\mathcal{Q}_1$ or of type $\mathcal{Q}_2$.
Since some well-studied collections of cells are natural polyocollections of type $\mathcal{Q}_1$ or of type $\mathcal{Q}_2$,
Theorem \ref{hs-r} directly yields the following corollary, which recovers their consistency with Conjecture \ref{conj2} in a unified way.
\begin{cor}
Let $\mathcal{P}$ be one of the following collections of cells. Then the $h$-polynomial of $\mathbb{K}[\mathcal{P}]$ is equal to
$\widetilde{r}_{\mathcal{P}}(t)$, and its regularity is equal to  $r(\mathcal{P})$. \\[-0.4cm]
\begin{itemize}
  \item[(1)] Grid polyominoes (\cite{2023dn})$;$ \\[-0.4cm]
  \item[(2)] Frame polyominoes (\cite{2024jn})$;$ \\[-0.4cm]
  \item[(3)] Collection of cells of type $\mathcal{Q}_1$ or of type $\mathcal{Q}_2$,  each of whose weakly connected components is convex (\cite{2025nrr})$;$ \\[-0.4cm]
\end{itemize}
\end{cor}

The key to the proof of Theorem \ref{hs-r} is constructing a suitable short exact sequence, which reduces the computation of the Hilbert series of the coordinate ring of a natural polyocollection of type $\mathcal{Q}_1$ to that of the coordinate rings of two new natural polyocollections of type $\mathcal{Q}_1$ of smaller rank. However, restricted to collections of cells, the method fails, as the following example shows. This failure is one key reason we focus on polyocollections.

\begin{ex}\label{exm1}
Let $\mathcal {C}$ be a collection of cells of type $\mathcal{Q}_1$, as shown in Figure \ref{fig12} (a),  and let  $\mathbf{v}\in V(\mathcal {C})$ be its uppermost and rightmost vertex. From Proposition \ref{induction-1}  it follows that
$$(\text{in}_{<_{\text{lex}}}(I_{\mathcal {C}}):x_{\mathbf{v}})=\text{in}_{<_{\text{lex}}}(I_{\mathcal {C}_{\mathbf{v}}})+\big(x_{\mathbf{u}_1}, x_{\mathbf{u}_2}\big)$$
where $\mathcal {C}_{\mathbf{v}}$ is shown in Figure \ref{fig12} (b). $\mathcal {C}_{\mathbf{v}}$ is clearly  a natural polocollection, but it is difficult to determine whether it is algebraically isomorphic to a collection of cells.
\end{ex}

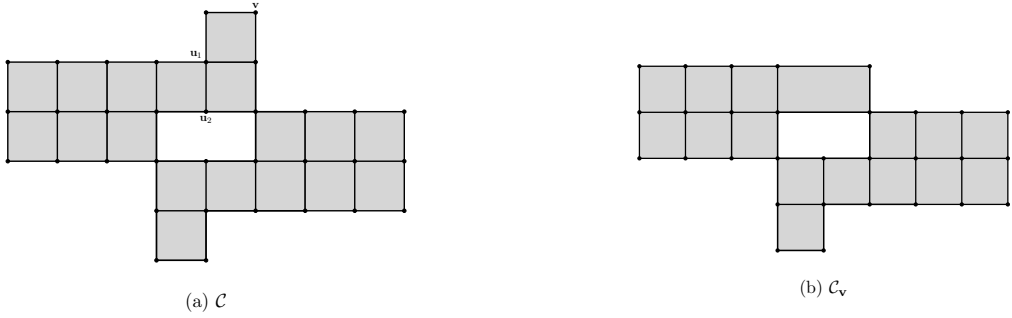
\begin{figure}[htbp]
\centering
\begin{minipage}{0.40\textwidth}
\centering
\resizebox{\linewidth}{!}{%
\begin{tikzpicture}[every node/.style={font=\small}]
\draw[fill=black!16,line width=0.8pt]
    (1,0)--(2,0)--(2,1)--(4,1)--(4,3)--(3,3)--(3,2)--(1,2)--(1,0)--cycle;
\draw[fill=black!16,line width=0.8pt]
    (1,2)--(1,3)--(3,3)--(3,4)--(0,4)--(0,2)--(1,2)--cycle;
\draw[fill=black!16,line width=0.8pt]
    (2,4)--(3,4)--(3,5)--(2,5)--(2,4)--cycle;

\draw[fill=black!16,line width=0.8pt]
    (-2,4)--(-2,2)--(0,2)--(0,4)--(-2,4)--cycle;

\draw[fill=black!16,line width=0.8pt]
    (4,1)--(6,1)--(6,3)--(4,3)--(4,1)--cycle;

  \draw[line width=0.8pt]   (0,3)--(4,3) (1,2)--(4,2) (1,1)--(4,1) (1,0)--(2,0) (-2,3)--(0,3) (4,2)--(6,2);

    \draw[line width=0.8pt] (1,0)--(1,4) (2,0)--(2,2) (2,3)--(2,4) (3,1)--(3,4) (-1,4)--(-1,2) (5,1)--(5,3);

  \draw[fill=black] (1,0) circle (1pt);
  \draw[fill=black] (2,0) circle (1pt);

  \draw[fill=black] (1,1) circle (1pt);
  \draw[fill=black] (2,1) circle (1pt);
  \draw[fill=black] (3,1) circle (1pt);
  \draw[fill=black] (4,1) circle (1pt);

  \draw[fill=black] (0,2) circle (1pt);
  \draw[fill=black] (1,2) circle (1pt);
  \draw[fill=black] (2,2) circle (1pt);
  \draw[fill=black] (3,2) circle (1pt);
  \draw[fill=black] (4,2) circle (1pt);

  \draw[fill=black] (0,3) circle (1pt);
  \draw[fill=black] (1,3) circle (1pt);
  \draw[fill=black] (2,3) circle (1pt);
  \draw[fill=black] (3,3) circle (1pt);
  \draw[fill=black] (4,3) circle (1pt);

  \draw[fill=black] (0,4) circle (1pt);
  \draw[fill=black] (1,4) circle (1pt);
  \draw[fill=black] (2,4) circle (1pt);
  \draw[fill=black] (3,4) circle (1pt);

  \draw[fill=black] (2,5) circle (1pt);
  \draw[fill=black] (3,5) circle (1pt);

   \draw[fill=black] (-2,4) circle (1pt);
  \draw[fill=black] (-1,4) circle (1pt);
    \draw[fill=black] (-2,3) circle (1pt);
  \draw[fill=black] (-1,3) circle (1pt);
     \draw[fill=black] (-2,2) circle (1pt);
  \draw[fill=black] (-1,2) circle (1pt);

     \draw[fill=black] (5,3) circle (1pt);
  \draw[fill=black] (6,3) circle (1pt);
     \draw[fill=black] (5,2) circle (1pt);
  \draw[fill=black] (6,2) circle (1pt);
     \draw[fill=black] (5,1) circle (1pt);
  \draw[fill=black] (6,1) circle (1pt);

 \node[above, font=\small,scale=0.6] at (3,5) {$\mathbf{v}$};
 \node[above, font=\small,scale=0.6] at (1.8,4) {$\mathbf{u}_1$};
 \node[below, font=\small,scale=0.6] at (2,3) {$\mathbf{u}_2$};

\textcolor[rgb]{1.00,1.00,1.00}{   \node[left, font=\small,scale=0.6] at (-3,0) {$1$};}

 \node[below, font=\small,scale=1] at (2,-0.5) {(a) ~~~~~~~~~ $\mathcal {C}$};

\end{tikzpicture}}
\end{minipage}\hfill
\begin{minipage}{0.40\textwidth}
\centering
\resizebox{\linewidth}{!}{%
\begin{tikzpicture}[every node/.style={font=\small}]
\draw[fill=black!16,line width=0.8pt]
    (1,0)--(2,0)--(2,1)--(4,1)--(4,3)--(3,3)--(3,2)--(1,2)--(1,0)--cycle;
\draw[fill=black!16,line width=0.8pt]
    (1,2)--(1,3)--(3,3)--(3,4)--(0,4)--(0,2)--(1,2)--cycle;

\draw[fill=black!16,line width=0.8pt]
    (-2,4)--(-2,2)--(0,2)--(0,4)--(-2,4)--cycle;

\draw[fill=black!16,line width=0.8pt]
    (4,1)--(6,1)--(6,3)--(4,3)--(4,1)--cycle;

  \draw[line width=0.8pt]   (0,3)--(4,3) (1,2)--(4,2) (1,1)--(4,1) (1,0)--(2,0) (-2,3)--(0,3) (4,2)--(6,2);

    \draw[line width=0.8pt] (1,0)--(1,4) (2,0)--(2,2)   (3,1)--(3,4) (-1,4)--(-1,2) (5,1)--(5,3);

  \draw[fill=black] (1,0) circle (1pt);
  \draw[fill=black] (2,0) circle (1pt);

  \draw[fill=black] (1,1) circle (1pt);
  \draw[fill=black] (2,1) circle (1pt);
  \draw[fill=black] (3,1) circle (1pt);
  \draw[fill=black] (4,1) circle (1pt);

  \draw[fill=black] (0,2) circle (1pt);
  \draw[fill=black] (1,2) circle (1pt);
  \draw[fill=black] (2,2) circle (1pt);
  \draw[fill=black] (3,2) circle (1pt);
  \draw[fill=black] (4,2) circle (1pt);

  \draw[fill=black] (0,3) circle (1pt);
  \draw[fill=black] (1,3) circle (1pt);
  \draw[fill=black] (3,3) circle (1pt);
  \draw[fill=black] (4,3) circle (1pt);

  \draw[fill=black] (0,4) circle (1pt);
  \draw[fill=black] (1,4) circle (1pt);
  \draw[fill=black] (3,4) circle (1pt);

     \draw[fill=black] (-2,4) circle (1pt);
  \draw[fill=black] (-1,4) circle (1pt);
    \draw[fill=black] (-2,3) circle (1pt);
  \draw[fill=black] (-1,3) circle (1pt);
     \draw[fill=black] (-2,2) circle (1pt);
  \draw[fill=black] (-1,2) circle (1pt);

     \draw[fill=black] (5,3) circle (1pt);
  \draw[fill=black] (6,3) circle (1pt);
     \draw[fill=black] (5,2) circle (1pt);
  \draw[fill=black] (6,2) circle (1pt);
     \draw[fill=black] (5,1) circle (1pt);
  \draw[fill=black] (6,1) circle (1pt);

 \textcolor[rgb]{1.00,1.00,1.00}{\node[left, font=\small,scale=0.6] at (2,5) {$1$};}

\textcolor[rgb]{1.00,1.00,1.00}{   \node[left, font=\small,scale=0.6] at (8,0) {$1$};}

 \node[below, font=\small,scale=1] at (2,-0.5) {(b)~~~~~~~~~ $\mathcal {C}_{\mathbf{v}}$};
\end{tikzpicture}}
\end{minipage}\hfill
\caption{An illustration of Example \ref{exm1}}\label{fig12}
\end{figure}

The above example raises a natural question for further investigation:  under what conditions is a polyocollection algebraically isomorphic to a collection of cells?


\end{document}